\documentclass[11pt]{article}
\usepackage[utf8]{inputenc}

\usepackage{epsfig,epsf,fancybox}
\usepackage{amsmath}
\usepackage{mathrsfs,bbm}
\usepackage{amssymb}
\usepackage{graphicx}
\usepackage{color}
\usepackage{multirow}
\usepackage{paralist}
\usepackage{verbatim}
\usepackage{galois}
\usepackage{algorithm}
\usepackage{algorithmic}
\usepackage{boxedminipage}
\usepackage{booktabs}
\usepackage{accents}
\usepackage{stmaryrd}
\usepackage{subfig}
\usepackage{appendix}
\usepackage{graphicx}
\usepackage{epstopdf}
\usepackage{appendix}
\usepackage{amsthm}
\usepackage{cases}  
\usepackage[top=1in,bottom=1.2in,left=1in,right=1in,xetex]{geometry}
\usepackage{extarrows}

\usepackage{titlesec}
\usepackage{titletoc}

\usepackage[unicode=true,pdfusetitle,
bookmarks=true,bookmarksnumbered=true,bookmarksopen=true,bookmarksopenlevel=3,
breaklinks=false,pdfborder={0 0 1},backref=false,colorlinks=false]
{hyperref}

\numberwithin{equation}{section}

\newtheorem{theorem}{Theorem}[section]

\newtheorem{lemma}[theorem]{Lemma}
\newtheorem{proposition}[theorem]{Proposition}

\newtheorem{definition}{Definition}[section]

\newtheorem{remark}{Remark}[section]

\newcommand{\f}{\mathscr{F}}
\newcommand{\lr}{\mathcal{L}}
\newcommand{\sr}{\mathcal{S}}
\newcommand{\e}{\mathbb{E}}
\newcommand{\br}{\mathbb{R}}
\newcommand{\pr}{\mathcal{P}}
\newcommand{\dd}{\partial}

\newcommand{\brn}{{\mathbb{R}^n}}

\newcommand{\brd}{{\mathbb{R}^d}}
\newcommand{\de}{\Delta}

\newcommand{\hv}{\widehat{v}}

\newcommand{\tx}{\widetilde{X}}

\newcommand{\tz}{\widetilde{Z}}

\newcommand{\dr}{\mathcal{D}}
\newcommand{\bd}{\mathbb{D}}

\newcommand{\argmin}{\mathop{\rm argmin}}

\allowdisplaybreaks[2]

\title{Second Order Fully Nonlinear Mean Field Games with  Degenerate Diffusions}

\usepackage{authblk}

\author[a]{Alain Bensoussan\footnote{E-mail: axb046100@utdallas.edu}}
\author[b]{Ziyu Huang\footnote{E-mail: zyhuang19@fudan.edu.cn}}
\author[c]{Shanjian Tang\footnote{E-mail: sjtang@fudan.edu.cn}}
\author[b]{Sheung Chi Phillip Yam\footnote{E-mail: scpyam@sta.cuhk.edu.hk}}

\affil[a]{\small \it International Center for Decision and Risk Analysis, Naveen Jindal School of Management, University of Texas at Dallas, Dallas, Texas, USA}
\affil[b]{\small \it Department of Statistics, The Chinese University of Hong Kong, Shatin, N.T., Hong Kong SAR}
\affil[c]{\small \it Department of Finance and Control Sciences, School of Mathematical
	Sciences,  Fudan University,  and Key Laboratory of Mathematics for Nonlinear Sciences (Fudan University), Ministry of Education, Shanghai 200433, China}

\begin{document}

\maketitle

\begin{abstract}
In this article, we study the global-in-time well-posedness of second order mean field games (MFGs) with both nonlinear drift functions simultaneously depending on the state, distribution and control variables, and the diffusion term depending on both  state and distribution. Besides, the diffusion term is allowed to be degenerate, unbounded and even nonlinear in the distribution, but it does not depend on the control. First,  we establish the global well-posedness of the corresponding forward-backward stochastic differential equations (FBSDEs), which  arise from the maximum principle under a so-called  $\beta$-monotonicity commonly used in the optimal control theory.  The $\beta$-monotonicity admits more interesting cases, as representative examples including but not limited to the displacement monotonicity, the small mean field effect condition or the Lasry-Lions monotonicity; and  ensures the well-posedness result in diverse non-convex examples.  In our settings, we pose assumptions directly on the drift and diffusion coefficients and the cost functionals, rather than  indirectly on the Hamiltonian, to make the conditions more visible. Our probabilistic method  tackles the nonlinear dynamics with a  linear but  infinite dimensional version, and together with our recently proposed cone property for the adjoint processes, following in an almost straightforward way the conventional approach to the classical stochastic control problem, we derive a sufficiently good regularity of the value functional, and finally show that it is the unique classical solution to the  MFG master equation. Our results require fairly few conditions on the functional coefficients for solution of the MFG, and a bit more conditions---which are least  stringent in the contemporary literature--- for  classical solution of the MFG master equation.\\ 



\end{abstract}


\newpage

\section{Introduction}\label{sec:intro}

Mean field games (MFGs), which have received its great popularity in recent years, were first introduced by Lasry and Lions in a series of papers \cite{CP,JM1,JM2,JM3} and also independently by Huang, Caines and Malhamé \cite{HM2,HM1,HM}. In such a problem, the controlled individual state is affected not only by its state and the imposed control, but also by the equilibrium probability distribution of the state of the overall population; since then there is a huge amount of theoretical research and application explorations. For PDE approaches to forward-backward system (or HJB-FP system) for MFGs, we refer to Bensoussan--Frehse--Yam \cite{AB_book}, Bertucci \cite{Bertucci}, Bertucci-Lasry-Lions \cite{Bertucci'}, Gomes--Pimentel--Voskanyan \cite{GDA}, Graber--M\'{e}sz\'{a}ros \cite{GM}, Huang--Tang \cite{HZ} and Porretta \cite{PA}. For the master equation analytical methods to MFGs, we refer to Bertucci-Lasry-Lions \cite{Bertucci}, Cardaliaguet--Cirant--Porretta \cite{CP1}, Cardaliaguet--Delarue--Lasry--Lions \cite{CDLL} and Gangbo--M\'esz\'aros--Mou--Zhang \cite{GW}. For probabilistic approaches to MFGs and master equations, we refer to Ahuja--Ren--Yang \cite{SA1}, Bensoussan--Huang--Tang--Yam \cite{AB11}, Bensoussan--Tai--Wong--Yam \cite{AB9'1,AB9'2}, Bensoussan--Wong--Yam--Yuan \cite{AB10'}, Buckdahn--Li--Peng-- Rainer \cite{BR}, Carmona--Delarue \cite{book_mfg}, Chassagneux--Crisan--Delarue \cite{CJF} and Huang--Tang \cite{HZ1}; while with the probabilistic approach to mean field type control problems, we refer to Bensoussan--Huang--Yam \cite{AB6,AB8}, Bensoussan--Huang--Tang--Yam \cite{AB10}, Bensoussan--Tai--Yam \cite{AB5}, Bensoussan--Yam \cite{AB}, Carmona--Delarue \cite{CR} and  Ricciardi \cite{RM}.

In this article, we study the well-posedness of generic second order MFGs with nonlinear drifts, and degenerate and unbounded state-distribution dependent diffusions; and here we only demand on very few conditions on the coefficient functions for solution of the MFG, which is even more flexible than the contemporary literature. To do so, we develop a novel and original method which resembles very much the conventional approach to the classical stochastic control problem; moreover, our methodology greatly simplifies the original MFG problem and can include more different settings; even for classical solution of the MFG master equation, our conditions are essentially least stringent in the literature.

Let $(\Omega,\f,\{\f_t,0\le t\le T\},\mathbb{P})$ be a completed filtered probability space (with the filtration being  augmented by all the $\mathbb{P}$-null sets)  in which an $n$-dimensional Brownian motion $\{B(t),\ 0\le t\le T\}$ is defined, and it is $\f_t$-adapted. We denote by $\mathcal{P}_{2}(\brn)$ the space of all probability measures of finite second order moments on $\brn$, equipped with the 2-Wasserstein metric $W_2$. Given the functional coefficients
\begin{align*}
	&b:[0,T]\times\brn\times\mathcal{P}_{2}(\brn)\times \brd\to \brn,\quad \sigma:[0,T]\times \brn \times\pr_2(\brn) \to \br^{n\times n},\\
	&f:[0,T]\times\brn\times\mathcal{P}_{2}(\brn)\times \brd\to \br,\quad g:\brn\times\mathcal{P}_{2}(\brn)\to \br,
\end{align*}
with an initial condition $(t,\mu)\in[0,T]\times\pr_2(\brn)$ and  a random vector $\xi\in L_{\f_t}^2$ independent of the Brownian motion such that $\lr(\xi)=\mu$, we now consider the following MFG problem: 
\begin{equation}\label{intro_1}
	\left\{
	\begin{aligned}
		&v_{t\xi}(\cdot)\in\argmin_{v(\cdot)\in\mathcal{L}_{\mathscr{F}}^2(t,T)}J_{t\xi}\left(v(\cdot);m(s),t\le s\le T\right),\\
		&J_{t\xi}\left(v(\cdot);m(s),0\le s\le T\right):=\e\left[\int_t^T f\left(s,X^v(s),m(s),v(s)\right)dt+g\left(X^v(T),m(T)\right)\right],\\
		&X^v(s)=\xi+\int_t^sb\left(r,X^v(r),m(r),v(r)\right)dr+\int_t^s\sigma\left(r,X^v(r) ,m(r)\right)dB(r), \quad s\in [t,T];\\[4mm]
		&m(s):=\mathcal{L}(X^{v_{t\xi}}(s)),\quad  s\in[t,T].
	\end{aligned}
	\right.
\end{equation}
MFG \eqref{intro_1} is expected to have a unique solution $v_{t\xi} \in \mathcal{L}_{\mathscr{F}}^2(t,T)$ under suitable assumptions on coefficient functions $(b,\sigma,f,g)$ (see Section~\ref{sec:MP}), and we denote by  $X_{t\xi}\in\sr_\f^2(t,T)$ the corresponding equilibrium state process, and by $\left\{m^{t,\mu}(s):=\lr(X_{t\xi}(s)),\ t\le s\le T\right\}$ the distribution flow of the equilibrium state process. Given the distribution flow $m^{t,\mu}$ and an initial state $x\in\brn$, we then consider the following stochastic optimal control problem: 
\begin{equation}\label{intro_1'}
	\left\{
	\begin{aligned}
		&v_{tx\mu}(\cdot)\in\argmin_{v(\cdot)\in\mathcal{L}_{\mathscr{F}}^2(t,T)}J_{tx}\left(v(\cdot);m^{t,\mu}\right)\\
		&\ \ \qquad :=\argmin_{v(\cdot)\in\mathcal{L}_{\mathscr{F}}^2(t,T)}\e\left[\int_t^T f\left(s,X^v(s),m^{t,\mu}(s)),v(s)\right)ds+g\left(X^v(T),m^{t,\mu}(T)\right)\right],\\
		&X^v(s)=x+\int_t^sb\left(r,X^v(r),m^{t,\mu}(r),v(r)\right)dr+\int_t^s\sigma\left(r,X^v(r) ,m^{t,\mu}(r)\right)dB(r).
	\end{aligned}
	\right.
\end{equation}
It is a standard stochastic control problem rather than a McKean-Vlasov one, as the apparently exogenous term $m^{t,\mu}$ is the distribution of the equilibrium state of the MFG \eqref{intro_1} instead of  that of the underlying controlled state. It  is expected to have a unique optimal control, and since it depends on $\xi$ only through its law $\mu$, it is reasonable to denote the optimal control by $v_{tx\mu}\in\mathcal{L}_{\mathscr{F}}^2(t,T)$. We shall use the conventional stochastic control approach to study the regularity of the value functional
\begin{equation}\label{intro_4}
	\begin{split}
		V(t,x,\mu):&=\inf_{v(\cdot)\in\lr_{\f}^2(t,T)}J_{tx}\left(v(\cdot);m^{t,\mu}\right)\\
		&=J_{tx}\left(v_{tx\mu}(\cdot);m^{t,\mu}\right),\quad (t,x,\mu)\in[0,T]\times\brn\times\pr_2(\brn),
	\end{split}	
\end{equation}
where $\mu$ is the initial condition of MFG \eqref{intro_1} and it serves as an augmented infinite-dimensional state of  Problem \eqref{intro_1'}, and $x$ is the spatial  initial condition for Problem \eqref{intro_1'}. One of our main results (Theorem~\ref{thm:1}) asserts that, the functional ${V}$ is the unique classical solution of the following MFG master equation
\begin{equation}\label{intro_5}
	\left\{
	\begin{aligned}
		&\dd_t V(t,x,\mu)+\frac{1}{2}\text{Tr}\left[\left(\sigma\sigma^\top\right) (t,x,\mu)D_x^2 V(t,x,\mu)\right]+H\left(t,x,\mu,D_x V(t,x,\mu)\right)\\
		&+\int_\brn \left\{ \left(D_p H \left(t,y,\mu,D_x V(t,y,\mu)\right)\right)^\top D_y\frac{dV}{d\nu}(t,x,\mu)({y})\right.\\
		&\quad\qquad\left. +\frac{1}{2}\text{Tr}\left[\left(\sigma\sigma^\top\right) (t,y,\mu)D_y^2\frac{dV}{d\nu}(t,x,\mu)({y})\right]\right\}d\mu(y)=0,\quad t\in[0,T);\\[3mm]
		&V(T,x,\mu)=g(x,\mu),\quad (x,\mu)\in \brn\times\pr_2(\brn),
	\end{aligned}	
	\right.
\end{equation}
where the Hamiltonian $H:[0,T]\times\brn\times\pr_2(\brn)\times\brn\to\br$ is defined as
\begin{align}
		&H (s,x,m,p):=\inf_{v\in\brd} L\left(s,x,m,v,p\right),\label{bh'}\\
		\text{and}\qquad &L(s,x,m,v,p):=p^\top  b(s,x,m,v)+f(s,x,m,v), \label{H'}
\end{align}
and  $\frac{dV}{d\nu}$ is the linear functional derivative (see \cite[Chapter 5]{book_mfg}) of $V$ in the measure argument (see Section~\ref{sec:notation} below). MFG \eqref{intro_1} and Problem \eqref{intro_1'} give a representation formula \eqref{intro_4} of the solution of the MFG master equation \eqref{intro_5}; for similar representation formula, we also refer to Buckdahn et al. \cite{BR} and Bertucci et al. \cite[(3.32) and (3.33)]{Bertucci}.
We emphasize that we only require very few conditions for the equilibrium solution of the MFG \eqref{intro_1} and Problem \eqref{intro_1'}, and require a bit more conditions  (but being essentially least stringent than those in the contemporary literature) for classical solution of the MFG master equation \eqref{intro_5}. 
From the viewpoint of control theory, in our article, we solve MFG \eqref{intro_1} and Problem \eqref{intro_1'} via the global solution of the  systems of forward-backward stochastic differential equations (FBSDEs) derived from the maximum principles for their optimal controls; then, we study the Jacobian and Hessian flows of respective systems of FBSDEs associated with MFG \eqref{intro_1} and Problem \eqref{intro_1'} to calculate the derivatives of the value functional $V$ (see \eqref{intro_4}) with respect to both spatial and measure argument. As a natural setting of MFG problems commonly found in applications, we impose assumptions directly on the coefficients of the state system and the cost functional, rather than indirectly on the Hamiltonian functional. Our conditions are thus more visible in a specific MFG problem, and admit more interesting examples such as \cite[Section 9]{AB10'},  which are difficult to be tackled under the Hamiltonian settings.

The  unique existence of global solutions of all the aforementioned FBSDEs, including those for the Jacobian and Hessian flows,  relies on the well-posedness of a general system of  FBSDEs satisfying  the $\beta$-monotonicity---which is commonly used in stochastic control theory, and is discussed in detail in our previous work \cite{AB10,AB11,AB12}. The $\beta$-monotonicity was inspired by the method of continuation in the coefficients of Hu-Peng \cite{YH2} and Peng-Wu \cite{SP}, and addresses a more general situation to incorporate more interesting cases as representative examples including but not limited to  the displacement monotonicity \cite{SA1,CR,HZ1}, the small mean field effect condition \cite{AB11,AB9'1} or the Lasry-Lions monotonicity \cite{book_mfg} for MFGs; indeed, the last three monotonicity conditions have their own interests, and one might overlap with another but without any simple inclusion. We refer to \cite{AB12} for more details  on diverse monotonicity assumptions for MFGs and their relations to $\beta$-monotonicity; there also available are  the definition of the $\beta$-monotonicity and a detailed construction  of an appropriate map $\beta$ for a specific mean field problem, together with a more subtle discussion.  In contrast to the approach via partitioning the time interval, the method of continuation in the coefficients under the $\beta$-monotonicity involves partitioning of the spatial domain. We partition the time interval  in \cite{AB10'} of the first order MFGs with nonlinear drift functions and in \cite{AB9'1} for the second order MFGs with linear drifts and constant diffusion terms; while we use the method of continuation in the coefficients under the $\beta$-monotonicity in \cite{AB11} for MFG with a linear drift and a state-distribution-control dependent diffusion. In particular, the advantages of either temporal or spatial based approach depend on the problems and the respective settings. For the second order nonlinear MFGs considered in this article, the latter seems to be more effective. The conditions in Section~\ref{sec:MP} are quite naturally seen from a  view of stochastic control, while they are not that obvious from an analytical view; moreover, we allow the diffusion coefficients to be degenerate, unbounded and nonlinear in the distribution variable, which is not easy to be dealt with via a pure PDE method. Our stochastic control approach also takes advantage of imposing less regularity on the coefficients over the analytical approach. In particular, when the coefficient functions are $C^1$ in both spatial and control variables and are just continuous in the distribution argument, MFG \eqref{intro_1} and Problem \eqref{intro_1'} can still be warranted with a unique solution; besides, when the coefficient functions are $C^2$ in both spatial and control variables and are also $C^1$ in the distribution argument, the value functional is shown to be $C^2$ in spatial variable and $C^1$ in distribution variable; furthermore, when the coefficient functions are $C^3$ in both spatial and control variables and $C^2$ in the distribution argument, the master equation \eqref{intro_5} is shown to have a unique classical solution. 

In the  literature, the first batch of solubility for mean field problems focus on  the linear settings, such as \cite{AB11,CR}. One major difficulty of our problem consists in the nonlinearity---of both drift and  diffusion in the distribution,  and also of the drift in both state and control---so as to ensure the global-in-time solution. The lately proposed cone property for the state process and the adjoint process is crucial to overcome the difficulty, and see~\cite{AB10'',AB10'} for its role in the first order mean field theory via the approach of partitioning the time domain. To the best of our knowledge, even for the standard stochastic control problems, this cone property is brand new in the literature. Particularly, for classical stochastic control problems in the absence of mean field terms, the FBSDEs approach is dated back to Bismut \cite{Bismut} for the linear quadratic optimal stochastic control problem; and the nonlinear cases  more conveniently appeal to the Bellman equations. Yet for MFGs (and also mean field type control problems), the Bellman equation and the master equations evolve in a measure space, and are considerably difficult to be solved. While the FBSDEs approach turns out to be more effective since the mean field term causes little difference to it from augmenting an extra state, as the FBSDEs technique is by default handling the control problem from an infinite-dimensional perspective; besides, our cone property is quite natural to be posed so as to monitor the growths of various sensitivities of adjoint processes with respect to different initial data. Especially for the nonlinear dynamics, the probabilistic method together with the cone property actually enable us to resolve the mean field problem in a linearized manner as a straightforward way aligning with the conventional approach in the classical stochastic control theory.

In this article, the running cost function is allowed to be non-separable and to have a quadratic growth, and is assumed to be strongly convex and to have a small mean field effect. The small mean field effect conditions are widely discussed in the literature in various settings, and we refer to our previous work \cite{AB12} for a detailed discussion on the relationship among this small mean field effect assumption, the $\beta$-monotonicity, the displacement monotonicity, and Lasry-Lions monotonicity. Especially, the convexity and the small mean field effect condition for the running cost function combined with the cone property together imply the $\beta$-monotonicity, and thus the well-posedness for the associated FBSDEs and the regularity of their solutions. In \cite{AB12}, we also proposed a new monotonicity condition for the running cost function, namely the running cost function consists of  two parts which all depend on the control and the state distribution, yet one of which has a strong convexity and a small mean field effect condition, while another has a newly introduced displacement quasi-monotonicity. We emphasize that the assumptions for the running cost function in this article can be generalized to a more general case as that in \cite{AB12}, while to avoid technicalities we only state the results in a modest way so that the reader can grasp the mainstream flow of ideas; for the statements and proofs under more general conditions, we shall leave them to the interested reader and in our future further explorations. Moreover, our present stochastic control method can be applied to some more general cases beyond usual convex settings, such as the case when the terminal functional $g$ is non-convex; also see \cite[Section 3.2]{AB12} for details. From the stochastic control perspective, our result can also be extended to mean field problems involving common noises with the distribution of the state being replaced with the distribution  conditioned on the filtration generated by the common noise; we refer to \cite{AB13} for mean field type control with common noises. 

The rest of this article is organized as follows. In  Section~\ref{sec:MP}, we give the well-posedness of FBSDEs arising from the maximum principle associated with MFG \eqref{intro_1} and Problem \eqref{intro_1'}, and  also introduce the cone property. In Section~\ref{sec:distribution}, we study the Jacobian flows for FBSDEs associated to MFG \eqref{intro_1} and Problem \eqref{intro_1'} so as to give the G\^ateaux derivatives of the solution processes of these FBSDEs with respect to the initial condition. In Section~\ref{sec:second_deri}, we study the Hessian flows of the second order derivatives of the same set of processes. In Section~\ref{sec:V}, we give the regularity of the value functional $V$, and we eventually establish that $V$ is the unique classical solution of the master equation \eqref{intro_5}. Some statements in Sections~\ref{sec:MP}, \ref{sec:distribution}, \ref{sec:second_deri} and \ref{sec:V} are proven in Appendices \ref{pf:lem:5}, \ref{pf:distribution}, \ref{pf:prop:9} and \ref{pf:prop:10}, respectively.

\subsection{Notations}\label{sec:notation}

For any $X\in L^2(\Omega,\f,\mathbb{P};\brn)$, we denote by $\lr(X)$ its law and by $\|X\|_2$ its $L^2$-norm. For every $t\in[0,T]$, we  denote by $L^2_{\f_t}$ the set of all $\f_t$-measurable square-integrable $\brn$-valued random vectors, and denote by $\lr^2_{\f}(0,T)$  the set of all $\f_t$-progressively-measurable $\brn$-valued processes $\alpha_\cdot=\{\alpha_t,\ 0\le t\le T\}$ such that $\e\left[\int_0^T |\alpha_t|^2dt\right]<+\infty$. We  denote by $\mathcal{S}^2_{\f}(0,T)$ the family of all $\f_t$-progressively-measurable $\brn$-valued processes $\alpha_\cdot=\{\alpha_t,\ 0\le t\le T\}$ such that $\e\left[\sup_{0\le t\le T} |\alpha_t|^2\right]<+\infty$. We denote by $\mathcal{P}_{2}(\brn)$ the space of all probability measures of finite second order moments on $\brn$, equipped with the 2-Wasserstein metric: $$W_2\left(m,m'\right):=\inf_{\pi\in\Pi\left(m,m'\right)}\sqrt{\int_{\brn\times\brn}\left|x-x'\right|^2\pi\left(dx,dx'\right)},$$
 where $\Pi\left(m,m'\right)$ is the set of joint probability measures with respective marginals $m$ and $m'$. We denote by $\delta_0$ the point mass distribution of the random variable $\xi$ such that $\mathbb{P}(\xi=\mathbf{0})=1$. More results on Wasserstein metric space can be found in \cite{AL}.

The linear functional derivative of a functional $k(\cdot):\mathcal{P}_{2}(\brn)\to\br$ at $m\in \mathcal{P}_{2}(\brn)$ is another functional $\mathcal{P}_{2}(\brn)\times \brn\ni(m,y)\mapsto\dfrac{dk}{d\nu}(m)(y)$,  being jointly continuous and satisfying 
$$\int_{\brn}\Big|\dfrac{dk}{d\nu}(m)(y)\Big|^{2}dm(y)\leq c(m)
$$
 for some positive constant $c(m)>0$ which is bounded on any bounded subsets of $\pr_2(\brn)$, such that 
\begin{equation*}
    \lim_{\epsilon\to0}\dfrac{k((1-\epsilon)m+\epsilon m')-k(m)}{\epsilon}=\int_\brn\dfrac{dk}{d\nu}(m)(y)\left(dm'(y)-dm(y)\right), \quad \forall m'\in\mathcal{P}_{2}(\brn);
\end{equation*}
we refer the reader to \cite{AB,book_mfg} for more details about the notion of linear functional derivatives. In particular, the linear functional derivatives in $\pr_2(\brn)$ are connected with the G\^ateaux derivatives in $L^2(\Omega,\f,\mathbb{P};\brn)$ in the following manner. For a linearly functional differentiable functional $k:\pr_2(\brn)\to\br$ such that the derivative $D_y\frac{d k}{d\nu}(\mu)(y)$ is jointly continuous in $(\mu,y)$ and $D_y\frac{d k}{d\nu}(\mu)(y)\le c(\mu)(1+|y|)$ for $(\mu,y)\in\pr_2(\brn)\times\brn$,  the functional $K(X):=k(\lr(X)), \  X\in L^2(\Omega,\f,\mathbb{P};\brn)$ has the following G\^ateaux derivative:
\begin{align}\label{lem01_1}
	D_X K(X)(\omega)=D_y\frac{d k}{d\nu}(\lr(X))(X(\omega)). 
\end{align}
Furthermore, if $k$ is twice linearly functional differentiable, then the functional $K$ is also twice G\^ateaux differentiable, and the G\^ateaux derivative at $X$ along a direction $Z\in L^2(\Omega,\f,\mathbb{P};\brn)$ is given by
\begin{align*}
    D_X^2 K(X)\left(Z\right)=\left(D_y^2\frac{d k}{d\nu}(\lr(X))(X)\right)^\top Z+\widetilde{\e}\left[ \left(D_{y'}D_y\frac{d^2 k}{d\nu}(\lr(X))\left(X,\tx\right)\right)^\top \tz\right]. 
\end{align*}
From here onward, for any random variable $\xi$, we write $\widetilde{\xi}$ for its independent copy, and $\widetilde{\e}[\widetilde{\xi}]$ for the corresponding expectation taken. We also refer to \cite[Section 2.5]{AB13} for G\^ateaux derivatives of functionals of conditional probability measures, in case of coping with the present problem but under the additional common noise setting.

For the sake of convenience, in this article, we write $f|_{a}^b:=f(b)-f(a)$ for the difference of a  functional $f$ between two points $b$ and $a$. For any matrix $Q\in\br^{n\times n}$ and vector $x\in\brn$, we use the notation $Qx^{\otimes 2}:=x^\top Qx$.

\section{Solubility of MFG \eqref{intro_1} and Problem \eqref{intro_1'}}\label{sec:MP}

In this section, we shall solve MFG \eqref{intro_1} and Problem \eqref{intro_1'} under the following assumptions. 

\textbf{(A1)} (i) The coefficient $b$ is continuous in $t$, 
and continuously differentiable in $x$, $v$ and $m$, with the derivatives $D_x b,\ D_v b$ being bounded by $L$, and the derivative $D_y\frac{db}{d\nu}$ being bounded by $l_b^m\le L$. All the derivatives $D_x^2 b, \ D_vD_x b, \ D_v^2 b, \ D_y \frac{d}{d\nu}D_xb, \ D_y \frac{d}{d\nu}D_v b$ exist, and they are continuous in all their arguments, and there exist nonnegative constants $L_b^0,L_b^1,L_b^2\le L$, such that for any $s\in[0,T]$, $m\in \pr_2(\brn)$, $v\in\brd$ and $x,y\in\brn$,
\begin{equation}\label{generic:condition:b}
\begin{aligned}
    \left|D_x^2 b(s,x,m,v)\right|,\quad \left| D_y \frac{d}{d\nu} D_x b(s,x,m,v)(y)\right|\le\ & \frac{L_b^0}{1+|x|+|v|+W_2(m,\delta_0)}, \\[3mm]
    \left|D_v D_x b(s,x,m,v)\right|,\quad \left| D_y \frac{d}{d\nu} D_v b(s,x,m,v)(y)\right|\le\ & \frac{L_b^1}{1+|x|+|v|+W_2(m,\delta_0)}, \\[3mm]
    \left| D_v^2 b(s,x,m,v)\right|\le\ & \frac{L_b^2}{1+|x|+|v|+W_2(m,\delta_0)}.
\end{aligned}
\end{equation}
Moreover, there exists $\lambda_b >0$ such that,
\begin{align}\label{generic:condition:b'}
    (D_v b)(D_v b)^\top(s,x,m,v) \geq \lambda_b I_n,\quad \forall (s,x,v)\in [0,T]\times \brn\times \brd.
\end{align}

(ii) The coefficient $\sigma=(\sigma^1,\dots,\sigma^n)$ is continuous in $t$ and linear in $x$ as $\sigma^j(s,x,m)=\sigma_0^j(s,m)+\sigma_1^j(s)x$ for $1\le j\le n$, with $\sigma_0^j$ being $l_\sigma^m(\le L)$ continuous in $m$ and the norm of matrix $\sigma^j_1$ being bounded  by $L$. The derivative $D_y\frac{d\sigma^j_0}{d\nu}$ exists and is bounded by $l_{\sigma}^m$. 

\textbf{(A2)} The function $f$ is continuous in $t$ and has a quadratic growth in $(x,m,v)$; the function $g$ has a quadratic growth in $(x,m)$. All the following derivatives 
$$D_x^2 f, \ D_vD_x f, \ D_v^2 f, \ D_y \frac{d}{d\nu}D_xf, \ D_y \frac{d}{d\nu}D_v f, \ D_x^2 g, \quad \text{and} \quad \ D_y \frac{d}{d\nu}D_xg$$
 exist, and are continuous in all their arguments. They also satisfy for any $(s,x,v)\in [0,T]\times\brn\times \brd$, $m,m'\in \pr_2(\brn)$ and $y,y'\in\brn$,
\begin{align*}
	\left| \left(D_x^2,D_vD_x,D_v^2\right) f (s,x,m,v) \right|+\left| D_x^2 g(x,m) \right|\le\ & L,\\[3mm]
	\left| D_y\frac{d f}{d \nu} (s,x,m',v)(y')-D_y\frac{d f}{d \nu} (s,x,m,v)(y)\right|&+\left| D_y\frac{d g}{d \nu} (x,m')(y')-D_y\frac{d g}{d \nu} (x,m)(y)\right|\\
	\le\ & L\left(W_2(m,m')+|y'-y|\right).
\end{align*}
Moreover, there exist nonnegative constants $L_f^0,L_f^1,L_g\le L$, such that
\begin{equation}\label{small_mf_condition}
	\left| D_y \frac{d}{d\nu} D_x f(s,x,m,v)(y)\right|\le L_f^0, \quad \left| D_y \frac{d}{d\nu} D_v f(s,x,m,v)(y)  \right|\le L_f^1,\quad \left| D_y \frac{d}{d\nu} D_x g(x,m)(y)\right|\le L_g.
\end{equation}

\textbf{(A3)} There exist $\lambda_v>0,\lambda_x\geq 0$ and  $\lambda_g\geq 0$ such that, for any $(s,m)\in[0,T]\times\pr_2(\brn)$, $x,x'\in\brn$ and $v,v'\in\brd$,
\begin{align*}
	&f\left(s,x',m,v'\right)-f(s,x,m,v)\geq \left(D_x f (s,x,m,v)\right)^\top  \left(x'-x\right)+\lambda_x \left|x'-x\right|^2\\
	&\qquad\qquad\qquad\qquad\qquad\qquad\qquad +\left(D_v f (s,x,m,v)\right)^\top  \left(v'-v\right)+\lambda_v \left|v'-v\right|^2,\\[3mm]
	&g\left(x',m\right)-g(x,m)\geq \left(D_x g (x,m)\right)^\top  \left(x'-x\right)+\lambda_g \left|x'-x\right|^2.
\end{align*}

We now illustrate more on our assumptions in (A1)-(A3). From Condition \eqref{generic:condition:b}, we can see that when the drift function $b$ is linear in $x$ and $v$ with distribution-independent coefficients, then $L_b^0=L_b^1=L_b^2=0$. Therefore, (A1) extends the standard linear assumption on $b$ in the literature, such as \cite{AB11,book_mfg,HZ2}. Our Assumption (A1) is also used in \cite{AB10'',AB10'} for the first order mean field theory with a generic drift. Condition \eqref{small_mf_condition} is actually the small mean field effect condition with suitable choice of the parameters $L_f^0$, $L_f^1$ and $L_g$, we refer to \eqref{lem:2_0} for its role; we also refer to \cite{AB11} for similar assumptions; indeed, the boundedness condition \eqref{generic:condition:b}, and the positive definite condition \eqref{generic:condition:b'}, the small mean field effect condition \eqref{small_mf_condition} in (A2), the convexity conditions in (A3), and the conditions of the parameters in the following \eqref{lem:2_0} altogether imply the $\beta$-monotonicity, and thus the well-posedness for the associated FBSDEs and the regularity of the solutions. For the definition of the $\beta$-monotonicity and the usage of it in proving the solvability of MFGs with generic drifts, we refer to \cite[Condition 2.1]{AB12} and \cite[Section 5.1]{AB12}, respectively, and shall not repeat these again in this article. The $\beta$-monotonicity can include the displacement monotonicity and small mean field effect condition as special cases (see \cite[Section 2]{AB12}), and it can also include some examples which cannot be dealt by the displacement monotonicity (see \cite[Section 3]{AB12}); we also refer to \cite{AB12} for details and also discussions on different monotonicity assumptions for MFGs. Our Assumption (A3) for $f$ can be replaced by a more general condition as Condition 3.3 in \cite{AB12}, namely partitioning it into two parts, so that both parts still depend on the control and the state distribution, yet one satisfies a strong convexity and a small mean field effect condition, while the other has a newly introduced displacement quasi-monotonicity (defined in (3.18) in \cite{AB12}). However, this will involve much more parameters and notations, so we leave the case with the abovementioned assumption to the readers to verify. 

We also need the following condition on the bounding parameters: 
\begin{definition}
We say that the mean field system \eqref{intro_1} has Property (S), if the followings hold:
\begin{equation}\label{lem:2_0}
    \begin{aligned}
        &2\lambda_g> L_g;\\
        &2\lambda_x> L_f^0+\frac{2L^2L_b^0}{\lambda_b} +2c_1(L,\lambda_b) l_b^m +\sqrt{2}c_2(L,\lambda_b)l_{\sigma}^m + \frac{3nL^2(l_\sigma^m)^2}{2\lambda_g-L_g};\\[3mm]
        &2\lambda_v>\frac{L^2 L_b^2}{\lambda_b}+ \frac{L_f^1+\frac{3L^2L_b^1}{\lambda_b}+c_1(L,\lambda_b)l_b^m +c_2(L,\lambda_b)l_{\sigma}^m}{2\sqrt{2\lambda_x- L_f^0-\frac{2L^2L_b^0}{\lambda_b} -2c_1(L,\lambda_b)l_b^m   -\sqrt{2}c_2(L,\lambda_b)l_{\sigma}^m - \frac{3nL^2(l_\sigma^m)^2}{2\lambda_g-L_g}}},
    \end{aligned}
\end{equation}
\normalsize
where
\begin{align*}
    c_1(L,\lambda_b):= \frac{L^2}{\lambda_b} \left(1+\frac{L^2}{\lambda_b}\right),\quad c_2(L,\lambda_b):=4L\sqrt{\frac{nL}{\lambda_b}\left(1+\frac{L^2}{\lambda_b}\right)\left(1+\frac{L^2(1+nL)}{\lambda_b}\left(1+\frac{L^2}{\lambda_b}\right)\right)}.
\end{align*}     
\end{definition}

Recall in Assumption (A2) that, the bounding parameters $L_f^0$, $L_f^1$ and $L_g$ represent the dependence of the derivatives $D_x f$, $D_v f$ and $D_x g$ on the distribution variable, respectively; and recall in Assumption (A2) that, the parameters $l_b^m$ and $l_\sigma^m$ represent the dependence of the coefficients $b$ and $\sigma$ in the distribution variable, respectively. The parameters $L_b^0$, $L_b^1$ and $L_b^2$ also indicate how the non-linear the drift function $b$ is. Condition \eqref{lem:2_0} actually states that the convexity of the cost functions is required to be larger than a constant depending on the above parameters. As a particular case, when the coefficients $b$ and $\sigma$ are independent of $m$ and the drift function $b$ is linear in $x$ and $v$, then $l_b^m=l_\sigma^m=L_b^0=L_b^1=L_b^2=0$ and Condition \eqref{lem:2_0} is 
\begin{align*}
    &2\lambda_g> L_g,\quad 2\lambda_x> L_f^0,\quad 2\lambda_v>\frac{L_f^1}{2\sqrt{2\lambda_x- L_f^0}},
\end{align*}
which is equivalent with the small mean field assumption proposed in \cite{AB9'1} and used in \cite{AB11} for the linear case. 

Under the above assumptions, MFG \eqref{intro_1} is associated with the following system of FBSDEs arising from the stochastic maximum principle:
\begin{equation}\label{FB:mfg_generic}
\left\{
    \begin{aligned}
        &X(s) = \xi+\int_t^s b\left(r,X(r),\lr(X(r)),v(r)\right)dr+ \int_t^s \left[\sigma_0\left(r,\lr(X(r))\right)+\sigma_1(r)X(r)\right]dB(r),\\[3mm]
        &P(s) = -\int_s^T Q(r) dB(r)+ D_x g\left(X(T),\lr(X(T))\right)\\
        &\quad\qquad +\int_s^T \bigg[ D_x b\left(r,X(r),\lr(X(r)),v(r)\right)^\top P(r)+\sum_{j=1}^n \left(\sigma^j_1(r)\right)^\top Q^j(r)\\
        &\quad\qquad\qquad\qquad + D_x f\left(r,X(r),\lr(X(r)),v(r)\right)\bigg]dr,\\[3mm]
        &D_v b\left(s,X(s),\lr(X(s)),v(s)\right)^\top P(s)+ D_v f\left(s,X(s),\lr(X(s)),v(s)\right)=0,\quad s\in[t,T],
    \end{aligned}
\right.
\end{equation}
whose solution will be denoted by $\left(X_{t\xi},P_{t\xi},Q_{t\xi},v_{t\xi}\right)$ with the subscript $``t\xi"$ indicating the dependence  on the initial data. The third equation in the system \eqref{FB:mfg_generic} is the necessary condition for the optimal control arising from the maximum principle; and this maximum principle \eqref{FB:mfg_generic} is also  sufficient  under the convex settings. In our previous work \cite[Theorem 5.1]{AB12}, we have proven the sufficient maximum principle for MFG \eqref{intro_1}: if the system of FBSDEs \eqref{FB:mfg_generic} has a solution $\left(X_{t\xi},P_{t\xi},Q_{t\xi},v_{t\xi}\right)\in \sr^2_\f(t,T)\times\sr^2_{\f}(t,T)\times\left(\lr^2_{\f}(t,T)\right)^n\times\lr^2_{\f}(t,T)$, then, $v_{t\xi}$ is the unique solution of MFG \eqref{intro_1}. In \cite[Theorem 5.2]{AB12}, we give a well-posedness result of FBSDEs \eqref{FB:mfg_generic} by using the well-posedness result for FBSDEs in Hilbert spaces under $\beta$-monotonicity. Here, we first give the $L^2$-boundedness and continuity of the solution with respect to the initial condition $\xi$.

\begin{theorem}\label{lem:2}
	Under Assumptions (A1)-(A3) and the validity of Property (S), there is a unique adapted solution $\left(X_{t\xi},P_{t\xi},Q_{t\xi},v_{t\xi}\right)$ of the FBSDEs \eqref{FB:mfg_generic}, and $v_{t\xi}$ is the unique solution of MFG \eqref{intro_1}. For $\xi,\xi'\in L_{\f_t}^2$, we have	
	\small
	\begin{align}
		&\e\left[\sup_{t\le s\le T}\left|\left(X_{t\xi}(s), P_{t\xi}(s),v_{t\xi}(s)\right)\right|^2+\int_t^T |Q_{t\xi}(s)|^2ds \right] \notag \\[2mm]
        \le\ & C(L,T,\lambda_b,\lambda_v,\lambda_x,l_b^m, l_\sigma^m,L_b^0,L_b^1,L_b^2,L_f^0,L_f^1)\left(1+\|\xi\|_2^2\right), \label{lem2_1}\\[3mm]
		&\e\bigg[\sup_{t\le s\le T}\left|\left(X_{t\xi'}(s)-X_{t\xi}(s), P_{t\xi'}(s)-P_{t\xi}(s),v_{t\xi'}(s)-v_{t\xi}(s)\right)\right|^2+ \int_t^T \left|Q_{t\xi'}(s)-Q_{t\xi}(s)\right|^2 ds\bigg] \notag \\[2mm]
        \le\ &  C(L,T,\lambda_b,\lambda_v,\lambda_x,l_b^m, l_\sigma^m,L_b^0,L_b^1,L_b^2,L_f^0,L_f^1)\e\|\xi'- \xi\|_2^2. \label{lem2_2}
	\end{align}	
	\normalsize
\end{theorem}

\begin{remark}
    Assumptions (A1)-(A3) and  Property (S) altogether imply the $\beta$-monotonicity \cite{AB12} corresponding to FBSDEs \eqref{FB:mfg_generic}, and thus the well-posedness for the associated FBSDEs by the continuation method in the coefficients \cite{YH2,SP}. The continuation method in the coefficients under the $\beta$-monotonicity \cite{AB10,AB11} involves partition of  the spatial domain, which is quite different from that of the temporal partition. In general, the latter leads us from a PDE perspective while the former is taken from a probabilistic perspective. The latter is used in our previous work \cite{AB10'} for the first order MFGs with nonlinear drift functions and in \cite{AB9'1} for the second order MFG with linear drift and constant diffusion; while the former approach is used in our recent work \cite{AB11} for MFGs with linear drifts and linear diffusions, see \cite{AB12} for a more detailed discussion. In view of \cite{AB11,AB12,AB9'1,AB10'} and also \eqref{lem2_2} and Lemma~\ref{lem:3} of the present article, to give the global solvability of the required FBSDEs, the boundedness result of the corresponding Jacobian flows helps in both the aforementioned approaches. We emphasize that it is not that one approach is better than the other; indeed, it is upon whether an approach is more convenient to one problem in a certain related setting; for the second order nonlinear MFG considered in the current article, the former seems to be more convenient. 
\end{remark}


\begin{proof}
	Both well-posedness  and sufficiency of the maximum principle are proven in our previous work \cite[Theorems 5.1 and 5.2]{AB12}; here, we only prove \eqref{lem2_1} and \eqref{lem2_3}, and the proof of \eqref{lem2_2} is similar to that of \eqref{lem2_1}. For simplicity of notations, we omit the subscript $t\xi$ in $(X_{t\xi},P_{t\xi},Q_{t\xi},v_{t\xi})$. Applying It\^o's formula to $P(s)^\top X(s)$ and taking expectation, we have
\begin{align*}
    &\e\left[D_xg(X(T),\lr(X(T)))^\top X(T)-P(t)^\top \xi \right] \notag\\[3mm]
    =\ & \e\bigg[\int_t^T \left[b(s,X(s),\lr(X(s)),v(s))-D_x b (s,X(s),\lr(X(s)),v(s)) X(s)\right]^\top P(s)  \\
    &\qquad + \sum_{j=1}^{n} \sigma_0^j (s,\lr(x(s)))^\top Q^j(s)-D_xf (s,X(s),\lr(X(s)),v(s))^\top X(s) ds\bigg].
\end{align*}
Substituting the third equation of FBSDEs \eqref{FB:mfg_generic} into the last equality, we obtain
\begin{align}
    &\e\left[D_xg(X(T),\lr(X(T)))^\top X(T)-P(t)^\top \xi \right] \notag\\
    =\ & \e\bigg\{\int_t^T \left[b(s,X(s),\lr(X(s)),v(s))-\left[(D_x b,D_v b) (s,X(s),\lr(X(s)),v(s)) \right]\begin{pmatrix}
        X(s)\\ v(s)
    \end{pmatrix}\right]^\top P(s) \notag \\
    &\qquad + \sum_{j=1}^{n} \sigma_0^j (s,\lr(x(s)))^\top Q^j(s)-\left[\begin{pmatrix}
        D_x f\\ D_v f
    \end{pmatrix}(s,X(s),\lr(X(s)),v(s))\right]^\top \begin{pmatrix}
        X(s)\\ v(s)
    \end{pmatrix} ds\bigg\} \notag \\
    =\ & \e\Bigg\{\int_t^T \left[b(s,\cdot,\lr(X(s)),\cdot)\bigg|^{(X(s),v(s))}_{(0,0)} -\left[(D_x b,D_v b) (s,X(s),\lr(X(s)),v(s)) \right]\begin{pmatrix}
        X(s)\\ v(s)
    \end{pmatrix}\right]^\top P(s)  \notag \\
    &\qquad -\left[\begin{pmatrix}
        D_x f\\ D_v f
    \end{pmatrix}(s,\cdot,\lr(X(s)),\cdot)\bigg|^{(X(s),v(s))}_{(0,0)} \right]^\top \begin{pmatrix}
        X(s)\\ v(s)
    \end{pmatrix} +\sum_{j=1}^n\bigg[\sigma^j_0(s,\cdot)\Big|^{\lr(X(s))}_{\delta_0} \bigg]^\top Q^j(s) \notag \\
    &\qquad + \left[b(s,0,\cdot,0))\Big|^{\lr(X(s))}_{\delta_0}\right]^\top P(s)-\left[\begin{pmatrix}
        D_x f\\ D_v f
    \end{pmatrix}(s,0,\cdot,0)\bigg|^{\lr(X(s))}_{\delta_0} \right]^\top \begin{pmatrix}
        X(s)\\ v(s)
    \end{pmatrix} \notag\\
    &\qquad + b(s,0,\delta_0,0))^\top P(s) + \sum_{j=1}^{n}  \sigma^j_0(s,\delta_0)^\top Q^j(s)- \left[\begin{pmatrix}
        D_x f\\ D_v f
    \end{pmatrix}(s,0,\delta_0,0)\right]^\top \begin{pmatrix}
        X(s)\\ v(s)
    \end{pmatrix} ds\Bigg\}. \label{lem:2_1}
\end{align}
From the third equation of FBSDEs \eqref{FB:mfg_generic} and Condition \eqref{generic:condition:b'}, we have
\begin{align*}
    P(s)=-\left((D_vb) (D_vb)^\top\right)^{-1} (D_vb) (D_v f)(s,X(s),\lr(X(s)),v(s)), 
\end{align*}
and then, from Assumption (A2), we know that
\begin{align}
    |P(s)| &\le \frac{1}{\lambda_b} \left|D_vb(s,X(s),\lr(X(s)),v(s))\right|\cdot \left|D_v f(s,X(s),\lr(X(s)),v(s))\right| \notag \\
    &\le \frac{1}{\lambda_b} \cdot L \cdot L \left[1+|X(s)|+W_2(\lr(X(s)),\delta_0)+|v(s)|\right] \notag \\
    &= \frac{L^2}{\lambda_b} \left[1+|X(s)|+W_2(\lr(X(s)),\delta_0)+|v(s)|\right]. \label{lem:2_2}
\end{align}
Then, from Condition \eqref{generic:condition:b}, we further deduce know that
\begin{align}
    &\left|\left[b(s,\cdot,\lr(X(s)),\cdot)\Big|^{(X(s),v(s))}_{(0,0)} -\left[(D_x b,D_v b) (s,X(s),\lr(X(s)),v(s)) \right]\begin{pmatrix}
        X(s)\\ v(s)
    \end{pmatrix}\right]^\top P(s)\right| \notag \\
    \le\ & \frac{L^2 L_b^0}{\lambda_b} |X(s)|^2 +\frac{2L^2 L_b^1}{\lambda_b} |X(s)|\cdot |v(s)| +\frac{L^2 L_b^2}{\lambda_b} |v(s)|^2. \label{lem:2_3}
\end{align}
From \eqref{lem:2_2} and Assumption (A1), we know that 
\begin{align}
    \left|\left[b(s,0,\cdot,0))\Big|^{\lr(X(s))}_{\delta_0}\right]^\top P(s)\right| \le \frac{L^2l_b^m}{\lambda_b} W_2(\lr(X(s)),\delta_0) \left[1+|X(s)|+W_2(\lr(X(s)),\delta_0)+|v(s)|\right]. \label{lem:2_4}
\end{align}
From the convexity assumption in (A3), we see that
\begin{equation}\label{lem:2_5}
\begin{aligned}
    \left[\begin{pmatrix}
        D_x \\ D_v 
    \end{pmatrix}f(s,\cdot,\lr(X(s)),\cdot)\bigg|^{(X,v)(s)}_{(0,0)} \right]^\top \begin{pmatrix}
        X\\ v
    \end{pmatrix} (s)\geq\ & 2\lambda_v|v(s)|^2+2\lambda_x |X(s)|^2, \\
    \left[D_xg(\cdot,\lr(X(T)))\Big|^{X(T)}_{0} \right]^\top X(T)\geq\ & 2\lambda_g |X(T)|^2.
\end{aligned}
\end{equation}
From Condition \eqref{small_mf_condition}, we have:
\begin{equation}\label{lem:2_5'}
\begin{aligned}
    &\left|\left[\begin{pmatrix}
        D_x \\ D_v 
    \end{pmatrix}f(s,0,\cdot,0)\bigg|^{\lr(X(s))}_{\delta_0} \right]^\top \begin{pmatrix}
        X\\ v
    \end{pmatrix}(s)\right| \le\ L_f^0\  W_2(\lr(X(s)),\delta_0)\ |X(s)|+L_f^1 \ W_2(\lr(X(s)),\delta_0)\ |v(s)|,\\
    &\left|\left[D_xg(0,\cdot)\Big|^{\lr(X(T))}_{\delta_0}\right]^\top X(T)\right|\le L_g W_2(\lr(X(T)),\delta_0)|X(T)|.
\end{aligned}
\end{equation}
From Assumption (A1) and Cauchy's inequality, we also know that
\begin{align}
    \e \bigg|\int_t^T \sum_{j=1}^n\bigg[\sigma^j_0(s,\cdot)\Big|^{\lr(X(s))}_{\delta_0} \bigg]^\top Q^j(s) ds \bigg| \le\ & l_\sigma^m \int_t^T W_2(\lr(X(s)),\delta_0) \sum_{j=1}^n \left\|Q^j(s)\right\|_2 ds \notag \\
    \le\ & l_\sigma^m \sqrt{n} \left(\int_t^T W_2^2(\lr(X(s)),\delta_0) ds\right)^{\frac{1}{2}} \cdot \left(\int_t^T \|Q(s)\|_2^2 ds\right)^{\frac{1}{2}}. \label{sigma_m_1}
\end{align}
Applying It\^o's formula to $|P(s)|^2$, we have
\begin{align*}
    &\e\left[|D_x g(X(T),\lr(X(T)))|^2-|P(t)|^2 \right]\\
    =\ &\e\bigg[\int_t^T \bigg(\sum_{j=1}^n \left|Q^j(s)\right|^2-2\sum_{j=1}^n P(s)^\top \left(\sigma^j_1(s)\right)^\top Q^j(s)\\
    &\qquad\qquad -2P(s)^\top D_x b(s,X(s),\lr(X(s)),v(s))^\top P(s) -2P(s)^\top D_x f(s,X(s),\lr(X(s)),v(s))\bigg) ds\bigg].
\end{align*}
Further, 
\begin{align*}
    &\e\bigg[\int_t^T \sum_{j=1}^n \left|Q^j(s)\right|^2 ds\bigg]\\
    \le\ & 2\e\bigg[\int_t^T \bigg( \sum_{j=1}^n P(s)^\top \left(\sigma^j_1(s)\right)^\top Q^j(s) +P(s)^\top D_x b(s,X(s),\lr(X(s)),v(s))^\top P(s)\\
    &\qquad\qquad  +P(s)^\top D_x f(s,X(s),\lr(X(s)),v(s))\bigg) ds\bigg]+\e\left[|D_x g(X(T),\lr(X(T)))|^2\right]\\
    \le\ & 2L \int_t^T \bigg( \sum_{j=1}^n \|P(s)\|_2 \cdot\left\|Q^j(s)\right\|_2 +\|P(s)\|_2^2+ \|P(s)\|_2\cdot \left(1+2\|X(s)\|_2+\|v(s)\|_2\right) \bigg) ds\\
    &+3L^2\left(1+2\|X(T)\|_2^2\right)\\
    \le\ & \frac{1}{2}\int_t^T\sum_{j=1}^n \left\|Q^j(s)\right\|_2^2ds +(2nL^2+2L)\int_t^T \|P(s)\|_2^2 ds\\
    &+2L \int_t^T \|P(s)\|_2\cdot \left(1+2\|X(s)\|_2+\|v(s)\|_2\right) ds+3L^2\left(1+2\|X(T)\|_2^2\right).
\end{align*}
Then, using the estimate \eqref{lem:2_2}, we see
\begin{align}
    &\int_t^T \left\|Q(s)\right\|_2^2ds \notag \\
    \le\ & \frac{16L^3}{\lambda_b^2}\left(\lambda_b+L^2(1+nL)\right)\int_t^T \!\! \left(1+2\|X(s)\|^2_2+\|v(s)\|^2_2\right) ds+6L^2\left(1+2\|X(T)\|_2^2\right). \label{sigma_m_2}
\end{align}
Substituting \eqref{sigma_m_2} into \eqref{sigma_m_1}, we have
\begin{align}
    &\e \bigg|\int_t^T \sum_{j=1}^n\bigg(\sigma^j_0(s,\cdot)\bigg|^{\lr(X(s))}_{\delta_0} \bigg)^\top Q^j(s) ds \bigg| \notag \\
    \le\ & \frac{4\sqrt{nL^3} l_\sigma^m }{\lambda_b}\sqrt{\lambda_b+L^2(1+nL)} \ \left(\int_t^T \|X(s)\|_2^2 ds\right)^{\frac{1}{2}}  \left(\int_t^T  \|v(s)\|^2_2 ds\right)^{\frac{1}{2}} \notag \\
    &+\frac{4\sqrt{2nL^3} l_\sigma^m }{\lambda_b}\sqrt{\lambda_b+L^2(1+nL)}  \int_t^T \|X(s)\|_2^2 ds +2\sqrt{3n}Ll_\sigma^m  \left(\int_t^T \|X(s)\|_2^2 ds\right)^{\frac{1}{2}} \|X(T)\|_2 \notag \\
    &+C(L,\lambda_b,T) \left(\int_t^T \|X(s)\|_2^2 ds\right)^{\frac{1}{2}}. \label{sigma_m_3}
\end{align}
Substituting \eqref{lem:2_3}-\eqref{sigma_m_3} back into \eqref{lem:2_1}, we have
\begin{align}
    &\e\left[2\lambda_g |X(T)|^2+\int_t^T 2\lambda_v|v(s)|^2+2\lambda_x|X(s)|^2 ds \right] \notag\\
    \le\ & \e\Bigg\{\int_t^T \bigg[ \frac{L^2 L_b^0}{\lambda_b} |X(s)|^2 +\frac{2L^2 L_b^1}{\lambda_b} |X(s)|\cdot |v(s)| +\frac{L^2 L_b^2}{\lambda_b} |v(s)|^2+\frac{L^2l_b^m}{\lambda_b} W^2_2(\lr(X(s)),\delta_0)  \notag \\
    &\qquad\qquad + \left(\frac{L^2l_b^m}{\lambda_b} + L_f^0\right) W_2(\lr(X(s)),\delta_0)\  |X(s)|+\left(\frac{L^2l_b^m}{\lambda_b}+L_f^1 \right) W_2(\lr(X(s)),\delta_0)|v(s)| \bigg] ds \notag \\
    &\qquad +L_g W_2(\lr(X(T)),\delta_0)\ |X(T)|\Bigg\} \notag \\
    &+\e\bigg\{\int_t^T \bigg[ \frac{L^2l_b^m}{\lambda_b} W_2(\lr(X(s)),\delta_0)+ |b(s,0,\delta_0,0))| \cdot|P(s)|+ \sum_{j=1}^{n} \left|\sigma_0(s)\right| \cdot \left|Q^j(s)\right| \notag \\
    &\quad\qquad\qquad + |D_xf(s,0,\delta_0,0)|\cdot |X(s)|+ |D_vf(s,0,\delta_0,0)|\cdot |v(s)| \bigg]ds \notag\\
    &\qquad + |D_xg(0,\delta_0)|\cdot |X(T)| + |P(t)|\cdot |\xi|\bigg\} \notag \\
    & +\frac{4\sqrt{nL^3} l_\sigma^m }{\lambda_b}\sqrt{\lambda_b+L^2(1+nL)} \left(\int_t^T \|X(s)\|_2^2 ds\right)^{\frac{1}{2}} \left(\int_t^T  \|v(s)\|^2_2 ds\right)^{\frac{1}{2}} \notag \\
    & +\frac{4\sqrt{2nL^3} l_\sigma^m }{\lambda_b}\sqrt{\lambda_b+L^2(1+nL)} \int_t^T \|X(s)\|_2^2 ds +2\sqrt{3n}Ll_\sigma^m  \left(\int_t^T \|X(s)\|_2^2 ds\right)^{\frac{1}{2}} \|X(T)\|_2 \notag \\
    & +C(L,\lambda_b,T) \left(\int_t^T \|X(s)\|_2^2 ds\right)^{\frac{1}{2}} \notag \\
    \le\ & \frac{L^2 L_b^2}{\lambda_b}\int_t^T \|v(s)\|_2^2 ds + L_g\|X(T)\|_2^2 +2\sqrt{3n}Ll_\sigma^m  \left(\int_t^T \|X(s)\|_2^2 ds\right)^{\frac{1}{2}} \|X(T)\|_2 \notag \\
    &+ \left[\frac{L^2 (L_b^0+2l_b^m)}{\lambda_b} + L_f^0 + \frac{4\sqrt{2nL^3} l_\sigma^m }{\lambda_b}\sqrt{\lambda_b+L^2(1+nL)}\right]\int_t^T \|X(s)\|_2^2 ds\notag\\
    &+\left[\frac{L^2 (2L_b^1+l_b^m)}{\lambda_b}+L_f^1 + \frac{4\sqrt{nL^3} l_\sigma^m }{\lambda_b}\sqrt{\lambda_b+L^2(1+nL)}\right] \left(\int_t^T \|X(s)\|_2^2 ds\right)^{\frac{1}{2}} \left(\int_t^T\|v(s)\|_2^2ds\right)^{\frac{1}{2}} \notag \\
    &+C(L,\lambda_b, T)\Bigg[ \left(\int_t^T \|X(s)\|_2^2 ds\right)^{\frac{1}{2}}+ \left(\int_t^T \|P(s)\|_2^2 ds\right)^{\frac{1}{2}}+ \left(\int_t^T \|Q(s)\|_2^2 ds\right)^{\frac{1}{2}} \notag \\
    &\ \qquad\qquad\qquad + \left(\int_t^T \|v(s)\|_2^2 ds\right)^{\frac{1}{2}}+\|X(T)\|_2 + \|P(t)\|_2 \cdot \|\xi\|_2 \Bigg]. \notag 
\end{align}
Then, we see from Young's inequality that 
\begin{align*}
    &\Bigg|\sqrt{2\lambda_g-L_g}\|X(T)\|_2 - \frac{\sqrt{3n}Ll_\sigma^m}{\sqrt{2\lambda_g-L_g}} \left(\int_t^T \|X(s)\|_2^2 ds\right)^{\frac{1}{2}} \Bigg|^2\\
    &+\Bigg|\sqrt{\frac{\Lambda_1}{\lambda_b}}\left(\int_t^T \|X(s)\|_2^2 ds\right)^{\frac{1}{2}}
    - \frac{\Lambda_2}{2\sqrt{\Lambda_1}} \left(\int_t^T \|v(s)\|_2^2 ds\right)^{\frac{1}{2}}\Bigg|^2\\
    &+ \left(2\lambda_v-\frac{L^2 L_b^2}{\lambda_b}-\frac{\Lambda_2}{2\sqrt{\Lambda_1}}  \right) \int_t^T \|v(s)\|_2^2 ds\\
    \le\ & C(L,\lambda_b, T)\Bigg[ \left(\int_t^T \|X(s)\|_2^2 ds\right)^{\frac{1}{2}}+ \left(\int_t^T \|P(s)\|_2^2 ds\right)^{\frac{1}{2}}+ \left(\int_t^T \|Q(s)\|_2^2 ds\right)^{\frac{1}{2}} \notag \\
    &\qquad\qquad\qquad + \left(\int_t^T \|v(s)\|_2^2 ds\right)^{\frac{1}{2}}+\|X(T)\|_2 + \|P(t)\|_2 \cdot \|\xi\|_2 \Bigg], 
\end{align*}
with
\begin{align*}
    \Lambda_1:=\ & 2\lambda_b\lambda_x- \lambda_bL_f^0-L^2 (L_b^0+2l_b^m)- 4\sqrt{2nL^3} l_\sigma^m \sqrt{\lambda_b+L^2(1+nL)}- \frac{3n\lambda_bL^2(l_\sigma^m)^2}{2\lambda_g-L_g},\\
    \Lambda_2:=\ & \sqrt{\lambda_b} L_f^1+\frac{L^2}{\sqrt{\lambda_b}} (2L_b^1+l_b^m) + 4\sqrt{\frac{nL^3}{\lambda_b}} l_\sigma^m \sqrt{\lambda_b+L^2(1+nL)}.
\end{align*}
From \eqref{lem:2_0}, we know that for any $\epsilon>0$,
\begin{align}
    &\int_t^T \|v(s)\|_2^2 \ ds \notag \\
    \le\ &C\Bigg[ \left(\int_t^T \|X(s)\|_2^2\  ds\right)^{\frac{1}{2}}+ \left(\int_t^T \|P(s)\|_2^2\  ds\right)^{\frac{1}{2}}+ \left(\int_t^T \|Q(s)\|_2^2\  ds\right)^{\frac{1}{2}} \notag \\
    &\quad\qquad + \left(\int_t^T \|v(s)\|_2^2\  ds\right)^{\frac{1}{2}}+\|X(T)\|_2 + \|P(t)\|_2 \cdot \|\xi\|_2 \Bigg] \notag \\
    \le\ & \epsilon \left[\sup_{t\le s\le T}\|X(s)\|_2^2+\sup_{t\le s\le T}\|P(s)\|_2^2+\int_t^T  \Big( \left\|Q(s)\right\|_2^2 + \|v(s)\|_2^2 \Big)\  ds\right]+\frac{C}{\epsilon} \left(1+\|\xi\|_2^2\right).  \label{lem:2_6}
\end{align}
Using Gr\"onwall's inequality, we have
\begin{equation}\label{lem:2_7}
	\e\bigg[\sup_{t\le s\le T}|X(s)|^2\bigg]\le C(L,T)\ \e\left[1+|\xi|^2+\int_t^T |v(s)|^2\ ds\right].
\end{equation}
The estimate \eqref{lem:2_7}, usual BSDE estimates in \cite{YH2,SP,MR1696772} together with the well-known Burkholder-Davis-Gundy (BDG) inequality altogether give
\begin{equation}\label{lem:2_8}
	\e\bigg[\sup_{t\le s\le T}|P(s)|^2+\int_t^T |Q(s)|^2 \ ds\bigg]\le C(L,T)\ \e\left[1+|\xi|^2+\int_t^T |v(s)|^2\ ds\right].
\end{equation}   
Substituting \eqref{lem:2_7} and \eqref{lem:2_8} into \eqref{lem:2_6}, we have
\begin{align*}
    & \int_t^T \|v(s)\|_2^2\  ds \le \epsilon \, C(L,T)\left[\int_t^T  \|v(s)\|_2^2 \  ds\right]+C\left(\epsilon+\frac{1}{\epsilon}\right) \left(1+\|\xi\|_2^2\right).
\end{align*}
By choosing $\epsilon$ small enough such that $\epsilon  C(L,T)\le\frac{1}{2}$, we have
\begin{align*}
    \int_t^T \|v(s)\|_2^2 ds \le C\left(1+\|\xi\|_2^2\right), 
\end{align*}
which can be submitted back into \eqref{lem:2_7} and \eqref{lem:2_8} to obtain
\begin{equation}\label{lem:2_9}
\begin{aligned}
    \e\left[\sup_{t\le s\le T}\left|\left(X(s), P(s)\right)^\top\right|^2+\int_t^T |Q(s)|^2ds \right] \le C\left(1+\|\xi\|_2^2\right).
\end{aligned}
\end{equation}
From the third equation in FBSDEs \eqref{FB:mfg_generic}, we know that 
\begin{align*}
    &D_v f\left(s,X(s),\lr(X(s)),v(s)\right)=-D_v b\left(s,X(s),\lr(X(s)),v(s)\right)^\top P(s),
\end{align*}
and then,
\begin{align*}
    &[D_v f(s,X(s),\lr(X(s)),v(s))-D_v f(s,X(s),\lr(X(s)),0)]^\top v(s)\\
    =\ & -v(s)^\top D_v b(s,X(s),\lr(X(s)),v(s))^\top P(s) -D_v f(s,X(s),\lr(X(s)),0)^\top v(s).
\end{align*}
Therefore, from the convexity of $f$ in $v$ in Assumption (A3), we have
\begin{align*}
    2\lambda_v|v(s)|^2\le\ & |v(s)|\cdot |D_v b(s,X(s),\lr(X(s)),v(s))|\cdot |P(s)| + |D_v f(s,X(s),\lr(X(s)),0)|\cdot |v(s)|\\
    \le\ & L|v(s)| \left[1+|P(s)|+|X(s)|+W_2(\lr(X(s)),\delta_0)\right],
\end{align*}
and then, 
\begin{align}\label{lem:2_10}
    |v(s)|\le\ & \frac{L}{2\lambda_v}\left[1+|P(s)|+|X(s)|+W_2(\lr(X(s)),\delta_0)\right].
\end{align}
From \eqref{lem:2_9} and \eqref{lem:2_10}, we obtain \eqref{lem2_1}.
\end{proof}

In particular, in this proof, we emphasize the following claims: from the third equation of FBSDEs \eqref{FB:mfg_generic}, Condition \eqref{generic:condition:b'} and its \textit{Schur complement}, we know that 
\begin{align*}
    P_{t\xi}(s)=-\left((D_vb) (D_vb)^\top\right)^{-1} (D_vb) (D_v f)\left(s,X_{t\xi}(s),\lr(X_{t\xi}(s)),v_{t\xi}(s)\right), 
\end{align*}
by then, from Assumptiuon (A2), we can come up with the following estimate:
\begin{align}\label{lem:2_2(zw)}
    \left|P_{t\xi}(s)\right|\le \frac{L^2}{\lambda_b} \big[1+\left|X_{t\xi}(s)\right|+W_2(\lr(X_{t\xi}(s)),\delta_0)+\left|v_{t\xi}(s)\right|\big];
\end{align}
again, from the third equation of FBSDEs \eqref{FB:mfg_generic}, we know that 
\begin{align*}
    &\big[D_v f\left(s,X_{t\xi}(s),\lr(X_{t\xi}(s)),v_{t\xi}(s)\right)-D_v f(s,X_{t\xi}(s),\lr(X_{t\xi}(s)),0)\big]^\top v_{t\xi}(s)\\
    =\ & -v_{t\xi}(s)^\top D_v b\left(s,X_{t\xi}(s),\lr(X_{t\xi}(s)),v_{t\xi}(s)\right)^\top P_{t\xi}(s) -D_v f\left(s,X_{t\xi}(s),\lr(X_{t\xi}(s)),0\right)^\top v_{t\xi}(s),
\end{align*}
and therefore, from the convexity of $f$ in the argument $v$ in accordance with Assumption (A3), we have
\begin{align}\label{lem:2_10(zw)}
    \left|v_{t\xi}(s)\right|\le\ & \frac{L}{2\lambda_v}\big[1+|P_{t\xi}(s)|+|X_{t\xi}(s)|+W_2(\lr(X_{t\xi}(s)),\delta_0)\big].
\end{align}
We next show that the Lagrangian $L$ defined in \eqref{H'} has a minimizer in $v$ when the variables $(x,m,p)$ are in the following cone set with a suitable constant $K>0$:
\begin{align*}
    \mathcal{C}_K:=\left\{(x,m,p)\in\brn\times\pr_2(\brn)\times\brn: |p|\le K(1+|x|+W_2(m,\delta_0))\right\};
\end{align*}
we also call that $(x,m,p)$ satisfies a cone property with $K$ when $(x,m,p)\in c_K$. The concept of the ``cone property'' (or ``cone condition'') was proposed in \cite{AB10'',AB10'} in the study of the first order generic mean field problems. In view of \eqref{lem:2_2(zw)} and \eqref{lem:2_10(zw)}, the next proposition shows that the cone property is automatically satisfied under Assumptions (A1)-(A3) and the subsequent Condition \eqref{lem:2_12}; to motivate this last condition, recall that $\lambda_v$ corresponds to the convexity of $f$ in $v$, and $\lambda_b$ represents the non-degeneracy of $D_v b$ and henceforth the positive definiteness of the matrix $(D_v b)(D_v b)^\top$ (see \eqref{generic:condition:b'}), so Condition \eqref{lem:2_12} is valid whenever either $\lambda_v$ or $\lambda_b$ is large enough.

\begin{proposition}\label{prop:cone}
    Under Assumptions (A1)-(A3), suppose that
    \begin{align}\label{lem:2_12}
        &2\lambda_v> \frac{L^2}{\lambda_b}\left(L+L_b^2+\frac{LL_b^2}{2\lambda_v}\right).
    \end{align}
    Then, for any $s\in [0,T]$ and $(x,m,p)\in\mathcal{C}_K$ with $K:=\left(1-\frac{L^3}{2\lambda_v\lambda_b}\right)^{-1}\frac{L^2}{\lambda_b}\left(1+\frac{L}{2\lambda_v}\right)$, i.e.
    \begin{align}\label{cone_condition}
        |p|\le \left(1-\frac{L^3}{2\lambda_v\lambda_b}\right)^{-1}\frac{L^2}{\lambda_b}\left(1+\frac{L}{2\lambda_v}\right) \left[1+|x|+W_2(m,\delta_0)\right], 
    \end{align}
    the map $\brd\ni v\mapsto L(s,x,m,v,p)\in\br$ has a unique minimizer $\hv(s,x,m,p)$, and we have
    \begin{equation}\label{D_pH&D_xH}
    \begin{aligned}
        D_p H(t,x,m,p)=\ &  b\left(t,x,m,\hv(s,x,m,p)\right);\\
        D_x H(t,x,m,p)=\ & D_x b\left(t,x,m,\hv(s,x,m,p)\right)^\top p+ D_x f\left(t,x,m,\hv(s,x,m,p)\right). 
    \end{aligned}
    \end{equation}
    As a consequence, under conditions \eqref{lem:2_0} and \eqref{lem:2_12}, we have
    \begin{equation}\label{lem2_3}
        v_{t\xi}(s)=\hv\left(s,X_{t\xi}(s),\lr\left(X_{t\xi}(s)\right),P_{t\xi}(s) \right),\quad s\in[t,T].
    \end{equation}
    Therefore, FBSDEs \eqref{FB:mfg_generic} for $(X_{t\xi},P_{t\xi},Q_{t\xi})$  reads
    \begin{equation}\label{FB:mfg_generic'}
    \left\{
    \begin{aligned}
        &X_{t\xi}(s) = \xi+\int_t^s D_pH\left(r,X_{t\xi}(r),\lr(X_{t\xi}(r)),P_{t\xi}(r)\right)dr\\
        &\qquad\qquad + \int_t^s \left[\sigma_0\left(r,\lr(X_{t\xi}(r))\right)+\sigma_1(r)X_{t\xi}(r)\right]dB(r),\\[3mm]
        &P_{t\xi}(s) = \int_s^T \bigg[ D_x H\left(r,X_{t\xi}(r),\lr(X_{t\xi}(r)),P_{t\xi}(r)\right)+\sum_{j=1}^n \left(\sigma^j_1(r)\right)^\top Q_{t\xi}^j(r)\bigg]dr\\
        &\qquad\qquad -\int_s^T Q_{t\xi}(r) dB(r)+ D_x g\left(X_{t\xi}(T),\lr(X_{t\xi}(T))\right),\quad s\in[t,T].
    \end{aligned}
    \right.
    \end{equation}
\end{proposition}

\begin{proof}
From Assumptions (A1) and (A3) and Condition \eqref{lem:2_12}, for any $(s,x,\mu)$ and any $p$ satisfying the cone property \eqref{cone_condition}, we can compute that
\begin{align*}
    &D_v^2 L\left(s,x,m,v,p\right)\\
    =\ & D_v^2 b\left(s,x,m,v\right)^\top p+D_v^2 f\left(s,x,m,v\right)\\
    \geq\ & -\frac{L_b^2}{1+|x|+|v|+W_2(m,\delta_0)} \left(1-\frac{L^3}{2\lambda_v\lambda_b}\right)^{-1}\frac{L^2}{\lambda_b}\left(1+\frac{L}{2\lambda_v}\right) \left[1+|x|+W_2(m,\delta_0)\right]+2\lambda_v \\
    \geq\ &-\left(1-\frac{L^3}{2\lambda_v\lambda_b}\right)^{-1}\frac{L^2L_b^2}{\lambda_b}\left(1+\frac{L}{2\lambda_v}\right)+2\lambda_v>0,
\end{align*}
from which we see that the map $\brd\ni v\mapsto L\left(s,x,m,v,p \right)$ is strictly convex, and there exists a unique minimizer $\hv(s,x,m,p)$. Moreover, for any $\epsilon>0$ and any $(x',m',p')$ such that $|x'-x|,\ W_2(m,m'),\ |p'-p|\le\epsilon$, we can also derive that
\begin{align*}
   & D_v^2 L\left(s,x',m',v,p' \right) \geq\  -\frac{L_b^2|p'|}{1+|x'|+|v'|+W_2(m',\delta_0)} +2\lambda_v \\
    \geq\ & -\frac{L_b^2|p|}{1+|x'|+|v'|+W_2(m',\delta_0)}-\frac{L_b^2\epsilon}{1+|x'|+|v'|+W_2(m',\delta_0)} +2\lambda_v \\
    \geq\ & -\frac{L_b^2}{1+|x'|+|v'|+W_2(m',\delta_0)}\left(1-\frac{L^3}{2\lambda_v\lambda_b}\right)^{-1}\frac{L^2}{\lambda_b}\left(1+\frac{L}{2\lambda_v}\right) \left[1+|x|+W_2(m,\delta_0)\right]\\
    &-\frac{L_b^2\epsilon}{1+|x'|+|v'|+W_2(m',\delta_0)} +2\lambda_v \\
    \geq\ & -\left(1-\frac{L^3}{2\lambda_v\lambda_b}\right)^{-1}\frac{L^2 L_b^2}{\lambda_b}\left(1+\frac{L}{2\lambda_v}\right) +2\lambda_v  \\
    &-\frac{L_b^2 \epsilon}{1+|x'|+|v'|+W_2(m',\delta_0)} \left[1+2\left(1-\frac{L^3}{2\lambda_v\lambda_b}\right)^{-1}\frac{L^2}{\lambda_b}\left(1+\frac{L}{2\lambda_v}\right)\right].
\end{align*}
From Condition \eqref{lem:2_12}, whenever $\epsilon$ is small enough, the right hand side of the inequality remains negative. Therefore, the minimizing map $\hv$ is well-defined in the neighbourhood of $(x,m,p)$. Then, from Assumptions (A1) and (A2) and the first order condition 
\begin{align*}
    D_v L\left(s, x,m,\hv(s,x,m,p),p\right)=0, 
\end{align*}
we obtain \eqref{D_pH&D_xH}. We next try to establish \eqref{lem2_3}. From \eqref{lem:2_2(zw)} and \eqref{lem:2_10(zw)}, we know that
\begin{align*}
    \left|P_{t\xi}(s)\right|\le\ & \frac{L^2}{\lambda_b} \left[1+\left|X_{t\xi}(s)\right|+W_2\left(\lr(X_{t\xi}(s)),\delta_0\right)\right]+\frac{L^2}{\lambda_b}\left|v_{t\xi}(s)\right|\\[3mm]
    \le\ &  \frac{L^2}{\lambda_b} \left[1+\left|X_{t\xi}(s)\right|+W_2\left(\lr(X_{t\xi}(s)),\delta_0\right)\right]\\
    &+\frac{L^3}{2\lambda_v\lambda_b}\left[1+\left|P_{t\xi}(s)\right|+\left|X_{t\xi}(s)\right|+W_2\left(\lr(X_{t\xi}(s)),\delta_0\right)\right].
\end{align*}
Since $\frac{L^3}{2\lambda_v\lambda_b}<1$, we further see that
\begin{align}\label{lem:2_11}
    \left|P_{t\xi}(s)\right|\le \left(1-\frac{L^3}{2\lambda_v\lambda_b}\right)^{-1}\frac{L^2}{\lambda_b}\left(1+\frac{L}{2\lambda_v}\right)\left[1+\left|X_{t\xi}(s)\right|+W_2\left(\lr(X_{t\xi}(s)),\delta_0\right)\right],
\end{align}
that is, $P(s)$ satisfies the cone property $\mathcal{C}_K$ with $K=\left(1-\frac{L^3}{2\lambda_v\lambda_b}\right)^{-1}\frac{L^2}{\lambda_b}\left(1+\frac{L}{2\lambda_v}\right)$. And since $v_{t\xi}(s)$ satisfies 
\begin{align*}
    D_v L\left(s,X_{t\xi}(s),\lr\left(X_{t\xi}(s)\right),v,P_{t\xi}(s) \right)=0,
\end{align*}
we also deduce \eqref{lem2_3}; and \eqref{FB:mfg_generic'} is a direct consequence of \eqref{D_pH&D_xH}.
\end{proof}

In the rest of this article, we use $C(\alpha_1,\dots, \alpha_k)$ to denote a constant depending only on parameters$(\alpha_1,\dots, \alpha_k)$; otherwise, we always use an ordinary $C$ to denote by a constant depending all on $(L,T,\lambda_b,\lambda_v,\lambda_x,l_b^m, l_\sigma^m,L_b^0,L_b^1,L_b^2,L_f^0,L_f^1)$. As long as $\lr\left(X_{t\xi}(s)\right)$ for $t\le s\le T$ given, Problem \eqref{intro_1'} is certainly a classical stochastic control problem, meanwhile $X_{t\xi}$ is a given process. Problem \eqref{intro_1'} is associated with the following FBSDEs arising from maximum principle: \ for $s\in[t,T],$
\begin{equation}\label{intro_3}
\left\{	
	\begin{aligned}
		&X_{tx\mu}(s)=x+\int_t^s b\left(r,X_{tx\mu}(r),\lr(X_{t\xi}(r)),v_{tx\mu}(r)\right)dr\\
		&\qquad\qquad \qquad+\int_t^s \left[ \sigma_0(r,\lr(X_{t\xi}(s)))+\sigma_1(r)X_{tx\mu}(r)\right]dB(r),\\
		&P_{tx\mu}(s)=-\int_s^T Q_{tx\mu}(r)dB(r)+D_x g\left(X_{tx\mu}(T),\lr(X_{t\xi}(T))\right)\\
        &\quad\qquad \qquad+\int_s^T \bigg[D_x b(r,X_{tx\mu}(r),\lr(X_{t\xi}(r)),v_{tx\mu}(r))^\top P_{tx\mu}(r)+\sum_{j=1}^n\left(\sigma_1^j(r)\right)^\top Q_{tx\mu}^j(r) \\
        &\quad\qquad\qquad\qquad \quad+D_x f \left(r,X_{tx\mu}(r),\lr(X_{t\xi}(r)),v_{tx\mu}(r)\right)\bigg]dr,\\
        &D_v b\left(s,X_{tx\mu}(s),\lr(X_{t\xi}(s)),v_{tx\mu}(s)\right)^\top P_{tx\mu}(s)+ D_v f\left(s,X_{tx\mu}(s),\lr(X_{t\xi}(s)),v_{tx\mu}(s)\right)=0.
	\end{aligned}
\right.
\end{equation}
Here, we use the subscript $tx\mu$ to denote the dependence of the system \eqref{intro_3} on the initial condition; and since the system depends on $\xi$ only through its law $\mu$, so it is reasonable to use the subscript $\mu$ instead of $\xi$. From the well-known sufficiency maximum principle \cite{Fleming} for stochastic control problems, we know that, if FBSDEs \eqref{intro_3} has a solution $\left(X_{tx\mu},P_{tx\mu},Q_{tx\mu},v_{tx\mu}\right)\in \sr^2_\f(t,T)\times\sr^2_{\f}(t,T)\times\left(\lr^2_{\f}(t,T)\right)^n\times\lr^2_{\f}(t,T)$, then, $v_{tx\mu}$ is the unique solution of Problem \eqref{intro_1'}. We next state the well-posedness of FBSDEs \eqref{intro_3}, and also the $L^2$-boundedness and continuity of the solution with respect to spatial variable $x$ and measure argument $\mu$.

\begin{theorem}\label{lem:5}
	Under Assumptions (A1)-(A3) and validity of \eqref{lem:2_0}, there is a unique adapted solution $(X_{tx\mu},P_{tx\mu},Q_{tx\mu},v_{tx\mu})$ of FBSDEs \eqref{intro_3}, and $v_{tx\mu}$ is the unique optimal control of Problem \eqref{intro_1'}. For $x,x'\in\brn$ and $\mu,\mu'\in\pr_2(\brn)$, we have
	\begin{align}
		&\e\bigg[\sup_{t\le s\le T}\left|\left(X_{tx\mu}(s), P_{tx\mu}(s),v_{tx\mu}(s)\right)^\top\right|^2+ \int_t^T \left|Q_{tx\mu}(s)\right|^2ds\bigg]\le C\left(1+|x|^2+W^2_2(\mu,\delta_0)\right), \label{lem5_1}\\[3mm]
		&\e\bigg[\sup_{t\le s\le T}\left|\left(X_{tx'\mu'}(s)-X_{tx\mu}(s), P_{tx'\mu'}(s)-P_{tx\mu}(s),v_{tx'\mu'}(s)-v_{tx\mu}(s)\right)^\top\right|^2 \notag \\
        &\ \quad\qquad\qquad\qquad\qquad\qquad + \int_t^T \left|Q_{tx'\mu'} (s) -Q_{tx\mu}(s)\right|^2 ds\bigg] \le C\left(|x'-x|^2+W_2^2\left(\mu,\mu'\right)\right).\label{lem5_2}
	\end{align}
    Moreover, if \eqref{lem:2_12} holds, then we have 
    \begin{equation}\label{lem5_3}
        v_{tx\mu}(s)=\hv\left(s,X_{tx\mu}(s),\lr\left(X_{t\xi}(s)\right),P_{tx\mu}(s) \right),\quad s\in[t,T], 
    \end{equation}
    henceforth, FBSDEs \eqref{intro_3} for $(X_{tx\mu},P_{tx\mu},Q_{tx\mu})$ also reads
    \begin{equation}\label{intro_3'}
    \left\{
    \begin{aligned}
        &X_{tx\mu}(s) = x+\int_t^s D_pH\left(r,X_{tx\mu}(r),\lr(X_{t\xi}(r)),P_{tx\mu}(r)\right)dr\\
        &\qquad\qquad + \int_t^s \left[\sigma_0\left(r,\lr(X_{t\xi}(r))\right)+\sigma_1(r)X_{tx\mu}(r)\right]dB(r),\\
        &P_{tx\mu}(s) = \int_s^T \bigg[ D_x H\left(r,X_{tx\mu}(r),\lr(X_{t\xi}(r)),P_{tx\mu}(r)\right)+\sum_{j=1}^n \left(\sigma^j_1(r)\right)^\top Q_{tx\mu}^j(r)\bigg]dr\\
        &\qquad\qquad -\int_s^T Q_{tx\mu}(r) dB(r)+ D_x g\left(X_{tx\mu}(T),\lr(X_{t\xi}(T))\right),\quad s\in[t,T].
    \end{aligned}
    \right.
    \end{equation}
\end{theorem}

The proof of Theorem~\ref{lem:5} is similar to that for Theorem~\ref{lem:2}, and it is put in Appendix~\ref{pf:lem:5}. Similar to \eqref{lem:2_2(zw)}, among the derivations of the proof of Theorem~\ref{lem:5}, we emphasize the following claims: from the third equation of FBSDEs \eqref{intro_3}, Condition \eqref{generic:condition:b'} and Assumption (A2) altogether yield that:
\begin{align}\label{zw:1}
    \left|P_{tx\mu}(s)\right|\le \frac{L^2}{\lambda_b} \big[1+\left|X_{tx\mu}(s)\right|+W_2\left(\lr(X_{t\xi}(s)),\delta_0\right)+\left|v_{tx\mu}(s)\right|\big];
\end{align}
also, the third equation of FBSDEs \eqref{intro_3} and the convexity of $f$ in $v$ proposed in Assumption (A3) can imply
\begin{align}\label{zw:1'}
    |v_{tx\mu}(s)|\le\ & \frac{L}{2\lambda_v}\left[1+|P_{tx\mu}(s)|+|X_{tx\mu}(s)|+ W_2\left(\lr(X_{t\xi}(s)),\delta_0\right)\right].
\end{align}
Then, Condition \eqref{lem:2_12}, the last two inequalities \eqref{zw:1} and \eqref{zw:1'} yield the following cone property $\mathcal{C}_K$ with $K=\left(1-\frac{L^3}{2\lambda_v\lambda_b}\right)^{-1}\frac{L^2}{\lambda_b}\left(1+\frac{L}{2\lambda_v}\right)$:
\begin{align}\label{cone_property_SC}
    |P_{tx\mu}(s)|\le \left(1-\frac{L^3}{2\lambda_v\lambda_b}\right)^{-1}\frac{L^2}{\lambda_b}\left(1+\frac{L}{2\lambda_v}\right)\left[1+|X_{tx\mu}(s)|+W_2\left(\lr(X_{t\xi}(s)),\delta_0\right)\right],
\end{align}
which further deduces \eqref{lem5_3} in view of Proposition~\ref{prop:cone}. And the formulation of FBSDEs \eqref{intro_3'} is a consequence of FBSDEs \eqref{intro_3} and the relations \eqref{D_pH&D_xH}.

\begin{remark}
    We here illustrate more about the relations among the FBSDEs \eqref{FB:mfg_generic}, \eqref{FB:mfg_generic'}, \eqref{intro_3} and \eqref{intro_3'}. The system \eqref{FB:mfg_generic} is equivalent to the system \eqref{FB:mfg_generic'}, since that the cone property $\mathcal{C}_K$ \eqref{lem:2_11} is satisfied, and so relations \eqref{D_pH&D_xH} hold. In a similar manner, the system \eqref{intro_3} is equivalent to the system \eqref{intro_3'}. Moreover, from the uniqueness result of FBSDEs \eqref{FB:mfg_generic}, when we set $x=\xi$ in the system \eqref{intro_3}, then it is equivalent to the system \eqref{FB:mfg_generic}; similar discussion can also be found in \cite[Section 3]{BR}. 
\end{remark}

Up to now, Theorem~\ref{lem:2} warrants the solvability of MFG \eqref{intro_1}, and Theorem~\ref{lem:5} ensures the solvability of Problem \eqref{intro_1'}. Our results can also be readily extended to MFG with a common noise $\{B^0_s,0\le s\le T\}$ in a similar manner. The major difference is that the distribution $m(s)=\lr(X_{t\xi}(s))$ of the equilibrium state in MFG \eqref{intro_1} should be be replaced by the conditional distribution $m(s)=\lr\left(X_{t\xi}(s)|\f^0_s\right)$, where the filtration $\f^0_\cdot$ is generated by the common noise $B^0$; and the modified approach can even allow us to tackle the controlled SDE with the common noise term $\sigma^0(s,m(s))dB^0(s)$. In such a case, the system of FBSDEs \eqref{FB:mfg_generic} becomes a conditional distribution dependent FBSDEs for $(X_{t\xi},P_{t\xi},Q_{t\xi},Q^0_{t\xi},v_{t\xi})$, where $Q^0_{t\xi}(s)$ is the coefficient of the $dB^0(s)$ term of the backward equation. More related discussion can also be found in \cite{AB13} particularly for mean field type control problem with a common noise. 

\begin{remark}
    Our methodology can also be applied to second order mean field type control (MFTC) problems with nonlinear drifts and degenerate diffusions. The system of FBSDEs corresponding to a MFTC problem is different from that to MFG (see \cite{AB11} for instance): since the system also depends on the law of the current state process, by then, there will be some extra terms in the FBSDEs corresponding to a MFTC problem than FBSDEs \eqref{FB:mfg_generic} for MFG \eqref{intro_1}. To handle these extra terms, in Assumption \eqref{generic:condition:b}, we also need the following boundedness condition for the derivative $D_{y'}D_y\frac{d^2 b}{d\nu^2}$:
    \begin{align*}
        \left| D_{y'}D_y\frac{d^2 b}{d\nu^2}(s,x,m,v)(y,y')\right|\le \frac{L_b^0}{1+|x|+|v|+W_2(m,\delta_0)}.
    \end{align*}
    Besides, since the MFTC problem is a control problem (but not with an additional fixed point problem), we shall not require any monotonicity conditions. For the MFTC problem, we can also have the cone property as in Proposition~\ref{prop:cone}, which together with our proposed stochastic control approach can deduce the well-posedness of the associated FBSDEs. We shall study this topic in an alternative future work.
\end{remark}

From now on, we focus on  the regularity of the value function $V$ defined in \eqref{intro_4} and then the solubility of the master equation \eqref{intro_5}. From Theorem~\ref{lem:5}, $V$ defined in \eqref{intro_4} satisfies
\begin{equation}\label{intro_4'}
	\begin{split}
		V(t,x,\mu)=\e\left[\int_t^T f\left(s,X_{tx\mu}(s),\lr(X_{t\xi}(s)),v_{tx\mu}(s)\right)ds+g\left(X_{tx\mu}(T),\lr(X_{t\xi}(T))\right)\right].
	\end{split}	
\end{equation}
In the following sections, we shall calculate the G\^ateaux derivative of $\left(X_{t\xi}(s),P_{t\xi}(s),Q_{t\xi}(s),v_{t\xi}(s)\right)$ in $\xi$, and that of  $\left(X_{tx\mu}(s),P_{tx\mu}(s),Q_{tx\mu}(s),v_{tx\mu}(s)\right)$ in $(x,\mu)$, which will be used to derive  the classical regularity of the value functional $V$.

\section{G\^ateaux derivatives of the processes}\label{sec:distribution}

We here give the first order G\^ateaux derivatives of the solutions of the two systems of FBSDEs \eqref{FB:mfg_generic} and \eqref{intro_3}. 

\subsection{Jacobian flows of FBSDEs \eqref{FB:mfg_generic}}

We here give the G\^ateaux differentiability of $\left(X_{t\xi}(s),P_{t\xi}(s),Q_{t\xi}(s),v_{t\xi}(s)\right)$ in the initial $\xi\in L_{\f_t}^2$. For the sake of convenience, in the rest of this article, we denote by $\theta:=(x,m,v)\in \brn\times\pr_2(\brn)\times\brd$ and $\Theta:=(x,p,q,v)\in \brn\times\brn\times\brn\times\brd$. We denote by $\theta_{t\xi}(s)$ the process $(X_{t\xi}(s),\lr(X_{t\xi}(s)),v_{t\xi}(s))$ and $\Theta_{t\xi}(s)$ the process $(X_{t\xi}(s),P_{t\xi}(s),Q_{t\xi}(s),v_{t\xi}(s))$. In view of Assumption (A2) and (A3), for any $x,y\in\brn$, $v,u\in\brd$ and $(s,m)\in[0,T]\times\pr_2(\brn)$,
\begin{equation}\label{convex'}
	\begin{split}
		\left[\left(
		\begin{array}{cc}
			D_v^2 f & D_vD_x f\\
			D_xD_v f & D_x^2 f
		\end{array}
		\right)(s,x,m,v) \right] \left(
		\begin{array}{cc}
			u\\
			y
		\end{array}
		\right)^{\otimes 2} & \geq  2\lambda_v |u|^2+2\lambda_x |y|^2,\\
		\text{and}\qquad  D_x^2 g(x,m)y^{\otimes 2} &\geq 2\lambda_g |y|^2.
	\end{split}
\end{equation}
Given the process $\Theta_{t\xi}$, for any $\eta\in L_{\f_t}^2$, consider the following FBSDEs for the indetermined process $\left(\mathcal{D}_\eta X_{t\xi},\mathcal{D}_\eta P_{t\xi},\mathcal{D}_\eta Q_{t\xi},\mathcal{D}_\eta v_{t\xi}\right)\in \left(L^2(t,T;L^2(\Omega,\f,\mathbb{P};\brn))\right)^4$:
\begin{align}
		\mathcal{D}_\eta X_{t\xi}(s)=\ & \eta+\int_t^s \bigg\{D_x b(r, \theta_{t\xi}(r))\mathcal{D}_\eta X_{t\xi}(r) +D_v b(r, \theta_{t\xi}(r))\mathcal{D}_\eta v_{t\xi}(r) \notag \\
        &\qquad\qquad +\widetilde{\e}\bigg[ D_y \frac{db}{d\nu} (r,\theta_{t\xi}(r))\left(\widetilde{X_{t\xi}}(r)\right) \widetilde{\mathcal{D}_\eta X_{t\xi}}(r)\bigg] \bigg\}dr \notag \\
		&+\int_t^s \widetilde{\e}\bigg[ D_y \frac{d\sigma_0}{d\nu} (r,\lr(X_{t\xi}(r)))\left(\widetilde{X_{t\xi}}(r)\right) \widetilde{\mathcal{D}_\eta X_{t\xi}}(r)\bigg]+  \sigma_1(r)\mathcal{D}_\eta X_{t\xi}(r) dB(r), \notag \\
		\mathcal{D}_\eta P_{t\xi}(s)=\ & D_x^2 g (X_{t\xi}(T),\lr(X_{t\xi}(T)))^\top  \mathcal{D}_\eta X_{t\xi}(T) \notag \\
        &+\widetilde{\e}\left[\left(D_y \frac{d}{d\nu}D_x g (X_{t\xi}(T),\lr(X_{t\xi}(T)))\left(\widetilde{X_{t\xi}}(T)\right) \right)^\top  \widetilde{\mathcal{D}_\eta X_{t\xi}}(T)\right] \notag \\
		&+\int_s^T\Bigg\{ D_xb(r,\theta_{t\xi}(r))^\top \mathcal{D}_\eta P_{t\xi}(r)+\sum_{j=1}^n\left(\sigma_1^j (r)\right)^\top \mathcal{D}_\eta Q_{t\xi}^j(r) \notag \\
        &\quad\qquad + \left(P_{t\xi}(r)\right)^\top \bigg\{D_x^2 b(r,\theta_{t\xi}(r)) \mathcal{D}_\eta X_{t\xi}(r)+\widetilde{\e}\left[ D_y\frac{d}{d\nu}D_x b(r,\theta_{t\xi}(r))\left(\widetilde{X_{t\xi}}(r)\right) \widetilde{\mathcal{D}_\eta X_{t\xi}}(r)\right] \notag \\
        &\quad\qquad\qquad\qquad\qquad + D_vD_x b(r,\theta_{t\xi}(r)) \mathcal{D}_\eta v_{t\xi}(r)\bigg\}   \notag\\
        &\quad\qquad +D_x^2 f  (r,\theta_{t\xi}(r))^\top \mathcal{D}_\eta X_{t\xi}(r) +\widetilde{\e}\left[\left(D_y \frac{d}{d\nu}D_x f  (r,\theta_{t\xi}(r))\left(\widetilde{X_{t\xi}}(r)\right)\right)^\top \widetilde{\mathcal{D}_\eta X_{t\xi}}(r)\right] \notag \\
        &\quad\qquad + D_vD_x f (r,\theta_{t\xi}(r))^\top  \mathcal{D}_\eta v_{t\xi}(r)\Bigg\}dr -\int_s^T\mathcal{D}_\eta Q_{t\xi}(r)dB(r), \quad s\in[t,T], \label{FB:dr}
\end{align}
with the following condition
\begin{align}
    0=\ & D_x D_v f  (s,\theta_{t\xi}(s))^\top \mathcal{D}_\eta X_{t\xi}(s) +\widetilde{\e}\left[\left(D_y \frac{d}{d\nu}D_v f  (s,\theta_{t\xi}(s))\left(\widetilde{X_{t\xi}}(s)\right)\right)^\top \widetilde{\mathcal{D}_\eta X_{t\xi}}(s)\right] \notag \\
    &+D_v^2 f (s,\theta_{t\xi}(s))^\top  \mathcal{D}_\eta v_{t\xi}(s)+D_v b(s,\theta_{t\xi}(s))^\top \mathcal{D}_\eta P_{t\xi}(s) \notag\\
    &+ \left(P_{t\xi}(s)\right)^\top \bigg\{D_x D_v b(s,\theta_{t\xi}(s)) \mathcal{D}_\eta X_{t\xi}(s)+ D_v^2 b(s,\theta_{t\xi}(s)) \mathcal{D}_\eta v_{t\xi}(s) \notag \\
    &\qquad\qquad\qquad +\widetilde{\e}\left[ D_y\frac{d}{d\nu}D_v b(s,\theta_{t\xi}(s))\left(\widetilde{X_{t\xi}}(s) \right) \widetilde{\mathcal{D}_\eta X_{t\xi}}(s)\right]\bigg\}, \label{FB:dr_condition}
\end{align}
where $\left(\widetilde{X_{t\xi}}(s),\widetilde{\mathcal{D}_\eta X_{t\xi}}(s)\right)$ is an independent copy of $\left(X_{t\xi}(s), \mathcal{D}_\eta X_{t\xi}(s)\right)$. We next give the well-posedness of FBSDEs \eqref{FB:dr}-\eqref{FB:dr_condition}, and also the regularity of the solution.

\begin{lemma}\label{lem:3}
	Under Assumptions (A1)-(A3) and the validity of \eqref{lem:2_0}, there is a unique adapted solution $\left(\mathcal{D}_\eta X_{t\xi},\mathcal{D}_\eta P_{t\xi},\mathcal{D}_\eta Q_{t\xi},\mathcal{D}_\eta v_{t\xi}\right)$ of FBSDEs \eqref{FB:dr}-\eqref{FB:dr_condition}, such that
	\begin{align}\label{lem3_1}
		\e\left[\sup_{t\le s\le T}\left|\left(\mathcal{D}_\eta X_{t\xi}(s),\mathcal{D}_\eta P_{t\xi}(s),\mathcal{D}_\eta v_{t\xi}(s)\right)^\top\right|^2+ \int_t^T\left|\mathcal{D}_\eta Q_{t\xi}(s)\right|^2 ds\right]\le C\|\eta\|_2^2.
	\end{align}
\end{lemma}

\begin{proof}
Note the fact that $\Theta_{t\xi}$ is already known process, and \eqref{FB:dr}-\eqref{FB:dr_condition} is a system of FBSDEs with official solution $\left(\mathcal{D}_\eta X_{t\xi},\mathcal{D}_\eta P_{t\xi},\mathcal{D}_\eta Q_{t\xi},\mathcal{D}_\eta v_{t\xi}\right)$. Therefore, the well-posedness of FBSDEs \eqref{FB:dr} is given by \cite[Lemma 2.1]{AB12} (about FBSDEs in Hilbert spaces under $\beta$-monotonicity), where the corresponding $\beta$-monotonicity is fulfilled with the conditions \eqref{convex'} and \eqref{FB:dr_condition}. The present proof is similar to that for \cite[Theorem 5.2]{AB12}, and we also refer to \cite[Lemma 3.1]{AB11} with a similar proof. Here, we only give a proof of \eqref{lem3_1}. In this proof, for notational convenience, we use the notations $(\mathcal{X},\mathcal{P},\mathcal{Q},\mathcal{V})$ to denote the process $\left(\mathcal{D}_\eta X_{t\xi},\mathcal{D}_\eta P_{t\xi},\mathcal{D}_\eta Q_{t\xi},\mathcal{D}_\eta v_{t\xi}\right)$. By applying It\^o's formula to $\mathcal{P}(s)^\top \mathcal{X}(s)$, we have
\begin{align*}
    &\e\Bigg\{D_x^2 g (X_{t\xi}(T),\lr(X_{t\xi}(T))) \left(  \mathcal{X}(T)\right)^{\otimes 2} \\
    &\quad +\widetilde{\e}\left[\left(\widetilde{\mathcal{X}}(T)\right)^\top\left(D_y \frac{d}{d\nu}D_x g (X_{t\xi}(T),\lr(X_{t\xi}(T)))\left(\widetilde{X_{t\xi}}(T)\right) \right) \right]\mathcal{X}(T)-\mathcal{P}(t)^\top \eta \Bigg\}\\
    =\ & \e \int_t^T \Bigg\{ \mathcal{P}(s)^\top \widetilde{\e}\bigg[ D_y \frac{db}{d\nu} (s,\theta_{t\xi}(s))\left(\widetilde{X_{t\xi}}(s)\right) \widetilde{\mathcal{X}}(s)\bigg]+ \mathcal{P}(s)^\top D_v b(s, \theta_{t\xi}(s))\mathcal{V}(s)\\
    &\ \quad\qquad -\sum_{j=1}^n \mathcal{Q}^j(s)^\top  \widetilde{\e}\bigg[ D_y \frac{d\sigma^j}{d\nu} (s,\lr(X_{t\xi}(s)))\left(\widetilde{X_{t\xi}}(s)\right) \widetilde{\mathcal{X}}(r)\bigg] \\
    &\ \quad\qquad -P_{t\xi}(s)^\top \bigg[D_x^2 b(s,\theta_{t\xi}(s)) \mathcal{X}(s)+\widetilde{\e}\left[ D_y\frac{d}{d\nu}D_x b(s,\theta_{t\xi}(s))\left(\widetilde{X_{t\xi}}(s) \right)\widetilde{\mathcal{X}}(s)\right]\\
    &\ \qquad\qquad\qquad\qquad + D_vD_x b(s,\theta_{t\xi}(s)) \mathcal{V}(s)\bigg]  \mathcal{X}(s) \\
    &\ \quad\qquad - \mathcal{X}(s)^\top D_x^2 f  (s,\theta_{t\xi}(s))^\top \mathcal{X}(s) -\mathcal{X}(s)^\top \widetilde{\e}\left[\left(D_y \frac{d}{d\nu}D_x f (s,\theta_{t\xi}(s))\left(\widetilde{X_{t\xi}}(s)\right)\right)^\top \widetilde{\mathcal{X}}(s)\right]  \\
    &\ \quad\qquad -\mathcal{X}(s)^\top D_vD_x f (s,\theta_{t\xi}(s))^\top  \mathcal{V}(s)\Bigg\}ds.
\end{align*}
Substituting \eqref{FB:dr_condition} into this last equality, we have
\begin{align}
    &\e\Bigg\{D_x^2 g (X_{t\xi}(T),\lr(X_{t\xi}(T))) \left(  \mathcal{X}(T)\right)^{\otimes 2} \notag \\
    &\quad +\widetilde{\e}\left[\left(\widetilde{\mathcal{X}}(T)\right)^\top\left(D_y \frac{d}{d\nu}D_x g (X_{t\xi}(T),\lr(X_{t\xi}(T)))\left(\widetilde{X_{t\xi}}(T)\right) \right) \right]\mathcal{X}(T)-\mathcal{P}(t)^\top \eta \Bigg\}\notag \\
    =\ & \e \int_t^T \Bigg\{ \mathcal{P}(s)^\top \widetilde{\e}\bigg[ D_y \frac{db}{d\nu} (s,\theta_{t\xi}(s))\left(\widetilde{X_{t\xi}}(s)\right) \widetilde{\mathcal{X}}(s)\bigg]\notag \\
    &\quad\quad\qquad -\sum_{j=1}^n \mathcal{Q}^j(s)^\top  \widetilde{\e}\bigg[ D_y \frac{d\sigma^j}{d\nu} (s,\lr(X_{t\xi}(s)))\left(\widetilde{X_{t\xi}}(s)\right) \widetilde{\mathcal{X}}(s)\bigg] \notag \\
    &\quad\quad\qquad - P_{t\xi}(s)^\top \bigg\{ \left[\begin{pmatrix}
        D_v^2 b & D_vD_xb\\ D_xD_v b & D_x^2b
    \end{pmatrix}(s,\theta_{t\xi}(s))\right]\begin{pmatrix}
        \mathcal{V}(s)\\ \mathcal{X}(s)
    \end{pmatrix}^{\otimes 2} \notag \\
    &\quad\qquad\qquad\qquad\qquad +\widetilde{\e}\left[ D_y\frac{d}{d\nu}D_v b(s,\theta_{t\xi}(s))\left(\widetilde{X_{t\xi}}(s)\right) \widetilde{\mathcal{X}}(s)\right] \mathcal{V}(s)  \notag \\
    &\quad\qquad\qquad\qquad\qquad +\widetilde{\e}\left[ D_y\frac{d}{d\nu}D_x b(s,\theta_{t\xi}(s))\left(\widetilde{X_{t\xi}}(s)\right) \widetilde{\mathcal{X}}(s)\right]  \mathcal{X}(s) \bigg\} \notag \\
    &\quad\quad\qquad -\left[\begin{pmatrix}
        D_v^2 f & D_vD_xf\\ D_xD_v f & D_x^2f
    \end{pmatrix}(s,\theta_{t\xi}(s))\right]\begin{pmatrix}
        \mathcal{V}(s)\\ \mathcal{X}(s)
    \end{pmatrix}^{\otimes 2} \notag \\
    &\quad\quad\qquad -\mathcal{X}(s)^\top \widetilde{\e}\left[\left(D_y \frac{d}{d\nu}D_x f  (s,\theta_{t\xi}(s))\left(\widetilde{X_{t\xi}}(s)\right)\right)^\top \widetilde{\mathcal{X}}(s)\right] \notag  \\
    &\quad\quad\qquad -\mathcal{V}(s)^\top \widetilde{\e}\left[\left(D_y \frac{d}{d\nu}D_v f  (s,\theta_{t\xi}(s))\left(\widetilde{X_{t\xi}}(s)\right)\right)^\top \widetilde{\mathcal{X}}(s)\right] \Bigg\}ds. \label{lem:3_1}
\end{align}
From Condition \eqref{FB:dr_condition} and Assumption (A1), we know that
\begin{align}
    \mathcal{P}(s)= -&\left[\left((D_v b)(D_v b)^\top\right)^{-1}(D_vb) (s,\theta_{t\xi}(s))\right]  \notag \\
    &\cdot \Bigg\{ D_x D_v f (s,\theta_{t\xi}(s))^\top \mathcal{X}(s)+\widetilde{\e}\left[\left(D_y \frac{d}{d\nu}D_v f  (s,\theta_{t\xi}(s))\left(\widetilde{X_{t\xi}}(s)\right)\right)^\top \widetilde{\mathcal{X}}(s)\right] \notag \\
    &\qquad +D_v^2 f (s,\theta_{t\xi}(s))^\top  \mathcal{V}(s) \notag\\
    &\qquad + \left(P_{t\xi}(s)\right)^\top \bigg\{D_x D_v b(s,\theta_{t\xi}(s)) \mathcal{X}(s)+\widetilde{\e}\left[ D_y\frac{d}{d\nu}D_v b(s,\theta_{t\xi}(s))\left(\widetilde{X_{t\xi}}(s) \right) \widetilde{\mathcal{X}}(s)\right] \notag \\
    &\qquad\qquad\qquad\qquad + D_v^2 b(s,\theta_{t\xi}(s)) \mathcal{V}(s)\bigg\} \Bigg\}, \label{lem:3_2'}
\end{align}
and then, from the assumptions (A2), \eqref{generic:condition:b}, \eqref{generic:condition:b'} and the estimate \eqref{lem:2_2}, we have
\begin{align}\label{lem:3_2}
    |\mathcal{P}(s)|\le &\frac{L^2}{\lambda_b} \left(1+\frac{L^2}{\lambda_b}\right) \left(|\mathcal{X}(s)|+ \left\|\widetilde{\mathcal{X}}(s)\right\|_2 +|\mathcal{V}(s)|\right).
\end{align}
Then, from \eqref{lem:3_2} and Assumption (A1), we know that
\begin{align}
    \left|\mathcal{P}(s)^\top \widetilde{\e}\bigg[ D_y \frac{db}{d\nu} (s,\theta_{t\xi}(s))\left(\widetilde{X_{t\xi}}(s)\right) \widetilde{\mathcal{X}}(s)\bigg]\right| \le\ & \frac{L^2l_b^m}{\lambda_b} \left(1+\frac{L^2}{\lambda_b}\right) \left(|\mathcal{X}(s)|+ \left\|\widetilde{\mathcal{X}}(s)\right\|_2 +|\mathcal{V}(s)|\right)\left\|\widetilde{\mathcal{X}}(s)\right\|_2. \label{lem:3_3}
\end{align}
From the estimate \eqref{lem:2_2} and Assumption \eqref{generic:condition:b}, we know that 
\begin{align}
    &\Bigg| P_{t\xi}(s)^\top \bigg\{ \left[\begin{pmatrix}
        D_v^2 b & D_vD_xb\\ D_xD_v b & D_x^2b
    \end{pmatrix}(s,\theta_{t\xi}(s))\right]\begin{pmatrix}
        \mathcal{V}(s)\\ \mathcal{X}(s)
    \end{pmatrix}^{\otimes 2} \notag \\
    &\ \qquad\qquad +\widetilde{\e}\left[ D_y\frac{d}{d\nu}D_v b(s,\theta_{t\xi}(s))\left(\widetilde{X_{t\xi}}(s)\right) \widetilde{\mathcal{X}}(s)\right] \mathcal{V}(s)  \notag \\
    &\ \qquad\qquad +\widetilde{\e}\left[ D_y\frac{d}{d\nu}D_x b(s,\theta_{t\xi}(s))\widetilde{X_{t\xi}}(s) \widetilde{\mathcal{X}}(s)\right]  \mathcal{X}(s) \bigg\}\Bigg| \notag \\
    \le\ & \frac{L^2}{\lambda_b} \left[L_b^0 |\mathcal{X}(s)|\left(|\mathcal{X}(s)|+ \left\|\widetilde{\mathcal{X}}(s)\right\|_2\right)+L_b^1 |\mathcal{V}(s)| \left(2|\mathcal{X}(s)|+ \left\|\widetilde{\mathcal{X}}(s)\right\|_2\right) +L_b^2 |\mathcal{V}(s)|^2\right]. \label{lem:3_4}
\end{align}
From the convexity condition \eqref{convex'}, we further have
\begin{equation}\label{lem:3_5}
\begin{aligned}
    \left[\begin{pmatrix}
        D_v^2 f & D_vD_xf\\ D_xD_v f & D_x^2f
    \end{pmatrix}(s,\theta_{t\xi}(s))\right]\begin{pmatrix}
        \mathcal{V}(s)\\ \mathcal{X}(s)
    \end{pmatrix}^{\otimes 2}\geq \ & 2\lambda_v |\mathcal{V}(s)|^2+2\lambda_x |\mathcal{X}(s)|^2;\\
    \text{and}\qquad D_x^2 g (X_{t\xi}(T),\lr(X_{t\xi}(T))) \left(  \mathcal{X}(T)\right)^{\otimes 2} \geq\ & 2\lambda_g |\mathcal{X}(T)|^2.
\end{aligned}
\end{equation}
From Condition \eqref{small_mf_condition}, we know that 
\begin{equation}\label{lem:3_6}
\begin{aligned}
    \left|\mathcal{X}(s)^\top \widetilde{\e}\left[\left(D_y \frac{d}{d\nu}D_x f  (s,\theta_{t\xi}(s))\left(\widetilde{X_{t\xi}}(s)\right)\right)^\top \widetilde{\mathcal{X}}(s)\right]\right| \le\ & L_f^0\  |\mathcal{X}(s)|\cdot \left\|\widetilde{\mathcal{X}}(s)\right\|_2; \\
    \left|\mathcal{V}(s)^\top \widetilde{\e}\left[\left(D_y \frac{d}{d\nu}D_v f  (s,\theta_{t\xi}(s))\left(\widetilde{X_{t\xi}}(s)\right)\right)^\top \widetilde{\mathcal{X}}(s)\right] \right| \le\ & L_f^1 \  |\mathcal{V}(s)|\cdot \left\|\widetilde{\mathcal{X}}(s)\right\|_2;\\
    \left| \widetilde{\e}\left[\left(\widetilde{\mathcal{X}}(T)\right)^\top\left(D_y \frac{d}{d\nu}D_x g (X_{t\xi}(T),\lr(X_{t\xi}(T)))\left(\widetilde{X_{t\xi}}(T)\right) \right) \right]\mathcal{X}(T)\right| \le\ & L_g \  |\mathcal{X}(T)|\cdot \left\|\widetilde{\mathcal{X}}(T)\right\|_2.
\end{aligned}
\end{equation}
From Assumption (A1) and Cauchy's inequality, we also know that
\begin{align}
    &\Bigg|\e \int_t^T \sum_{j=1}^n \mathcal{Q}^j(s)^\top  \widetilde{\e}\bigg[ D_y \frac{d\sigma_0^j}{d\nu} (s,\lr(X_{t\xi}(s)))\left(\widetilde{X_{t\xi}}(s)\right) \widetilde{\mathcal{X}}(s)\bigg] ds \Bigg| \notag \\
    \le\ & l_{\sigma}^m \int_t^T \left\|\widetilde{X}(s)\right\|_2 \sum_{j=1}^n \left\|Q^j(s) \right\|_2 ds \notag \\
    \le\ & \sqrt{n}\ l_{\sigma}^m \left(\int_t^T \left\|{X}(s)\right\|_2^2 ds \right)^{\frac{1}{2}}\left(\int_t^T \| Q(s)\|_2^2 ds\right)^{\frac{1}{2}}. \label{sigma_m_5}
\end{align}
By applying It\^o's formula to $|\mathcal{P}(s)|^2$, we know that
\small
\begin{align*}
    &\e\!\!\left[\left|D_x^2 g (X_{t\xi}(T),\lr(X_{t\xi}(T)))^\top  \mathcal{X}(T) +\widetilde{\e}\left[\!\!\left(D_y \frac{d}{d\nu}D_x g (X_{t\xi}(T),\lr(X_{t\xi}(T)))\left(\widetilde{X_{t\xi}}(T)\right) \right)^\top  \widetilde{\mathcal{X}}(T)\right]\right|^2-|\mathcal{P}(t)|^2\right]\\
    =\ & \e\int_t^T \Bigg\{\sum_{j=1}^n \left|\mathcal{Q}^j(s)\right|^2-2\sum_{j=1}^n\mathcal{P}(s)^\top \left(\sigma_1^j (s)\right)^\top \mathcal{Q}^j(s)-2\mathcal{P}(s)^\top D_xb(s,\theta_{t\xi}(s))^\top \mathcal{P}(s)  \\
    &-2 \left(P_{t\xi}(s)\right)^\top\left[D_x^2 b(s,\theta_{t\xi}(s)) \mathcal{X}(s)+\widetilde{\e}\left[ D_y\frac{d}{d\nu}D_x b(s,\theta_{t\xi}(s))\left(\widetilde{X_{t\xi}}(s)\right) \widetilde{\mathcal{X}}(s)\right]+ D_vD_x b(s,\theta_{t\xi}(s)) \mathcal{V}(s)\right]\mathcal{P}(s) \notag\\
    &\quad\qquad -2 \mathcal{P}(s)^\top D_x^2 f  (s,\theta_{t\xi}(s))^\top \mathcal{X}(s)-2\mathcal{P}(s)^\top \widetilde{\e}\left[\left(D_y \frac{d}{d\nu}D_x f  (s,\theta_{t\xi}(s))\left(\widetilde{X_{t\xi}}(s)\right)\right)^\top \widetilde{\mathcal{X}}(s)\right] \notag \\
    &\quad\qquad -2 \mathcal{P}(s)^\top D_vD_x f (s,\theta_{t\xi}(s))^\top  \mathcal{V}(s)\Bigg\}ds,
\end{align*}
\normalsize
then, from the estimate \eqref{lem:2_2} and Assumptions (A1) and (A2), by using Cauchy's inequality, we can deduce that
\begin{align*}
    &\e\int_t^T \sum_{j=1}^n \left|\mathcal{Q}^j(s)\right|^2ds \\
    \le\ & \int_t^T \bigg[2L\sum_{j=1}^n\|\mathcal{P}(s)\|_2\cdot \left\|\mathcal{Q}^j(s)\right\|_2+2L\|\mathcal{P}(s)\|_2^2 +\frac{2 L^2}{\lambda_b} \left(2L_b^0 \|\mathcal{X}(s)\|_2+ L_b^1 \| \mathcal{V}(s)\|_2 \right)\cdot \|\mathcal{P}(s)\|_2 \notag\\
    &\quad\qquad +4L \|\mathcal{P}(s)\|_2 \cdot \|\mathcal{X}(s)\|_2 +2L \|\mathcal{P}(s)\|_2\cdot \|\mathcal{V}(s)\|_2\bigg]ds+4L^2\|\mathcal{X}(T)\|_2^2\\
    \le\ & \frac{1}{2}\ \e\int_t^T \sum_{j=1}^n \left|\mathcal{Q}^j(s)\right|^2ds +4L^2\|\mathcal{X}(T)\|_2^2 \\
    &+ \int_t^T \left[2L(1+nL) \|\mathcal{P}(s)\|_2^2+2L\left(1+\frac{L^2}{\lambda_b}\right) \left(2 \|\mathcal{X}(s)\|_2+  \| \mathcal{V}(s)\|_2 \right)\cdot \|\mathcal{P}(s)\|_2\right]ds.
\end{align*}
Then, from the estimate \eqref{lem:3_2}, we have
\begin{align*}
    \int_t^T\left\|\mathcal{Q}(s)\right\|_2^2ds \le\ &  \frac{4L^3}{\lambda_b}\left(1+\frac{L^2}{\lambda_b}\right)\left(1+\frac{L^2(1+nL)}{\lambda_b}\left(1+\frac{L^2}{\lambda_b}\right)\right)\int_t^T   \left(2 \|\mathcal{X}(s)\|_2+ \| \mathcal{V}(s)\|_2 \right)^2 ds\notag\\
    &+8L^2\|\mathcal{X}(T)\|_2^2. 
\end{align*}
Substituting the last estimate into \eqref{sigma_m_5}, we have
\small
\begin{align}
    &\Bigg|\e \int_t^T \sum_{j=1}^n \mathcal{Q}^j(s)^\top  \widetilde{\e}\bigg[ D_y \frac{d\sigma^j}{d\nu} (s,\lr(X_{t\xi}(s)))\left(\widetilde{X_{t\xi}}(s)\right) \widetilde{\mathcal{X}}(s)\bigg] ds \Bigg| \notag \\
    \le\ & 2\sqrt{2n}\ L \ l_{\sigma}^m \left(\int_t^T \left\|{X}(s)\right\|_2^2 ds \right)^{\frac{1}{2}}\cdot \|\mathcal{X}(T)\|_2 \notag \\
    &+2L\ l_{\sigma}^m\sqrt{\frac{6nL}{\lambda_b}\left(1+\frac{L^2}{\lambda_b}\right)\left(1+\frac{L^2(1+nL)}{\lambda_b}\left(1+\frac{L^2}{\lambda_b}\right)\right)} \int_t^T \left\|{X}(s)\right\|_2^2 ds  \notag\\
    &+2L\ l_{\sigma}^m\sqrt{\frac{3nL}{\lambda_b}\left(1+\frac{L^2}{\lambda_b}\right)\left(1+\frac{L^2(1+nL)}{\lambda_b}\left(1+\frac{L^2}{\lambda_b}\right)\right)} \left(\int_t^T \left\|{X}(s)\right\|_2^2 ds \right)^{\frac{1}{2}}\left(\int_t^T \left\|{V}(s)\right\|_2^2 ds \right)^{\frac{1}{2}}. \label{sigma_m_6}
\end{align}
\normalsize
Substituting \eqref{lem:3_3}-\eqref{lem:3_6} and \eqref{sigma_m_6} back into \eqref{lem:3_1}, from Cauchy's inequality, we have
\small
\begin{align*}
    &\e\left[2\lambda_g |\mathcal{X}(T)|^2 + \int_t^T\left(2\lambda_v |\mathcal{V}(s)|^2+2\lambda_x |\mathcal{X}(s)|^2\right) ds \right]\notag \\
    \le\ & \e\left[ L_g \  |\mathcal{X}(T)|\cdot \left\|\widetilde{\mathcal{X}}(T)\right\|_2+ |\mathcal{P}(t)|\cdot |\eta| \right] +2\sqrt{2n}Ll_{\sigma}^m \left(\int_t^T \left\|{X}(s)\right\|_2^2 ds \right)^{\frac{1}{2}}\cdot \|\mathcal{X}(T)\|_2 \notag \\
    &+\e \int_t^T \Bigg\{ \frac{L^2l_b^m}{\lambda_b} \left(1+\frac{L^2}{\lambda_b}\right) \left(|\mathcal{X}(s)|+ \left\|\widetilde{\mathcal{X}}(s)\right\|_2 +|\mathcal{V}(s)|\right)\left\|\widetilde{\mathcal{X}}(s)\right\|_2 \notag \\
    &\quad\qquad\qquad +\frac{L^2}{\lambda_b} \left[L_b^0 |\mathcal{X}(s)|\left(|\mathcal{X}(s)|+ \left\|\widetilde{\mathcal{X}}(s)\right\|_2\right)+L_b^1 |\mathcal{V}(s)| \left(2|\mathcal{X}(s)|+ \left\|\widetilde{\mathcal{X}}(s)\right\|_2\right) +L_b^2 |\mathcal{V}(s)|^2\right] \notag \\
    &\quad\qquad\qquad +L_f^0\  |\mathcal{X}(s)|\cdot \left\|\widetilde{\mathcal{X}}(s)\right\|_2+L_f^1 \  |\mathcal{V}(s)|\cdot \left\|\widetilde{\mathcal{X}}(s)\right\|_2  \Bigg\}ds\\
    &\quad+ 2Ll_{\sigma}^m\sqrt{\frac{6nL}{\lambda_b}\left(1+\frac{L^2}{\lambda_b}\right)\left(1+\frac{L^2(1+nL)}{\lambda_b}\left(1+\frac{L^2}{\lambda_b}\right)\right)} \int_t^T \left\|{X}(s)\right\|_2^2 ds \\
    &+ 2Ll_{\sigma}^m\sqrt{\frac{3nL}{\lambda_b}\left(1+\frac{L^2}{\lambda_b}\right)\left(1+\frac{L^2(1+nL)}{\lambda_b}\left(1+\frac{L^2}{\lambda_b}\right)\right)} \left(\int_t^T \left\|{X}(s)\right\|_2^2 ds \right)^{\frac{1}{2}}\left(\int_t^T \left\|{V}(s)\right\|_2^2 ds \right)^{\frac{1}{2}} \\
    \le\ &  \|\mathcal{P}(t)\|_2\cdot \|\eta\|_2+L_g \ \left\|\mathcal{X}(T)\right\|_2^2 +2\sqrt{2n}Ll_{\sigma}^m \left(\int_t^T \left\|{X}(s)\right\|_2^2 ds \right)^{\frac{1}{2}}\cdot \|\mathcal{X}(T)\|_2 + \frac{L^2L_b^2}{\lambda_b}\int_t^T \|\mathcal{V}(s)\|_2^2 ds\\
    &+ \left\{L_f^1+ \frac{3L^2L_b^1}{\lambda_b}+\frac{L^2l_b^m}{\lambda_b} \left(1+\frac{L^2}{\lambda_b}\right) +2Ll_{\sigma}^m\sqrt{\frac{3nL}{\lambda_b}\left(1+\frac{L^2}{\lambda_b}\right)\left(1+\frac{L^2(1+nL)}{\lambda_b}\left(1+\frac{L^2}{\lambda_b}\right)\right)} \right\} \\
    &\qquad \cdot \left(\int_t^T \|\mathcal{V}(s)\|^2_2 ds\right)^{\frac{1}{2}} \left(\int_t^T \left\|\mathcal{X}(s)\right\|^2_2 ds\right)^{\frac{1}{2}} \\
    &+ \left\{L_f^0+ \frac{2L^2L_b^0}{\lambda_b} +\frac{2L^2l_b^m}{\lambda_b} \left(1+\frac{L^2}{\lambda_b}\right) +2Ll_{\sigma}^m\sqrt{\frac{6nL}{\lambda_b}\left(1+\frac{L^2}{\lambda_b}\right)\left(1+\frac{L^2(1+nL)}{\lambda_b}\left(1+\frac{L^2}{\lambda_b}\right)\right)} \right\} \\
    &\qquad \cdot \int_t^T \left\|\mathcal{X}(s)\right\|_2^2 ds.
\end{align*}
\normalsize
From Condition \eqref{lem:2_0} and Young's inequality, we further obtain that
\begin{align*}
    &\Bigg|\sqrt{2\lambda_g-L_g} \ \left\|\mathcal{X}(T)\right\|_2 - \frac{\sqrt{2n}Ll_{\sigma}^m }{\sqrt{2\lambda_g-L_g}}\left(\int_t^T \left\|{X}(s)\right\|_2^2\  ds \right)^{\frac{1}{2}} \Bigg|^2\\
    &+ \Bigg| \left(\Lambda_1\int_t^T \left\|{X}(s)\right\|_2^2\  ds \right)^{\frac{1}{2}}- \frac{1}{2}\Lambda_2  \left(\frac{1}{\Lambda_1}\int_t^T \left\|{V}(s)\right\|_2^2 \ ds \right)^{\frac{1}{2}}\Bigg|^2 \\
    &+ \left(2\lambda_v-\frac{L^2L_b^2}{\lambda_b} - \frac{\Lambda_2}{2\sqrt{\Lambda_1}}\right) \int_t^T \|\mathcal{V}(s)\|_2^2\  ds
   \quad  \le \quad  \|\mathcal{P}(t)\|_2\cdot \|\eta\|_2,
\end{align*}
where
\begin{align*}
    c_*(L,\lambda_b):= 2L\sqrt{\frac{3nL}{\lambda_b}\left(1+\frac{L^2}{\lambda_b}\right)\left(1+\frac{L^2(1+nL)}{\lambda_b}\left(1+\frac{L^2}{\lambda_b}\right)\right)},
\end{align*}
and
\begin{align*}
    \Lambda_1:=\ & 2\lambda_x - L_f^0 - \frac{2L^2L_b^0}{\lambda_b} - \frac{2L^2l_b^m}{\lambda_b} \left(1+\frac{L^2}{\lambda_b}\right) -\sqrt{2}c_*(L,\lambda_b)\ l_{\sigma}^m - \frac{2nL^2(l_\sigma^m)^2}{2\lambda_g-L_g},\\
    \Lambda_2:=\ & L_f^1+ \frac{3L^2L_b^1}{\lambda_b}+\frac{L^2l_b^m}{\lambda_b} \left(1+\frac{L^2}{\lambda_b}\right) +c_*(L,\lambda_b)\ l_{\sigma}^m.
\end{align*}
Therefore, for any $\epsilon>0$, we have
\begin{align}
    &\int_t^T \|\mathcal{V}(s)\|_2^2\ ds \le C \|\mathcal{P}(t)\|_2 \|\eta\|_2 \le \epsilon  \  \|\mathcal{P}(t)\|_2^2+\frac{C}{4\epsilon} \|\eta\|_2^2. \label{lem:3_7}
\end{align}
Using Gr\"onwall's inequality, we have
\begin{equation}\label{lem:3_8}
	\e\bigg[\sup_{t\le s\le T}|\mathcal{X}(s)|^2\bigg]\le C(L,T)\ \e\left[|\eta|^2+\int_t^T |\mathcal{V}(s)|^2ds\right].
\end{equation}
Then, from Assumption \eqref{generic:condition:b} and the estimate \eqref{lem:2_2}, the usual BSDE estimates in \cite{YH2,SP} together with BDG inequality altogether give us:
\begin{equation}\label{lem:3_9}
	\e\bigg[\sup_{t\le s\le T}|\mathcal{P}(s)|^2+\int_t^T |\mathcal{Q}(s)|^2 ds\bigg]\le C(L,T,\lambda_b)\ \e\left[|\eta|^2 +\int_t^T |\mathcal{V}(s)|^2ds\right].
\end{equation}    
Substituting \eqref{lem:3_9} into \eqref{lem:3_7} and by choosing $\epsilon$ small enough, we deduce that
\begin{align*}
    \int_t^T \|\mathcal{V}(s)\|_2^2 ds \le   C \|\eta\|_2^2. 
\end{align*}
Substituting the last estimate back into \eqref{lem:3_8} and \eqref{lem:3_9}, we have
\begin{align}\label{lem:3_10}
    &\e\bigg[\sup_{t\le s\le T}|(\mathcal{X}(s),\mathcal{P}(s))|^2+\int_t^T |\mathcal{Q}(s)|^2 ds\bigg]\le C\|\eta\|_2^2.
\end{align}
From Condition \eqref{FB:dr_condition} and the convexity of $f$ in $v$, we know that
\begin{align*}
    &2\lambda_v |\mathcal{V}(s)|^2 \\
    \le\ &D_v^2 f (s,\theta_{t\xi}(s)) (\mathcal{V}(s))^{\otimes 2} \\
    =\ & -\mathcal{V}(s)^\top D_x D_v f  (s,\theta_{t\xi}(s))^\top \mathcal{X}(s) - \mathcal{V}(s)^\top \widetilde{\e}\left[\left(D_y \frac{d}{d\nu}D_v f  (s,\theta_{t\xi}(s))\left(\widetilde{X_{t\xi}}(s)\right)\right)^\top \widetilde{\mathcal{X}}(s)\right] \notag \\
    &- \mathcal{V}(s)^\top \bigg\{D_x D_v b(s,\theta_{t\xi}(s)) \mathcal{X}(s)+\widetilde{\e}\left[ D_y\frac{d}{d\nu}D_v b(s,\theta_{t\xi}(s))\left(\widetilde{X_{t\xi}}(s) \right) \widetilde{\mathcal{X}}(s)\right]\\
    &\quad\qquad\qquad + D_v^2 b(s,\theta_{t\xi}(s)) \mathcal{V}(s)\bigg\}^\top P_{t\xi}(s) - \mathcal{V}(s)^\top D_v b(s,\theta_{t\xi}(s))^\top \mathcal{P}(s),
\end{align*}
and then, from Assumptions (A1) and (A2), we can deduce that
\begin{align*}
    2\lambda_v |\mathcal{V}(s)|\le \ & L |\mathcal{X}(s)|+L \|\mathcal{X}(s)\|_2+ L|\mathcal{P}(s)|+\frac{L^2}{\lambda_b} \left(L |\mathcal{X}(s)|+L \|\mathcal{X}(s)\|_2+L_b^2|\mathcal{V}(s)|\right). 
\end{align*}
Again from Condition \eqref{lem:2_0}, we conclude that
\begin{align}\label{lem:3_11}
    |\mathcal{V}(s)|\le \ & \left(2\lambda_v-\frac{L^2 L_b^2}{\lambda_b}\right)^{-1}\left[\left(L+\frac{L^3}{\lambda_b}\right) \left(|\mathcal{X}(s)|+\|\mathcal{X}(s)\|_2\right)+ L|\mathcal{P}(s)|\right]. 
\end{align}
From \eqref{lem:3_10} and \eqref{lem:3_11}, we obtain \eqref{lem3_1}. 
\end{proof}

We emphasize that the key of this proof is through \eqref{lem:2_2(zw)} and also the following the cone property:
\begin{align*}
    \left|\mathcal{D}_\eta P_{t\xi}(s)\right|\le &\frac{L^2}{\lambda_b} \left(1+\frac{L^2}{\lambda_b}\right) \left(\left|\mathcal{D}_\eta X_{t\xi}(s)\right|+ \left\|\mathcal{D}_\eta X_{t\xi}(s)\right\|_2 +\left|\mathcal{D}_\eta v_{t\xi}(s)\right|\right),
\end{align*}
which is a consequence of Assumptions (A2), \eqref{generic:condition:b}, \eqref{generic:condition:b'} and the estimate \eqref{lem:2_2(zw)}; see \eqref{lem:3_2'} and \eqref{lem:3_2} for details. We next show that the components of the solution of FBSDEs \eqref{FB:dr} are the G\^ateaux derivatives of $\Theta_{t\xi}(s)$ with respect to the initial condition $\xi$ along the direction $\eta$.

\begin{theorem}\label{lem:4}
	Under Assumptions (A1)-(A3) and the validity of \eqref{lem:2_0}, for any $\eta\in L_{\f_t}^2$, the solution $\left(\mathcal{D}_\eta X_{t\xi},\mathcal{D}_\eta P_{t\xi},\mathcal{D}_\eta Q_{t\xi},\mathcal{D}_\eta v_{t\xi}\right)$ of FBSDEs \eqref{FB:dr} satisfies
	\begin{equation}\label{lem4_1}
		\begin{split}
			\lim_{\epsilon\to0}\e\Bigg[& \sup_{t\le s\le T}\left|\frac{ X_{ t\xi^\epsilon}(s)-X_{t\xi}(s)}{\epsilon}-\mathcal{D}_\eta X_{t\xi}(s)\right|^2+ \sup_{t\le s\le T}\left|\frac{ P_{ t\xi^\epsilon}(s)-P_{t\xi}(s)}{\epsilon}-\mathcal{D}_\eta P_{t\xi}(s)\right|^2\\
            &+\sup_{t\le s\le T}\left|\frac{ v_{ t\xi^\epsilon}(s)-v_{t\xi}(s)}{\epsilon}-\mathcal{D}_\eta v_{t\xi}(s)\right|^2 + \int_t^T \left|\frac{Q_{ t\xi^\epsilon}(s)-Q_{t\xi}(s)}{\epsilon}-\mathcal{D}_\eta Q_{t\xi}(s)\right|^2 ds\Bigg]=0,
		\end{split}
	\end{equation}
	where $\xi^\epsilon:=\xi+\epsilon\eta$ for $\epsilon\in(0,1)$. In fact, the components of $\left(\mathcal{D}_\eta X_{t\xi},\mathcal{D}_\eta P_{t\xi},\mathcal{D}_\eta Q_{t\xi},\mathcal{D}_\eta v_{t\xi}\right)$ defined in \eqref{FB:dr} are the respective G\^ateaux derivatives of $(X_{t\xi}(s),P_{t\xi}(s),Q_{t\xi}(s),v_{t\xi}(s))$ with respect to $\xi$ along the direction $\eta$. Moreover, the G\^ateaux derivatives satisfy the boundedness condition \eqref{lem3_1}, and they are linear in $\eta$ and continuous in $\xi$. 
\end{theorem}

The proof of Theorem~\ref{lem:4} is given in Appendix~\ref{pf:lem:4}, and we refer to \cite[Theorem 3.2]{AB12} and \cite[Lemma 3.8]{AB9'1} for similar proofs; we also refer to \cite{AB5} for further discussion on their very nature as Fr\'echet derivatives. Theorem~\ref{lem:4} gives the differentiability of the processes $\Theta_{t\xi}$ with respect to the initial condition $\xi$. As a consequence, in the rest of this article, we denote by $D_\eta \Theta_{t\xi}:=\left(D_\eta X_{t\xi},D_\eta P_{t\xi},D_\eta Q_{t\xi},D_\eta v_{t\xi}\right)$ the G\^ateaux derivatives along the direction $\eta$, which satisfy the FBSDEs \eqref{FB:dr}. 

\subsection{Jacobian flows of FBSDEs \eqref{intro_3}}

We now consider the differentiability of the process $(X_{tx\mu},P_{tx\mu},Q_{tx\mu},v_{tx\mu})$ defined in \eqref{intro_3}, with respect to the initial data $(x,\mu)\in\brn\times\pr_2(\brn)$. For the sake of convenience, in the rest of this article, we denote by $\theta_{tx\mu}$ the process $(X_{tx\mu},\lr(X_{t\xi}),v_{tx\mu})$ and by $\Theta_{tx\mu}$ the process $(X_{tx\mu},P_{tx\mu},Q_{tx\mu},v_{tx\mu})$, where $\lr(\xi)=\mu$. Since the process $\Theta_{tx\mu}$ depends on $\xi$ only through its law $\mu$, it is natural to use the subscript $\mu$ in $\Theta_{tx\mu}$ to indicate such dependence. We begin by considering the differentiability of $\Theta_{tx\mu}(s)$ with respect to the initial state $x\in\brn$. Consider the following FBSDEs for $D_x\Theta_{tx\mu}=\left(D_xX_{tx\mu}, D_xP_{tx\mu},D_xQ_{tx\mu},D_xv_{tx\mu}\right)$: 
\begin{align}
		D_xX_{tx\mu}(s)=\ & I+\int_t^s \left[D_x b\left(r,\theta_{tx\mu}(r)\right)D_xX_{tx\mu}(r)+ D_v b\left(r,\theta_{tx\mu}(r)\right)D_xv_{tx\mu}(r) \right] dr \notag \\
        & +\int_t^s \sigma_1(r) D_xX_{tx\mu}(r) dB(r), \notag \\
		D_xP_{tx\mu}(s)=\ & -\int_s^T D_x Q_{tx\mu} (r)dB(r)+D_x^2 g\left(X_{tx\mu}(T),\lr(X_{t\xi}(T))\right)^\top D_xX_{tx\mu}(T) \notag \\
        & +\int_s^T \bigg[D_x b\left(r,\theta_{tx\mu}(r)\right)^\top D_x P_{tx\mu}(r)+\sum_{j=1}^n\left(\sigma_1^j(r)\right)^\top D_x Q_{tx\mu}^j(r) \notag \\
        & +\left(P_{tx\mu}(r)\right)^\top \left[D_x^2 b\left(r,\theta_{tx\mu}(r)\right)D_xX_{tx\mu}(r) + D_v D_x b\left(r,\theta_{tx\mu}(r)\right)D_xv_{tx\mu}(r)\right] \notag \\
        & +\left[D_x^2 f (r,\theta_{tx\mu}(r)) \right]^\top D_x X_{tx\mu}(r) + \left[D_vD_x f (r,\theta_{tx\mu}(r))\right]^\top D_x v_{tx\mu}(r) \bigg]dr,\quad s\in[t,T], \label{FB:x}
\end{align}
with the derived condition after the first order one:
\begin{align}
	0=\ &\left[D_xD_v f (s,\theta_{tx\mu}(s)) \right]^\top D_x X_{tx\mu}(s) + \left[D_v^2 f (s,\theta_{tx\mu}(s))\right]^\top D_x v_{tx\mu}(s)+D_v b(s,\theta_{tx\mu}(s))^\top D_xP_{tx\mu}(s) \notag \\
    &+\left(P_{tx\mu}(s)\right)^\top \left[D_xD_v b\left(s,\theta_{tx\mu}(s)\right)D_xX_{tx\mu}(s) + D_v^2  b\left(s,\theta_{tx\mu}(s)\right)D_xv_{tx\mu}(s)\right], \label{FB:x_condition}
\end{align}
where $I$ is the identity matrix. 
The following result gives the solvability of FBSDEs \eqref{FB:x} and shows that the components of $D_x \Theta_{tx\mu}(s)$ are the respective G\^ateaux derivatives of $\Theta_{t\xi}(s)$ in $x\in\brn$. Its proof is similar to those for Lemma~\ref{lem:3} and Theorem~\ref{lem:4}, and it is omitted here.

\begin{theorem}\label{prop:2}
	Under Assumptions (A1)-(A3) and the validity of \eqref{lem:2_0}, there is a unique solution $D_x \Theta_{tx\mu}$ of FBSDEs \eqref{FB:x}-\eqref{FB:x_condition}, which is the G\^ateaux derivative of $\Theta_{tx\mu}(s)$ with respect to $x$. Moreover, these G\^ateaux derivatives satisfy the following boundedness property:
	\begin{equation}\label{thm2_1}
		\begin{split}
			&\e\bigg[\sup_{t\le s\le T}|D_x X_{tx\mu}(s)|^l+\sup_{t\le s\le T}|D_x P_{tx\mu}(s)|^l+\sup_{t\le s\le T}|D_x v_{tx\mu}(s)|^l\bigg]^{\frac{1}{l}}\\
            &+\left[\int_t^T\e\left[|D_x Q_{tx\mu}(s)|^l\right]^{\frac{2}{l}} ds\right]^{\frac{1}{2}} \le C,\quad l=2,4,
		\end{split}
	\end{equation}
	and they are also continuous in $(x,\mu)$. 
\end{theorem}

The boundedness in $L^4$-norm of $D_x \Theta_{tx\mu}$ in \eqref{thm2_1} is required for sake of the second order derivatives (Hessian flows) of the processes $\Theta_{tx\mu}$; we also refer to \cite{AB11,AB9'2} for similar estimates and detailed analysis. To consider the G\^ateaux differentiability of $\Theta_{tx\mu}$ with respect to $\xi\in L_{\f_t}^2$ along the direction $\eta$, where $\lr(\xi)=\mu$ and $\xi$ and $\eta$ are both independent of the Brownian motion $B$, we shall characterize the G\^ateaux derivatives as the solution of a system of FBSDEs. Instead of using the process $D_\eta \Theta_{t\xi}$ defined in FBSDEs \eqref{FB:dr} directly, we prefer to use the following system of FBSDEs. The advantage of using this is that its initial condition is $0$, which will be more convenient when considering the linear functional differentiability of $\Theta_{tx\mu}(s)$ with respect to $\mu\in\pr_2(\brn)$. In view of the processes $\Theta_{t\xi}$ and  $D_x\Theta_{tx\mu}$, this system of FBSDEs is simply: for $s\in[t,T]$,
\small
\begin{align}
		\bd_\eta X_{t\xi}(s)=\ & \int_t^s \bigg\{D_x b(r, \theta_{t\xi}(r))\bd_\eta X_{t\xi}(r)+D_v b(r, \theta_{t\xi}(r))\bd_\eta v_{t\xi}(r) \notag \\
        &\qquad +\widetilde{\e}\bigg[ D_y \frac{db}{d\nu} (r,\theta_{t\xi}(r))\left(\widetilde{X_{t\xi}}(r)\right) \left(\widetilde{\bd_\eta X_{t\xi}}(r)+\widetilde{D_x X_{tx\mu}}(r)\left|_{x=\widetilde{\xi}}\ \right. \widetilde{\eta}\right)\bigg] \bigg\}dr \notag \\
        &+\int_t^s \bigg\{ \widetilde{\e}\bigg[ D_y \frac{d\sigma_0}{d\nu} (r,\lr(X_{t\xi}(r)))\left(\widetilde{X_{t\xi}}(r)\right) \left(\widetilde{\bd_\eta X_{t\xi}}(r)+\widetilde{D_x X_{tx\mu}}(r)\left|_{x=\widetilde{\xi}}\ \right. \widetilde{\eta}\right)\bigg] \notag \\
        &\quad\qquad + \sigma_1(r)\bd_\eta X_{t\xi}(r)\bigg\} dB(r), \notag \\
		\bd_\eta P_{t\xi}(s)=\ & -\int_s^T\bd_\eta Q_{t\xi}(r)dB(r)+ D_x^2 g (X_{t\xi}(T),\lr(X_{t\xi}(T)))^\top  \bd_\eta X_{t\xi}(T) \notag \\
        &+\widetilde{\e}\left[\left(D_y \frac{d}{d\nu}D_x g (X_{t\xi}(T),\lr(X_{t\xi}(T)))\left(\widetilde{X_{t\xi}}(T)\right) \right)^\top \left(\widetilde{\bd_\eta X_{t\xi}}(T)+ \widetilde{D_xX_{tx\mu}}(T)\left|_{x=\widetilde{\xi}}\ \right. \widetilde{\eta}\right)\right] \notag \\
		&+\int_s^T\Bigg\{ D_xb(r,\theta_{t\xi}(r))^\top \bd_\eta P_{t\xi}(r)+\sum_{j=1}^n\left(\sigma_1^j (r)\right)^\top \bd_\eta Q_{t\xi}^j(r) \notag \\
        &\quad\qquad + \left(P_{t\xi}(r)\right)^\top \bigg\{D_x^2 b(r,\theta_{t\xi}(r)) \bd_\eta X_{t\xi}(r)+ D_vD_x b(r,\theta_{t\xi}(r)) \bd_\eta v_{t\xi}(r) \notag \\
        &\quad\qquad\qquad\qquad\qquad +\widetilde{\e}\left[ D_y\frac{d}{d\nu}D_x b(r,\theta_{t\xi}(r))\left(\widetilde{X_{t\xi}}(r)\right) \left(\widetilde{\bd_\eta X_{t\xi}}(r)+ \widetilde{D_xX_{tx\mu}}(r)\left|_{x=\widetilde{\xi}}\ \right. \widetilde{\eta}\right)\right]\bigg\}  \notag\\
        &\quad+D_x^2 f  (r,\theta_{t\xi}(r))^\top \bd_\eta X_{t\xi}(r) + D_vD_x f (r,\theta_{t\xi}(r))^\top  \bd_\eta v_{t\xi}(r) \notag\\
        &\quad+\widetilde{\e}\left[\left(D_y \frac{d}{d\nu}D_x f  (r,\theta_{t\xi}(r))\left(\widetilde{X_{t\xi}}(r)\right)\right)^\top  \left(\widetilde{\bd_\eta X_{t\xi}}(r)+ \widetilde{D_xX_{tx\mu}}(r)\left|_{x=\widetilde{\xi}}\ \right. \widetilde{\eta}\right)\right] \Bigg\}dr, \label{FB:dr'}
\end{align}
\normalsize
with the following derived condition after the first order one: 
\small
\begin{align}
    0=\ & D_x D_v f  (s,\theta_{t\xi}(s))^\top \bd_\eta X_{t\xi}(s) +\widetilde{\e}\left[\left(D_y \frac{d}{d\nu}D_v f  (s,\theta_{t\xi}(s))\left(\widetilde{X_{t\xi}}(s)\right)\right)^\top \left(\widetilde{\bd_\eta X_{t\xi}}(s)+ \widetilde{D_xX_{tx\mu}}(s)\left|_{x=\widetilde{\xi}}\ \right. \widetilde{\eta}\right)\right] \notag \\
    &+D_v^2 f (s,\theta_{t\xi}(s))^\top  \bd_\eta v_{t\xi}(s)+D_v b(s,\theta_{t\xi}(s))^\top \bd_\eta P_{t\xi}(s) \notag\\
    &+ \left(P_{t\xi}(s)\right)^\top \bigg\{D_x D_v b(s,\theta_{t\xi}(s)) \bd_\eta X_{t\xi}(s)+ D_v^2 b(s,\theta_{t\xi}(s)) \bd_\eta v_{t\xi}(s) \notag \\
    &\qquad\qquad\qquad +\widetilde{\e}\left[ D_y\frac{d}{d\nu}D_v b(s,\theta_{t\xi}(s))\left(\widetilde{X_{t\xi}}(s) \right) \left(\widetilde{\bd_\eta X_{t\xi}}(s)+ \widetilde{D_xX_{tx\mu}}(s)\left|_{x=\widetilde{\xi}}\ \right. \widetilde{\eta}\right)\right]\bigg\}, \label{FB:dr'_condition}
\end{align}
\normalsize
where $\left(\widetilde{\xi},\widetilde{\eta},\widetilde{X_{t\xi}}(s),\widetilde{D_x X_{tx\mu}}(s),\widetilde{\bd_\eta X_{t\xi}}(s)\right)$ is an independent copy of 
$$(\xi,\eta,X_{t\xi}(s),D_x X_{tx\mu}(s),\bd_\eta X_{t\xi}(s)). $$
We now give a connection among the above FBSDEs \eqref{FB:dr'}-\eqref{FB:dr'_condition} and $D_\eta \Theta_{t\xi}$ and $D_x\Theta_{tx\mu}$. 

\begin{lemma}\label{prop:4}
	Under Assumptions (A1)-(A3) and the validity of \eqref{lem:2_0}, there is a unique adapted solution $\bd_\eta \Theta_{t\xi}:=\left(\bd_\eta X_{t\xi},\bd_\eta P_{t\xi},\bd_\eta Q_{t\xi},\bd_\eta v_{t\xi}\right)$ of FBSDEs \eqref{FB:dr'}-\eqref{FB:dr'_condition}, and they satisfy the following relation:
	\begin{equation}\label{prop4_1}
		\begin{split}
			&D_\eta \Theta_{t\xi}(s)=\Big(D_x\Theta_{tx\mu}(s)\big|_{x=\xi}\Big) \eta+\bd_\eta \Theta_{t\xi}(s),\quad s\in[t,T].
		\end{split}
	\end{equation}
\end{lemma}

In \eqref{prop4_1}, we can see that $\Big(D_x\Theta_{tx\mu}(s)\big|_{x=\xi}\Big) \eta$ and $\bd_\eta \Theta_{t\xi}(s)$ respectively represent the variation against the initial state and the variation against the initial distribution of the directional derivative of $\Theta_{t\xi}(s)$ along $\eta$; we also refer to \cite{AB11,BR} for similar motivations. The proof for Lemma~\ref{prop:4} is very similar to that for \cite[Lemma 4.2]{AB11}, but for the sake of reference for readers, we also put that in Appendix~\ref{pf:prop:4}. 

We now study the G\^ateaux differentiability of the map $\xi\mapsto \Theta_{tx\lr(\xi)}$ with respect to $\xi\in L_{\f_t}^2$ along the direction $\eta$. In view of processes $\Theta_{t\xi}$, $\Theta_{tx\mu}$, $D_x\Theta_{tx\mu}$ and $\bd_\eta \Theta_{t\xi}$, we consider the following FBSDEs for $\left(D_\eta X_{tx\xi},D_\eta P_{tx\xi},D_\eta Q_{tx\xi},D_\eta v_{tx\xi}\right)$: for $s\in[t,T]$, 
\small
\begin{align}
		D_\eta X_{tx\xi}(s)=\ & \int_t^s \bigg\{ \widetilde{\e}\bigg[ D_y \frac{db}{d\nu} (r,\theta_{tx\mu}(r))\left(\widetilde{X_{t\xi}}(r)\right) \left(\widetilde{D_x X_{tx\mu}}(r)\left|_{x=\widetilde{\xi}}\ \right. \widetilde{\eta}+\widetilde{\bd_\eta X_{t\xi}}(r)\right)\bigg] \notag \\
        &\quad\qquad +D_x b(r, \theta_{tx\mu}(r))D_\eta X_{tx\xi}(r)+D_v b(r, \theta_{tx\mu}(r))D_\eta v_{tx\xi}(r)  \bigg\}dr \notag \\
        &+\int_t^s \bigg\{ \widetilde{\e}\bigg[ D_y \frac{d\sigma_0}{d\nu} (r,\lr(X_{t\xi}(r)))\left(\widetilde{X_{t\xi}}(r)\right) \left(\widetilde{D_x X_{tx\mu}}(r)\left|_{x=\widetilde{\xi}}\ \right. \widetilde{\eta}+\widetilde{\bd_\eta X_{t\xi}}(r)\right)\bigg] \notag \\
        &\qquad\qquad + \sigma_1(r)D_\eta X_{tx\xi}(r) \bigg\}dB(r), \notag \\
		D_\eta P_{tx\xi}(s)=\ & -\int_s^TD_\eta Q_{tx\xi}(r)dB(r)+D_x^2 g (X_{tx\mu}(T),\lr(X_{t\xi}(T)))^\top  D_\eta X_{tx\xi}(T) \notag\\
        &+\widetilde{\e}\left[\left(D_y \frac{d}{d\nu}D_x g (X_{tx\mu}(T),\lr(X_{t\xi}(T)))\left(\widetilde{X_{t\xi}}(T)\right) \right)^\top  \left(\widetilde{D_x X_{tx\mu}}(T)\left|_{x=\widetilde{\xi}}\ \right. \widetilde{\eta}+\widetilde{\bd_\eta X_{t\xi}}(T)\right) \right] \notag \\
		&+\int_s^T\Bigg\{ D_xb(r,\theta_{tx\mu}(r))^\top D_\eta P_{tx\xi}(r)+\sum_{j=1}^n\left(\sigma_1^j (r)\right)^\top D_\eta Q_{tx\xi}^j(r) \notag \\
        &\quad+ \left(P_{tx\mu}(r)\right)^\top \bigg\{\widetilde{\e}\left[ D_y\frac{d}{d\nu}D_x b(r,\theta_{tx\mu}(r))\left(\widetilde{X_{t\xi}}(r)\right) \left(\widetilde{D_x X_{tx\mu}}(r)\left|_{x=\widetilde{\xi}}\ \right. \widetilde{\eta}+\widetilde{\bd_\eta X_{t\xi}}(r)\right)\right] \notag \\
        &\qquad\qquad\qquad\qquad+D_x^2 b(r,\theta_{tx\mu}(r)) D_\eta X_{tx\xi}(r)+ D_vD_x b(r,\theta_{tx\mu}(r)) D_\eta v_{tx\xi}(r) \bigg\} \notag\\
        &\quad+\widetilde{\e}\left[\left(D_y \frac{d}{d\nu}D_x f  (r,\theta_{tx\mu}(r))\left(\widetilde{X_{t\xi}}(r)\right)\right)^\top \left(\widetilde{D_x X_{tx\mu}}(r)\left|_{x=\widetilde{\xi}}\ \right. \widetilde{\eta}+\widetilde{\bd_\eta X_{t\xi}}(r)\right)\right] \notag \\
        &\quad\qquad\qquad +D_x^2 f  (r,\theta_{tx\mu}(r))^\top D_\eta X_{tx\xi}(r) + D_vD_x f (r,\theta_{tx\mu}(r))^\top  D_\eta v_{tx\xi}(r)\Bigg\}dr,  \label{FB:mu'} 
\end{align}
\normalsize
with (as a consequence of taking the G\^ateaux derivative in the first order condition)
\small
\begin{align}
    0=\ & \widetilde{\e}\left[\left(D_y \frac{d}{d\nu}D_v f  (s,\theta_{tx\mu}(s))\left(\widetilde{X_{t\xi}}(s)\right)\right)^\top \left(\widetilde{D_x X_{tx\mu}}(s)\left|_{x=\widetilde{\xi}}\ \right. \widetilde{\eta}+\widetilde{\bd_\eta X_{t\xi}}(s)\right)\right] \notag \\
    &+D_x D_v f  (s,\theta_{tx\mu}(s))^\top D_\eta X_{tx\xi}(s) +D_v^2 f (s,\theta_{tx\mu}(s))^\top  D_\eta v_{tx\xi}(s)+D_v b(s,\theta_{tx\mu}(s))^\top D_\eta P_{tx\xi}(s) \notag\\
    &+ \left(P_{tx\mu}(s)\right)^\top \bigg\{\widetilde{\e}\left[ D_y\frac{d}{d\nu}D_v b(s,\theta_{tx\mu}(s))\left(\widetilde{X_{t\xi}}(s) \right) \left(\widetilde{D_x X_{tx\mu}}(s)\left|_{x=\widetilde{\xi}}\ \right. \widetilde{\eta}+\widetilde{\bd_\eta X_{t\xi}}(s)\right)\right] \notag \\
    &\quad\qquad\qquad\qquad +D_x D_v b(s,\theta_{tx\mu}(s)) D_\eta X_{tx\xi}(s)+ D_v^2 b(s,\theta_{tx\mu}(s)) D_\eta v_{tx\xi}(s)\bigg\} , \label{FB:mu'_condition}
\end{align}
\normalsize
where $\left(\widetilde{\xi},\widetilde{\eta},\widetilde{X_{t\xi}}(s),\widetilde{D_x X_{tx\mu}}(s),\widetilde{\bd_\eta X_{t\xi}}(s)\right)$ is an independent copy of $$(\xi,\eta,X_{t\xi}(s),D_x X_{tx\mu}(s),\bd_\eta X_{t\xi}(s)).$$
 We then have the following. 

\begin{theorem}\label{prop:3}
	Under Assumptions (A1)-(A3) and the validity of \eqref{lem:2_0}, for any $\eta\in L_{\f_t}^2$, there is a unique adapted solution $D_\eta \Theta_{tx\xi}:=\left(D_\eta X_{tx\xi},D_\eta P_{tx\xi},D_\eta Q_{tx\xi},D_\eta v_{tx\xi}\right)$ of FBSDEs \eqref{FB:mu'}-\eqref{FB:mu'_condition}, and $D_\eta \Theta_{tx\xi}(s)$ are the respective G\^ateaux derivatives of the component processes of $\Theta_{tx\mu}(s)$ with respect to the lifting $\xi\sim\mu$ along the direction $\eta$. Moreover, the G\^ateaux derivatives satisfy the following boundedness property	
	\begin{align}\label{prop3_1}
		\e\bigg[\sup_{t\le s\le T}\left|\left(D_\eta X_{tx\xi}(s),D_\eta P_{tx\xi}(s),D_\eta v_{tx\xi}(s)\right)^\top\right|^2 +\int_t^T|D_\eta Q_{tx\xi}(s)|^2 ds\bigg]\le C\e\|\eta\|_2^2,
	\end{align}
	and they are linear in $\eta$ and continuous in $(x,\xi)$.
\end{theorem}

In the theorem, the G\^ateaux derivatives of $\Theta_{tx\mu}(s)$ with respect to the lifting $\xi\sim\mu$ along the direction $\eta$ actually should be denoted by $D_\eta \Theta_{tx\lr(\xi)}$, but for national convenience, we use $D_\eta \Theta_{tx\xi}$ instead by recalling that the process $\Theta_{tx\lr(\xi)}$ depends on $\xi$ only through its law $\xi$. The proof is similar to that of Lemma~\ref{lem:3}, and its key  is \eqref{zw:1} and the following the cone property:
\begin{align*}
    \left|D_\eta P_{tx\xi}(s)\right|\le &\frac{L^2}{\lambda_b} \left(1+\frac{L^2}{\lambda_b}\right) \left(\left|D_\eta X_{tx\xi}(s)\right|+ \left\|D_\eta X_{t\xi}(s)\right\|_2 +\left|D_\eta v_{tx\xi}(s)\right|\right),
\end{align*}
which is a consequence of the assumptions (A2), \eqref{generic:condition:b}, \eqref{generic:condition:b'} and the estimate \eqref{zw:1} (see \eqref{prop:3_2'} and \eqref{prop:3_2} for details). The detailed proof is given in Appendix~\ref{pf:prop:3}.

Now we calculate the derivative  of $\Theta_{tx\mu}(s)$ with respect to the initial distribution $\mu\in\pr_2(\brn)$. In view of the processes $\Theta_{t\xi}$, $\Theta_{tx\mu}$ and $D_x\Theta_{tx\mu}$, we consider the following FBSDEs for $\bd \Theta_{t\xi}:=\left(\bd X_{t\xi},\bd P_{t\xi},\bd Q_{t\xi},\bd v_{t\xi}\right)$:
\small
\begin{align}
		\bd X_{t\xi}(s,y)=\ & \int_t^s \bigg\{D_x b(r, \theta_{t\xi}(r))\bd X_{t\xi}(r,y)+D_v b(r, \theta_{t\xi}(r))\bd v_{t\xi}(r,y) \notag \\
        &\qquad +\widetilde{\e}\bigg[ D_y \frac{db}{d\nu} (r,\theta_{t\xi}(r))\left(\widetilde{X_{ty\mu}}(r)\right) \widetilde{D_y X_{ty\mu}}(r) \bigg] \notag \\
        &\qquad +\widetilde{\e}\bigg[ D_y \frac{db}{d\nu} (r,\theta_{t\xi}(r))\left(\widetilde{X_{t\xi}}(r)\right) \widetilde{\bd X_{t\xi}}(r,y)\bigg]\bigg\}dr \notag\\
        &+\int_t^s \bigg\{ \sigma_1(r)\bd X_{t\xi}(r,y) + \widetilde{\e}\bigg[ D_y \frac{d\sigma_0}{d\nu} (r,\lr(X_{t\xi}(r)))\left(\widetilde{X_{ty\mu}}(r)\right) \widetilde{D_y X_{ty\mu}}(r) \bigg] \notag \\
        &\qquad\qquad +\widetilde{\e}\bigg[ D_y \frac{d\sigma_0}{d\nu} (r,\lr(X_{t\xi}(r)))\left(\widetilde{X_{t\xi}}(r)\right) \widetilde{\bd X_{t\xi}}(r,y)\bigg] \bigg\}dB(r); \notag \\
		\bd P_{t\xi}(s,y)=\ & -\int_s^T\bd Q_{t\xi}(r,y)dB(r)+ D_x^2 g (X_{t\xi}(T),\lr(X_{t\xi}(T)))^\top  \bd X_{t\xi}(T,y) \notag \\
        &+\widetilde{\e}\left[\left(D_y \frac{d}{d\nu}D_x g (X_{t\xi}(T),\lr(X_{t\xi}(T)))\left(\widetilde{X_{ty\mu}}(T)\right) \right)^\top \widetilde{D_y X_{ty\mu}}(T)\right] \notag \\
        &+\widetilde{\e}\left[\left(D_y \frac{d}{d\nu}D_x g (X_{t\xi}(T),\lr(X_{t\xi}(T)))\left(\widetilde{X_{t\xi}}(T)\right) \right)^\top \widetilde{\bd X_{t\xi}}(T,y)\right] \notag \\
		&+\int_s^T\Bigg\{ D_xb(r,\theta_{t\xi}(r))^\top \bd P_{t\xi}(r,y)+\sum_{j=1}^n\left(\sigma_1^j (r)\right)^\top \bd Q_{t\xi}^j(r,y) \notag \\
        &\qquad\qquad + \left(P_{t\xi}(r)\right)^\top \bigg\{D_x^2 b(r,\theta_{t\xi}(r)) \bd X_{t\xi}(r,y)+ D_vD_x b(r,\theta_{t\xi}(r)) \bd v_{t\xi}(r,y) \notag \\
        &\qquad\qquad\qquad\qquad\qquad +\widetilde{\e}\left[ D_y\frac{d}{d\nu}D_x b(r,\theta_{t\xi}(r))\left(\widetilde{X_{ty\mu}}(r)\right) \widetilde{D_y X_{ty\mu}}(r)\right] \notag\\
        &\qquad\qquad\qquad\qquad\qquad +\widetilde{\e}\left[ D_y\frac{d}{d\nu}D_x b(r,\theta_{t\xi}(r))\left(\widetilde{X_{t\xi}}(r)\right) \widetilde{\bd X_{t\xi}}(r,y)\right]\bigg\}  \notag\\
        &\qquad\qquad +D_x^2 f  (r,\theta_{t\xi}(r))^\top \bd X_{t\xi}(r,y) + D_vD_x f (r,\theta_{t\xi}(r))^\top  \bd v_{t\xi}(r,y) \notag\\
        &\qquad\qquad+\widetilde{\e}\left[\left(D_y \frac{d}{d\nu}D_x f  (r,\theta_{t\xi}(r))\left(\widetilde{X_{ty\mu}}(r)\right)\right)^\top   \widetilde{D_y X_{ty\mu}}(r) \right] \notag \\
        &\qquad\qquad+\widetilde{\e}\left[\left(D_y \frac{d}{d\nu}D_x f  (r,\theta_{t\xi}(r))\left(\widetilde{X_{t\xi}}(r)\right)\right)^\top   \widetilde{\bd X_{t\xi}}(r,y)\right]\Bigg\}dr, \quad (s,y)\in[t,T]\times\brn; \label{FB:xi_y}
\end{align}
\normalsize
with the following condition 
\small
\begin{align}
    0=\ & \widetilde{\e}\left[\left(D_y \frac{d}{d\nu}D_v f  (s,\theta_{t\xi}(s))\left(\widetilde{X_{ty\mu}}(s)\right)\right)^\top \widetilde{D_y X_{ty\mu}}(s)\right] \notag \\
    &+\widetilde{\e}\left[\left(D_y \frac{d}{d\nu}D_v f  (s,\theta_{t\xi}(s))\left(\widetilde{X_{t\xi}}(s)\right)\right)^\top  \widetilde{\bd X_{t\xi}}(s,y)\right] \notag \\
    &+D_x D_v f  (s,\theta_{t\xi}(s))^\top \bd X_{t\xi}(s,y) +D_v^2 f (s,\theta_{t\xi}(s))^\top  \bd v_{t\xi}(s,y)+D_v b(s,\theta_{t\xi}(s))^\top \bd P_{t\xi}(s,y) \notag\\
    &+ \left(P_{t\xi}(s)\right)^\top \bigg\{D_x D_v b(s,\theta_{t\xi}(s)) \bd X_{t\xi}(s,y)+ D_v^2 b(s,\theta_{t\xi}(s)) \bd v_{t\xi}(s,y) \notag \\
    &\qquad\qquad\qquad +\widetilde{\e}\left[ D_y\frac{d}{d\nu}D_v b(s,\theta_{t\xi}(s))\left(\widetilde{X_{ty\mu}}(s) \right) \widetilde{D_y X_{ty\mu}}(s) \right]\notag\\
    &\qquad\qquad\qquad +\widetilde{\e}\left[ D_y\frac{d}{d\nu}D_v b(s,\theta_{t\xi}(s))\left(\widetilde{X_{t\xi}}(s) \right) \widetilde{\bd X_{t\xi}}(s,y)\right]\bigg\}, \label{FB:xi_y_v}
\end{align}
\normalsize
where $D_y X_{ty\mu}=D_x X_{tx\mu}\big|_{x=y}$, and $\left(\widetilde{X_{ty\mu}}(s),\widetilde{X_{t\xi}}(s),\widetilde{D_y X_{ty\mu}}(s),\widetilde{\bd X_{t\xi}}(s,y)\right)$ is an independent copy of $(X_{ty\mu}(s),X_{t\xi}(s),D_y X_{ty\mu}(s),\bd X_{t\xi}(s,y))$. Then, we also consider the following FBSDEs for $D \Theta_{tx\mu}:=\left(D X_{tx\mu},D P_{tx\mu},D Q_{tx\mu},D v_{tx\mu}\right)$:
\small
\begin{align}
		D X_{tx\mu}(s,y)=\ & \int_t^s \bigg\{ D_x b(r, \theta_{tx\mu}(r))D X_{tx\mu}(r,y)+D_v b(r, \theta_{tx\mu}(r))D v_{tx\mu}(r,y) \notag \\
        &\qquad +\widetilde{\e}\bigg[ D_y \frac{db}{d\nu} (r,\theta_{tx\mu}(r))\left(\widetilde{X_{ty\mu}}(r)\right) \widetilde{D_y X_{ty\mu}}(r) \bigg] \notag \\
        &\qquad +\widetilde{\e}\bigg[ D_y \frac{db}{d\nu} (r,\theta_{tx\mu}(r))\left(\widetilde{X_{t\xi}}(r)\right) \widetilde{\bd X_{t\xi}}(r,y)\bigg] \bigg\}dr \notag \\
        &+\int_t^s \bigg\{ \sigma_1(r)D X_{tx\mu}(r,y) + \widetilde{\e}\bigg[ D_y \frac{d\sigma_0}{d\nu} (r,\lr(X_{t\xi}(r)))\left(\widetilde{X_{ty\mu}}(r)\right) \widetilde{D_y X_{ty\mu}}(r) \bigg] \notag \\
        &\quad\qquad + \widetilde{\e}\bigg[ D_y \frac{d\sigma_0}{d\nu} (r,\lr(X_{t\xi}(r)))\left(\widetilde{X_{t\xi}}(r)\right) \widetilde{\bd X_{t\xi}}(r,y)\bigg] \bigg\} dB(r), \notag \\
		D P_{tx\mu}(s,y)=\ & D_x^2 g (X_{tx\mu}(T),\lr(X_{t\xi}(T)))^\top  D X_{tx\mu}(T,y) \notag\\
        &+\widetilde{\e}\left[\left(D_y \frac{d}{d\nu}D_x g (X_{tx\mu}(T),\lr(X_{t\xi}(T)))\left(\widetilde{X_{ty\mu}}(T)\right) \right)^\top  \widetilde{D_y X_{ty\mu}}(T)\right] \notag \\
        &+\widetilde{\e}\left[\left(D_y \frac{d}{d\nu}D_x g (X_{tx\mu}(T),\lr(X_{t\xi}(T)))\left(\widetilde{X_{t\xi}}(T)\right) \right)^\top \widetilde{\bd X_{t\xi}}(T,y) \right] \notag \\
		&+\int_s^T\Bigg\{ D_xb(r,\theta_{tx\mu}(r))^\top D P_{tx\mu}(r,y)+\sum_{j=1}^n\left(\sigma_1^j (r)\right)^\top D Q_{tx\mu}^j(r,y) \notag \\
        &\quad\qquad + \left(P_{tx\mu}(r)\right)^\top \bigg\{\widetilde{\e}\left[ D_y\frac{d}{d\nu}D_x b(r,\theta_{tx\mu}(r))\left(\widetilde{X_{ty\mu}}(r)\right) \widetilde{D_y X_{ty\mu}}(r) \right] \notag \\
        &\quad\qquad\qquad\qquad\qquad +\widetilde{\e}\left[ D_y\frac{d}{d\nu}D_x b(r,\theta_{tx\mu}(r))\left(\widetilde{X_{t\xi}}(r)\right) \widetilde{\bd X_{t\xi}}(r,y) \right] \notag \\
        &\quad\qquad+D_x^2 b(r,\theta_{tx\mu}(r)) D X_{tx\mu}(r,y)+ D_vD_x b(r,\theta_{tx\mu}(r)) D v_{tx\mu}(r,y) \bigg\} \notag\\
        &\quad\qquad +\widetilde{\e}\left[\left(D_y \frac{d}{d\nu}D_x f  (r,\theta_{tx\mu}(r))\left(\widetilde{X_{ty\mu}}(r)\right)\right)^\top \widetilde{D_y X_{ty\mu}}(r)\right] \notag \\
        &\quad\qquad +\widetilde{\e}\left[\left(D_y \frac{d}{d\nu}D_x f  (r,\theta_{tx\mu}(r))\left(\widetilde{X_{t\xi}}(r)\right)\right)^\top \widetilde{\bd X_{t\xi}}(r,y) \right] \notag \\
        &+D_x^2 f  (r,\theta_{tx\mu}(r))^\top D X_{tx\mu}(r,y) + D_vD_x f (r,\theta_{tx\mu}(r))^\top  D v_{tx\mu}(r,y)\Bigg\}dr \notag \\
        &-\int_s^TD Q_{tx\mu}(r,y)dB(r), \quad (s,y)\in[t,T]\times\brn, \label{FB:mu_y}
\end{align}
\normalsize
with (as a consequence of taking the linear functional derivative in the first order condition)
\small
\begin{align}
    0=\ & \widetilde{\e}\left[\left(D_y \frac{d}{d\nu}D_v f  (s,\theta_{tx\mu}(s))\left(\widetilde{X_{ty\mu}}(s)\right)\right)^\top \widetilde{D_y X_{ty\mu}}(s) \right] \notag \\
    &+\widetilde{\e}\left[\left(D_y \frac{d}{d\nu}D_v f  (s,\theta_{tx\mu}(s))\left(\widetilde{X_{t\xi}}(s)\right)\right)^\top \widetilde{\bd X_{t\xi}}(s,y) \right] \notag \\
    &+D_x D_v f  (s,\theta_{tx\mu}(s))^\top D X_{tx\mu}(s,y) +D_v^2 f (s,\theta_{tx\mu}(s))^\top  D v_{tx\mu}(s,y)+D_v b(s,\theta_{tx\mu}(s))^\top D P_{tx\mu}(s,y) \notag\\
    &+ \left(P_{tx\mu}(s)\right)^\top \bigg\{\widetilde{\e}\left[ D_y\frac{d}{d\nu}D_v b(s,\theta_{tx\mu}(s))\left(\widetilde{X_{ty\mu}}(s) \right) \widetilde{D_y X_{ty\mu}}(s)\right] \notag \\
    &\quad\qquad\qquad\qquad +\widetilde{\e}\left[ D_y\frac{d}{d\nu}D_v b(s,\theta_{tx\mu}(s))\left(\widetilde{X_{t\xi}}(s) \right) \widetilde{\bd X_{t\xi}}(s,y) \right] \notag \\
    &\quad\qquad\qquad\qquad +D_x D_v b(s,\theta_{tx\mu}(s)) D X_{tx\mu}(s,y)+ D_v^2 b(s,\theta_{tx\mu}(s)) D v_{tx\mu}(s,y)\bigg\}. \label{FB:mu_y_condition}
\end{align}
\normalsize
Since the FBSDEs \eqref{FB:mu_y}-\eqref{FB:mu_y_condition} depend on $\xi$ only through its law $\mu$, it is legitimate to use the subscript $\mu$ in $D\Theta_{tx\mu}$. Finally, the following result shows the well-posedness of FBSDEs \eqref{FB:xi_y}-\eqref{FB:xi_y_v}  and \eqref{FB:mu_y}-\eqref{FB:mu_y_condition}, and also about the linear functional differentiability of $\Theta_{tx\mu}$ in $\mu$. 

\begin{theorem}\label{prop:5}
	Under Assumptions (A1)-(A3) and the validity of \eqref{lem:2_0}, there is a unique adapted solution $\bd \Theta_{t\xi}(s,y)$ of FBSDEs \eqref{FB:xi_y}-\eqref{FB:xi_y_v}, and a unique adapted solution $D\Theta_{t\xi}(s,y)$ of FBSDEs \eqref{FB:mu_y}-\eqref{FB:mu_y_condition}, satisfying the boundedness property:
    \begin{equation}\label{prop5_01}
        \begin{aligned}
            &\e\bigg[\sup_{t\le s\le T}\left|\left(\bd X_{t\xi}(s,y),\bd P_{t\xi}(s,y),\bd v_{t\xi}(s,y)\right)^\top\right|^2+\int_t^T\left|\bd Q_{t\xi}(s,y)\right|^2 ds\bigg]\le C, \\
            &\e\bigg[\sup_{t\le s\le T}\left|\left(D X_{tx\mu}(s,y),D P_{tx\mu}(s,y),D v_{tx\mu}(s,y)\right)^\top\right|^2+\int_t^T\left|D Q_{tx\mu}(s,y)\right|^2 ds\bigg]\le C,
        \end{aligned}
    \end{equation}
    and these component processes are jointly continuous in $(y,\xi,x,\mu)$. Moreover, for any $\xi,\eta\in L_{\f_t}^2$ independent of the Brownian motion $B$ such that $\lr(\xi)=\mu$, we have
	\begin{equation}\label{prop5_03}
		\begin{split}
			&\bd_\eta X_{t\xi}(s)=\widehat{\e}\left[\bd X_{t\xi}\left(s,\widehat{\xi}\right) \widehat{\eta}\right],\quad \bd_\eta v_{t\xi}(s)=\widehat{\e}\left[\bd v_{t\xi}\left(s,\widehat{\xi}\right) \widehat{\eta}\right],\\
			&\bd_\eta P_{t\xi}(s)=\widehat{\e}\left[\bd P_{t\xi}\left(s,\widehat{\xi}\right) \widehat{\eta}\right],\quad \bd_\eta Q_{t\xi}(s)=\widehat{\e}\left[\bd Q_{t\xi}\left(s,\widehat{\xi}\right) \widehat{\eta}\right],\quad s\in[t,T];
		\end{split}
	\end{equation}
	and 
	\begin{equation}\label{prop5_04}
		\begin{split}
			&D_\eta X_{tx\xi}(s)=\widehat{\e}\left[D X_{tx\mu}\left(s,\widehat{\xi}\right) \widehat{\eta}\right],\quad D_\eta v_{tx\xi}(s)=\widehat{\e}\left[D v_{tx\mu}\left(s,\widehat{\xi}\right) \widehat{\eta}\right],\\
			&D_\eta P_{tx\xi}(s)=\widehat{\e}\left[D P_{tx\mu}\left(s,\widehat{\xi}\right) \widehat{\eta}\right],\quad D_\eta Q_{tx\xi}(s)=\widehat{\e}\left[D Q_{tx\mu}\left(s,\widehat{\xi}\right) \widehat{\eta}\right],\quad s\in[t,T],
		\end{split}
	\end{equation}
	where $(\widehat{\xi},\widehat{\eta})$ is an independent copy of $(\xi,\eta)$. That is, the mapping 	
	\begin{align*}
		\mu \mapsto (X_{tx\mu}(s),v_{tx\mu}(s),P_{tx\mu}(s),Q_{tx\mu}(s))
	\end{align*}
	is linearly functionally differentiable, with the respective derivatives being
	\begin{align*}
		&D_y\frac{d}{d\nu} X_{tx\mu}(s)(y)=D X_{tx\mu}(s,y),\quad D_y\frac{d}{d\nu} v_{tx\mu}(s)(y)=D v_{tx\mu}(s,y),\\
		&D_y\frac{d}{d\nu} P_{tx\mu}(s)(y)=D P_{tx\mu}(s,y),\quad D_y\frac{d}{d\nu} Q_{tx\mu}(s)(y)=D Q_{tx\mu}(s,y).
	\end{align*}
\end{theorem}

The proof of Theorem~\ref{prop:5} is similar to that for \cite[Theorem 4.4]{AB11}, but for the sake of reference for readers, we also put that in Appendix~\ref{pf:prop:5}.

\section{Hessian flows of the process $\Theta_{tx\mu}$}\label{sec:second_deri}

In this section, we give the second order derivatives (or
Hessian flows) of the process $\Theta_{tx\mu}$. We further adopt the following regularity-enhanced version of Assumptions (A1) and (A2):

\textbf{(A1')} The coefficients $b$ and $\sigma$ satisfy (A1). 
The derivatives $D_y^2 \frac{db}{d\nu}$ and $D_y^2 \frac{d\sigma_0}{d\nu}$ exist, and they are continuous in all their arguments and are globally bounded in norm by $L$. The following derivatives exist, and they are continuous in all their arguments:
\begin{align*}
	D_x^3 b,\ D_x^2 D_v b, \ D_x D_v^2 b,\ D_v^3b, D_y^2 \frac{d}{d\nu}D_x b,\ D_y^2 \frac{d}{d\nu}D_v b
\end{align*}
and they are globally bounded in norm by $\frac{L}{1+|x|+|v|+W_2(m,\delta_0)}$.

\textbf{(A2')} The functionals $f$ and $g$ satisfy (A2). The following derivatives exist, and they are continuous in all their arguments and are globally bounded by $L$ in norm:
\begin{align*}
	D_x^3 f,\ D_x^2 D_v f, \ D_x D_v^2 f,\ D_v^3f, D_y^2 \frac{d}{d\nu}D_x f,\ D_y^2 \frac{d}{d\nu}D_v f,\ D_x^3 g,\ D_y^2 \frac{d}{d\nu}D_x g.
\end{align*}

We first characterize the second order derivatives of $\Theta_{tx\mu}(s)$ in $x$ as a Hessian flow. From FBSDEs \eqref{FB:x}-\eqref{FB:x_condition} and Theorem~\ref{prop:2}, the second order derivatives fulfill the following system of FBSDEs. 
\small
\begin{align}
		D_x^2 X_{tx\mu}(s)=\ & \int_t^s \left[D_x b\left(r,\theta_{tx\mu}(r)\right)D_x^2 X_{tx\mu}(r)+ D_v b\left(r,\theta_{tx\mu}(r)\right)D_x^2 v_{tx\mu}(r) \right] dr \notag \\
        & + \int_t^s \left[\begin{pmatrix}
            D_x^2 b & D_xD_v b \\ D_vD_x b & D_v^2 b
        \end{pmatrix}\left(r,\theta_{tx\mu}(r)\right)\right]\begin{pmatrix}
            D_x X_{tx\mu}(r) \\ D_x X_{tx\mu}(r)
        \end{pmatrix}^{\otimes2} dr \notag \\
        & +\int_t^s \sigma_1(r) D_x^2 X_{tx\mu}(r) dB(r); \notag \\
		\text{and}\quad D_x^2 P_{tx\mu}(s)=\ & D_x^2 g\left(X_{tx\mu}(T),\lr(X_{t\xi}(T))\right)^\top D_x^2 X_{tx\mu}(T)+D_x^3 g\left(X_{tx\mu}(T),\lr(X_{t\xi}(T))\right)\left(D_xX_{tx\mu}(T)\right)^{\otimes2} \notag \\
        & +\int_s^T \bigg\{ D_x b\left(r,\theta_{tx\mu}(r)\right)^\top D_x^2 P_{tx\mu}(r)+\sum_{j=1}^n\left(\sigma_1^j(r)\right)^\top D_x^2 Q_{tx\mu}^j(r)  \notag \\
        &\qquad\qquad +\left(P_{tx\mu}(r)\right)^\top \Big[D_x^2 b\left(r,\theta_{tx\mu}(r)\right)D_x^2 X_{tx\mu}(r) + D_v D_x b\left(r,\theta_{tx\mu}(r)\right)D_x^2 v_{tx\mu}(r)\Big] \notag \\
        &\qquad\qquad +\left[D_x^2 f (r,\theta_{tx\mu}(r)) \right]^\top D_x^2 X_{tx\mu}(r) + \left[D_vD_x f (r,\theta_{tx\mu}(r))\right]^\top D_x^2 v_{tx\mu}(r) \bigg\} dr\notag \\
        &+2\int_s^T \left(D_x P_{tx\mu}(r)\right)^\top \left[\left(D_x^2b,D_vD_x b\right)\left(r,\theta_{tx\mu}(r)\right)\right]\begin{pmatrix}
            D_x X_{tx\mu}(r) \\ D_x v_{tx\mu}(r)
        \end{pmatrix} dr \notag \\
        & +\int_s^T \left(P_{tx\mu}(r)\right)^\top \left[\begin{pmatrix}
            D_x^3 b & D_xD_vD_x b \\ D_vD_x^2 b & D_v^2 D_x b
        \end{pmatrix}\left(r,\theta_{tx\mu}(r)\right)\right]\begin{pmatrix}
            D_xX_{tx\mu}(r) \\ D_xv_{tx\mu}(r)
        \end{pmatrix}^{\otimes2} dr \notag \\
        & +\int_s^T \left[\begin{pmatrix}
            D_x^3 f & D_xD_vD_x f \\ D_vD_x^2 f & D_v^2 D_x f
        \end{pmatrix}\left(r,\theta_{tx\mu}(r)\right)\right]\begin{pmatrix}
            D_xX_{tx\mu}(r) \\ D_xv_{tx\mu}(r)
        \end{pmatrix}^{\otimes2} dr \notag \\
        &-\int_s^T D_x^2 Q_{tx\mu} (r)dB(r),\quad s\in[t,T], \label{FB:xx}
\end{align}
\normalsize
with ( via taking the derivative with respect to $x$ in  Condition \eqref{FB:x_condition})
\begin{align}
	0=\ &\left[D_xD_v f (s,\theta_{tx\mu}(s)) \right]^\top D_x^2 X_{tx\mu}(s) + \left[D_v^2 f (s,\theta_{tx\mu}(s))\right]^\top D_x^2 v_{tx\mu}(s)+D_v b(s,\theta_{tx\mu}(s))^\top D_x^2 P_{tx\mu}(s) \notag \\
    &+\left(P_{tx\mu}(s)\right)^\top \left[D_xD_v b\left(s,\theta_{tx\mu}(s)\right)D_x^2 X_{tx\mu}(s) + D_v^2  b\left(s,\theta_{tx\mu}(s)\right)D_x^2 v_{tx\mu}(s)\right] \notag \\
    &+ \left[\begin{pmatrix}
            D_x^2D_v f & D_xD_v^2 f \\ D_vD_xD_v f & D_v^3 f
        \end{pmatrix}\left(s,\theta_{tx\mu}(s)\right)\right]\begin{pmatrix}
            D_xX_{tx\mu}(s) \\ D_xv_{tx\mu}(s)
        \end{pmatrix}^{\otimes2} \notag \\
    &+ 2 \left(D_x P_{tx\mu}(s)\right)^\top \left[\left(D_x D_v b,D_v^2 b\right)\left(s,\theta_{tx\mu}(s)\right)\right]\begin{pmatrix}
            D_x X_{tx\mu}(s) \\ D_x v_{tx\mu}(s)
        \end{pmatrix} \notag\\
    &+\left(P_{tx\mu}(s)\right)^\top \left[\begin{pmatrix}
            D_x^2 D_v b & D_xD_v^2 b \\ D_vD_xD_v b & D_v^3 b
        \end{pmatrix}\left(r,\theta_{tx\mu}(r)\right)\right]\begin{pmatrix}
            D_xX_{tx\mu}(r) \\ D_xv_{tx\mu}(r)
        \end{pmatrix}^{\otimes2}. \label{FB:xx_condition}
\end{align}
In view of the $L^4$-norm boundedness of $D_x \Theta_{tx\mu}$ in \eqref{thm2_1} of Theorem~\ref{prop:2}, the proof of the following result is similar to that leading to statements in the last section (such as Lemma~\ref{lem:3} and Theorem~\ref{prop:3}), and we omit it here. 

\begin{theorem}\label{prop:8}
	Under Assumptions (A1), (A2') and (A3), and the validity of \eqref{lem:2_0}, there is a unique adapted solution $D_x^2 \Theta_{tx\mu}:=\left(D_x^2 X_{tx\mu},D_x^2 P_{tx\mu},D_x^2 Q_{tx\mu},D_x^2 v_{tx\mu}\right)$ of FBSDEs \eqref{FB:xx}-\eqref{FB:xx_condition}, and the component processes of $D_x^2 \Theta_{tx\mu}(s)$ are the respective second order G\^ateaux derivatives of the corresponding component processes of $\Theta_{tx\mu}(s)$ with respect to $x$. Moreover, they satisfy the following boundedness property	
    \begin{align}\label{boundedness-property}
		\e\bigg[\sup_{t\le s\le T}\left|\left(D_x^2 X_{tx\mu}(s),D_x^2 P_{tx\mu}(s),D_x^2 v_{tx\mu}(s)\right)^\top\right|^2+ \int_t^T\left|D_x^2 Q_{tx\mu}(s)\right|^2 ds\bigg]\le C,
	\end{align}
    and these component processes are continuous in $(x,\mu)$. 
\end{theorem}

We next give the derivatives in $y$ of the processes $\bd \Theta_{t\xi}(s,y)$ and $D \Theta_{tx\mu}(s,y)$ defined in FBSDEs \eqref{FB:xi_y}-\eqref{FB:xi_y_v} and FBSDEs FBSDEs \eqref{FB:mu_y}-\eqref{FB:mu_y_condition}, respectively. The proof of the following result is given in Appendix~\ref{pf:prop:9}.

\begin{theorem}\label{prop:9}
	Under Assumptions (A1'), (A2') and (A3), and the validity of \eqref{lem:2_0}, the component processes of $\bd \Theta_{t\xi}(s,y)$ are G\^ateaux differentiable in $y$, and the G\^ateaux derivatives 
    \begin{align*}
        D_y\bd \Theta_{t\xi}:=\left(D_y\bd X_{t\xi},D_y\bd P_{t\xi},D_y\bd Q_{t\xi},D_y\bd v_{t\xi}\right)
    \end{align*}
    satisfy the boundedness property
    \begin{align*}
    	\e\bigg[\sup_{t\le s\le T}\left|\left(D_y \bd X_{t\xi}(s,y),D_y\bd P_{t\xi}(s,y),D_y \bd v_{t\xi}(s,y)\right)^\top\right|^2+\int_t^T\left|D_y\bd Q_{t\xi}(s,y)\right|^2 ds\bigg]\le C,
    \end{align*}
    and they are continuous in $(\xi,y)$. Moreover, the component processes $D \Theta_{tx\mu}(s,y)$ are G\^ateaux differentiable in $y$, and the corresponding G\^ateaux derivatives $D_z\dr Y_{tx\mu}(s,z)$, $D_z\dr p_{tx\mu}(s,z)$ and $D_yD \Theta_{tx\mu}:=(D_yD X_{tx\mu}, D_yD P_{tx\mu}, D_yD Q_{tx\mu}, D_yD v_{tx\mu})$ satisfy the boundedness property
    \begin{align*}
        \e\bigg[\sup_{t\le s\le T}\left|\left(D_y D X_{tx\mu}(s,y),D_y D P_{tx\mu}(s,y),D_y D v_{tx\mu}(s,y)\right)^\top\right|^2+ \int_t^T\left|D_y D Q_{tx\mu}(s,y)\right|^2 ds\bigg]\le C,
    \end{align*}
    and they are continuous in $(x,\mu,y)$.  
\end{theorem}

\section{Regularity of the value functional and solution of the master equation}\label{sec:V}

We here apply the results in Sections~\ref{sec:distribution} and \ref{sec:second_deri} to derive the spatial, distributional and temporal regularities of the value functional $V$ defined in \eqref{intro_4} as a functional in $(t,x,\mu)\in[0,T]\times\brn\times\pr_2(\brn)$. We begin with the spatial regularity of $V$ as follows. 

\begin{theorem}\label{prop:6}
	Under Assumptions (A1)-(A3) and the validity of \eqref{lem:2_0}, the value functional $V$ is $C^2$ in $x$ with the spatial derivatives
	\begin{align}
		D_x V(t,x,\mu)  =P_{tx\mu}(t), \qquad D_x^2 V(t,x,\mu)  =D_x P_{tx\mu}(t) \label{prop6_03},
	\end{align}
    and they also satisfy the growth conditions	
    \begin{equation}\label{prop6_04}
    	\begin{split}
    		|V(t,x,\mu)| \le \ & C\left[1+|x|^2+W^2_2(\mu,\delta_0)\right];\\
    		|D_x V(t,x,\mu)| \le\ & \left(1-\frac{L^3}{2\lambda_v\lambda_b}\right)^{-1}\frac{L^2}{\lambda_b}\left(1+\frac{L}{2\lambda_v}\right)\left[1+|x|+W_2(\mu,\delta_0)\right];\\
            \left|D_x^2 V(t,x,\mu)\right|\le\ & C.
    	\end{split}
    \end{equation}
    Meanwhile, the following continuity  holds true:
    \begin{align}
    	\left|D_x V(t,x',\mu')-D_x V(t,x,\mu)\right|\le C\left[|x'-x|+W_2(\mu,\mu')\right], \label{prop6_05}
    \end{align}
    and $D_x^2 V$ is continuous in $(x,\mu)$. 
\end{theorem}

\begin{proof}
The growth condition of $V$ is a direct consequence of Assumption (A2), \eqref{intro_4'}, and Theorems~\ref{lem:2} and \ref{lem:5}. We now prove \eqref{prop6_03}. Since $v_{tx\mu}$ and $v_{tx'\mu}$ are optimal controls for Problem \eqref{intro_1'} with the initial conditions $x$ and $x'$,  respectively, we have  from the definition of $V$ in \eqref{intro_4},
\begin{equation}\label{prop6_1}
	\begin{split}
		&J_{tx'}\left(v_{tx'\mu};\lr(X_{t\xi}(s)),t\le s\le T\right)-J_{tx}\left(v_{tx'\mu};\lr(X_{t\xi}(s)),t\le s\le T\right) \\
		\le\ & V(t,x',\mu)-V(t,x,\mu)\\
		\le\ &  J_{tx'}\left(v_{tx\mu};\lr(X_{t\xi}(s)),t\le s\le T\right)-J_{tx}\left(v_{tx\mu};\lr(X_{t\xi}(s)),t\le s\le T\right).
	\end{split}
\end{equation}
From Assumption (A2), we deduce that
\begin{align}
		&J_{tx'}\left(v_{tx\mu};\lr(X_{t\xi}(s)),t\le s\le T\right)-J_{tx}\left(v_{tx\mu};\lr(X_{t\xi}(s)),t\le s\le T\right) \notag \\
		=\ & \e\bigg[\int_t^T \left[f\left(s,X_{tx'\mu}^{v_{tx\mu}}(s),\lr(X_{t\xi}(s)),v_{tx\mu}(s)\right)-f\left(s,X_{tx\mu}(s),\lr(X_{t\xi}(s)),v_{tx\mu}(s)\right)\right]ds \notag \\
		&\quad +g\left(X^{v_{tx\mu}}_{tx'\mu}(T),\lr(X_{t\xi}(T))\right)-g\left(X_{tx\mu}(T),\lr(X_{t\xi}(T))\right)\bigg] \notag \\
		\le\ & \e\bigg[\int_t^T D_x f (s,X_{tx\mu}(s),\lr(X_{t\xi}(s)),v_{tx\mu}(s))^\top \left(X_{tx'\mu}^{v_{tx\mu}}(s)-X_{tx\mu}(s)\right) ds \notag \\
		&\qquad +D_x g (X_{tx\mu}(T),\lr(X_{t\xi}(T)))^\top \left(X_{tx'\mu}^{v_{tx\mu}}(T)-X_{tx\mu}(T)\right) \bigg] \notag \\
		& +C(L,T)\e\bigg[\sup_{t\le s\le T}\left|X_{tx'\mu}^{v_{tx\mu}}(s)-X_{tx\mu}(s)\right|^2\bigg], \label{prop6_2}
\end{align}
where $X_{tx'\mu}^{v_{tx\mu}}$ is the state process corresponding to  the initial data $x'$ and the control $v_{tx\mu}$:
\begin{align*}
	X_{tx'\mu}^{v_{tx\mu}}(s)=x'&+\int_t^s b\left(r,X_{tx'\mu}^{v_{tx\mu}}(r),\lr(X_{t\xi}(r)),v_{tx\mu}(r)\right)dr+\int_t^s\sigma\left(r,X_{tx'\mu}^{v_{tx\mu}}(r),\lr(X_{t\xi}(r))\right)dB(r).
\end{align*}
Recalling the backward equation of FBSDEs \eqref{intro_3} for the process $P_{tx\mu}$, by applying It\^o's formula to $(P_{tx\mu}(s))^\top \left(X_{tx'\mu}^{v_{tx\mu}}(s)-X_{tx\mu}(s)\right)$, we can deduce that
\begin{align*}
    &\e\left[D_x g (X_{tx\mu}(T),\lr(X_{t\xi}(T)))^\top \left(X_{tx'\mu}^{v_{tx\mu}}(T)-X_{tx\mu}(T)\right)\right]-(P_{tx\mu}(t))^\top (x'-x)\\
    =\ & \e\int_t^T\bigg\{ \left(P_{tx\mu}(s)\right)^\top \bigg[ b\left(s,X_{tx'\mu}^{v_{tx\mu}}(s),\lr(X_{t\xi}(s)),v_{tx\mu}(s)\right)-b\left(s,X_{tx\mu}(s),\lr(X_{t\xi}(s)),v_{tx\mu}(s)\right)\\
    &\quad\qquad\qquad\qquad\qquad -D_x b\left(s,X_{tx\mu}(s),\lr(X_{t\xi}(s)),v_{tx\mu}(s)\right) \left(X_{tx'\mu}^{v_{tx\mu}}(s)-X_{tx\mu}(s)\right) \bigg]\\
    &\quad\qquad - D_x f \left(s,X_{tx\mu}(s),\lr(X_{t\xi}(s)),v_{tx\mu}(s)\right)^\top \left(X_{tx'\mu}^{v_{tx\mu}}(s)-X_{tx\mu}(s)\right) \bigg\}ds, 
\end{align*}
and therefore, 
\begin{align}
    &\e\bigg[\int_t^T D_x f \left(s,X_{tx\mu}(s),\lr(X_{t\xi}(s)),v_{tx\mu}(s)\right)^\top \left(X_{tx'\mu}^{v_{tx\mu}}(s)-X_{tx\mu}(s)\right) ds \notag \\
    &\qquad+D_x g (X_{tx\mu}(T),\lr(X_{t\xi}(T)))^\top \left(X_{tx'\mu}^{v_{tx\mu}}(T)-X_{tx\mu}(T)\right)\bigg] \notag \\
    =\ & (P_{tx\mu}(t))^\top (x'-x) \notag \\
    &+ \e\int_t^T\bigg\{ \left(P_{tx\mu}(s)\right)^\top \bigg[ b\left(s,X_{tx'\mu}^{v_{tx\mu}}(s),\lr(X_{t\xi}(s)),v_{tx\mu}(r)\right)-b\left(s,X_{tx\mu}(s),\lr(X_{t\xi}(s)),v_{tx\mu}(s)\right) \notag\\
    &\quad\qquad\qquad\qquad\qquad -D_x b\left(s,X_{tx\mu}(s),\lr(X_{t\xi}(s)),v_{tx\mu}(s)\right) \left(X_{tx'\mu}^{v_{tx\mu}}(s)-X_{tx\mu}(s)\right) \bigg] \bigg\}ds. \label{prop6_2.1}
\end{align}
From Condition \eqref{generic:condition:b} and the estimate \eqref{zw:1}, we know that
\begin{align}
    &\bigg|\left(P_{tx\mu}(s)\right)^\top \bigg[ b\left(s,X_{tx'\mu}^{v_{tx\mu}}(s),\lr(X_{t\xi}(s)),v_{tx\mu}(r)\right)-b\left(s,X_{tx\mu}(s),\lr(X_{t\xi}(s)),v_{tx\mu}(s)\right) \notag\\
    &\quad\qquad\qquad -D_x b\left(s,X_{tx\mu}(s),\lr(X_{t\xi}(s)),v_{tx\mu}(s)\right) \left(X_{tx'\mu}^{v_{tx\mu}}(s)-X_{tx\mu}(s)\right) \bigg] \bigg| \notag \\
    \le\ & \frac{L^2L_b^0}{\lambda_b}  \left|X_{tx'\mu}^{v_{tx\mu}}(s)-X_{tx\mu}(s)\right|^2. \label{prop6_2.2}
\end{align}
Conbining \eqref{prop6_2}, \eqref{prop6_2.1} and \eqref{prop6_2.2}, we have
\begin{align}
		&J_{tx'}\left(v_{tx\mu};\lr(X_{t\xi}(s)),t\le s\le T\right)-J_{tx}\left(v_{tx\mu};\lr(X_{t\xi}(s)),t\le s\le T\right) \notag \\
		\le\ & (P_{tx\mu}(t))^\top (x'-x)+C(L,T,\lambda_b)\e\bigg[\sup_{t\le s\le T}\left|X_{tx'\mu}^{v_{tx\mu}}(s)-X_{tx\mu}(s)\right|^2\bigg]. \label{prop6_2.3}
\end{align}
Note the fact that $X_{tx\mu}=X_{tx\mu}^{v_{tx\mu}}$, then, from the Lipschitz continuity of $b$ and $\sigma$ in the spatial argument, by using Gr\"onwall's inequality, we have
\begin{equation}\label{prop6_3}
	\e\bigg[\sup_{t\le s\le T}\left|X_{tx'\mu}^{v_{tx\mu}}(s)-X_{tx\mu}(s)\right|^2\bigg]=\e\bigg[\sup_{t\le s\le T}\left|X_{tx'\mu}^{v_{tx\mu}}(s)-X_{tx\mu}^{v_{tx\mu}}(s)\right|^2\bigg]\le C(L,T)|x'-x|^2.
\end{equation}
Substituting \eqref{prop6_3} into \eqref{prop6_2.3}, we obtain that
\begin{equation}\label{prop6_5}
	\begin{split}
		&J_{tx'}\left(v_{tx\mu};\lr(X_{t\xi}(s)),t\le s\le T\right)-J_{tx}\left(v_{tx\mu};\lr(X_{t\xi}(s)),t\le s\le T\right) \\
		\le\ &( P_{tx\mu}(t))^\top  (x'-x) +C(L,T,\lambda_b)|x'-x|^2.
	\end{split}
\end{equation}
Similar to \eqref{prop6_5}, we can also obtain the lower bound
\begin{equation}\label{prop6_6}
	\begin{split}
		&J_{tx'}\left(v_{tx'\mu};\lr(X_{t\xi}(s)),t\le s\le T\right)-J_{tx}\left(v_{tx'\mu};\lr(X_{t\xi}(s)),t\le s\le T\right) \\
		\geq\ & (P_{tx'\mu}(t))^\top  (x'-x) -C(L,T,\lambda_b)|x'-x|^2.
	\end{split}
\end{equation}
From \eqref{lem5_2}, we have
\begin{equation}\label{prop6_7}
	|P_{tx'\mu}(t)-P_{tx\mu}(t)|\le C|x'-x|.
\end{equation}
Combining \eqref{prop6_1}, \eqref{prop6_5}, \eqref{prop6_6} and \eqref{prop6_7}, we deduce that
\begin{equation*}
	\left|V(t,x',\mu)-V(t,x,\mu)-(P_{tx\mu}(t))^\top (x'-x)\right|\le C|x'-x|^2,
\end{equation*} 
from which we obtain $D_x V(t,x,\mu)  =P_{tx\mu}(t)$. Then, from Theorem~\ref{prop:2}, we know that $D_x^2 V(t,x,\mu)  =D_x P_{tx\mu}(t)$. Estimate \eqref{prop6_04} follows as an immediate consequence of \eqref{lem5_1}, \eqref{thm2_1} and \eqref{lem:5_11},  and Estimate \eqref{prop6_05} is a consequence of \eqref{lem5_2}. The continuity of $D_x^2 V$ in $(x,\mu)$ is a direct consequence of Theorem~\ref{prop:2}.   
\end{proof}

The proof of Theorem~\ref{prop:6} also relies on the cone property \eqref{zw:1}. In particular, the growth constant of the derivative $D_x V$ in \eqref{prop6_04} need to be more precise; indeed, from \eqref{prop6_05} and Proposition~\ref{prop:cone}, we know that the map:
\begin{align*}
    \brd\ni v\mapsto L\left(t,x,\mu,v,D_xV(t,x,\mu)\right)\in\br
\end{align*}
has a unique minimizer $\hv\left(t,x,\mu,D_xV(t,x,\mu)\right)$, and we have 
\begin{align*}
    H\left(t,x,\mu,D_xV(t,x,\mu) \right)=\ & L\left(t,x,\mu,\hv\left(t,x,\mu,D_xV(t,x,\mu)\right),D_xV(t,x,\mu)\right),\\
    D_p H \left(t,x,\mu,D_xV(t,x,\mu) \right)=\ & b \left(t,x,\mu,\hv\left(t,x,\mu,D_xV(t,x,\mu)\right)\right),\\
    D_x H \left(t,x,\mu,D_xV(t,x,\mu) \right)=\ & D_x b \left(t,x,\mu,\hv\left(t,x,\mu,D_xV(t,x,\mu)\right)\right)^\top D_xV(t,x,\mu)\\
    &+ D_x f \left(t,x,\mu,\hv\left(t,x,\mu,D_xV(t,x,\mu)\right)\right).
\end{align*}
Besides, as a consequence of \eqref{prop6_03}, we also have the following characterization of $P_{t\xi}(t)$:	
\begin{align}\label{rk_1}
	P_{t\xi}(t)=P_{tx\mu}\big|_{x=\xi}=D_x V(t,x,\mu)\big|_{x=\xi}.
\end{align}
We next give the regularity of $V$ with respect to the distribution argument $\mu$. 

\begin{theorem}\label{prop:7}
	Under Assumptions (A1'), (A2') and (A3), and the validity of \eqref{lem:2_0}, the value functional $V$ is linearly functionally differentiable in $\mu$ with the corresponding derivatives	
	\begin{align}
			D_y\frac{dV}{d\nu}(t,x,\mu)(y)=\ &\int_t^T \e\bigg\{\left[D_x f (s,\theta_{tx\mu}(s)) \right]^\top D X_{tx\mu}(s,y)+ \left[D_v f)^\top  (s,\theta_{tx\mu}(s))  \right]^\top D v_{tx\mu}(s,y) \notag \\
			&\quad\qquad +\widetilde{\e}\bigg[\left[D_y \frac{df}{d\nu} (s,\theta_{tx\mu}(s))\left(\widetilde{X_{ty\mu}}(s)\right) \right]^\top \widetilde{D_y X_{ty\mu}}(s)  \notag  \\
			&\quad\qquad\qquad +\left[D_y \frac{df}{d\nu} (s,\theta_{tx\mu}(s))\left(\widetilde{X_{t\xi}}(s)\right) \right]^\top \widetilde{\bd X_{t\xi}}(s,y) \bigg]  \bigg\} ds  \notag  \\
			& + \e\bigg\{\left[D_x g (X_{tx\mu}(T),\lr(X_{t\xi}(T)))  \right]^\top D X_{tx\mu}(T,y) \notag \\
			&\qquad +\widetilde{\e}\bigg[\left[D_y \frac{dg}{d\nu} (X_{tx\mu}(T),\lr(X_{t\xi}(T)))\left(\widetilde{X_{ty\mu}}(T)\right)  \right]^\top \widetilde{D_y X_{ty\mu}}(T)  \notag  \\
			&\qquad\qquad +\left[D_y \frac{dg}{d\nu} (X_{tx\mu}(T),\lr(X_{t\xi}(T)))\left(\widetilde{X_{t\xi}}(T)\right) \right]^\top \widetilde{\bd X_{t\xi}}(T,y) \bigg] \bigg\}, \label{prop7_01}
	\end{align}
	and
	\begin{align}
			&D_y^2\frac{dV}{d\nu}(t,x,\mu)(y) \notag \\
			=\ & \int_t^T \e\bigg\{ \left[D_x f (s, \theta_{tx\mu}(s))  \right]^\top D_yD X_{tx\mu}(s,y)+ \left[D_v f (s, \theta_{tx\mu}(s))  \right]^\top D_y D v_{tx\mu}(s,y) \notag \\
			&\quad\qquad +\widetilde{\e}\bigg[\left[D_y \frac{df}{d\nu} (s, \theta_{tx\mu}(s))\left(\widetilde{X_{ty\mu}}(s)\right)  \right]^\top \widetilde{D_y^2X_{ty\mu}}(s)   \notag \\
			&\quad\qquad\qquad +\left(\widetilde{D_y X_{ty\mu}}(s) \right)^\top \left[ D_y^2 \frac{df}{d\nu}(s, \theta_{tx\mu}(s))\left(\widetilde{X_{ty\mu}}(s)\right)\right] \widetilde{D_yX_{ty\mu}}(s)   \notag \\
			&\quad\qquad\qquad +\left[D_y \frac{df}{d\nu} (s, \theta_{tx\mu}(s))\left(\widetilde{X_{t\xi}}(s)\right) \right]^\top \widetilde{D_y \bd X_{t\xi}}(s,y) \bigg]  \bigg\} ds  \notag \\
			& + \e\bigg\{\left[D_x g (X_{tx\mu}(T),\lr(X_{t\xi}(T)))  \right]^\top D_y D X_{tx\mu}(T,y) \notag \\
			&\qquad +\widetilde{\e}\bigg[\left[D_y \frac{dg}{d\nu} (X_{tx\mu}(T),\lr(X_{t\xi}(T)))\left(\widetilde{X_{ty\mu}}(T)\right) \right]^\top \widetilde{D_y^2X_{ty\mu}}(T)  \notag \\
			&\qquad\qquad + \left(\widetilde{D_yX_{ty\mu}}(T) \right)^\top \left[ D_y^2 \frac{dg}{d\nu}(X_{tx\mu}(T),\lr(X_{t\xi}(T)))\left(\widetilde{X_{ty\mu}}(T)\right) \right] \widetilde{D_yX_{ty\mu}}(T)  \notag  \\
			&\qquad\qquad + \left[D_y \frac{dg}{d\nu} (X_{tx\mu}(T),\lr(X_{t\xi}(T)))\left(\widetilde{X_{t\xi}}(T)\right) \right]^\top \widetilde{D_y \bd X_{t\xi}}(T,y) \bigg]\bigg\}, \label{prop7_02}
	\end{align}
	where $\left(D_y X_{ty\mu},D_y^2 X_{ty\mu}\right)=\left(D_x X_{tx\mu},D_x^2 X_{tx\mu} \right)\!\Big|_{x=y}$, and 
	\begin{align*}
		\left(\widetilde{X_{ty\mu}}(s),\widetilde{X_{t\xi}}(s),\widetilde{D_y X_{ty\mu}}(s),\widetilde{D_y^2 X_{ty\mu}}(s),\widetilde{\bd X_{t\xi}}(s,y),\widetilde{D_y\bd X_{t\xi}}(s,y)\right)
	\end{align*}	
	is an independent copy of 
	\begin{align*}
		\left(X_{ty\mu}(s),X_{t\xi}(s),D_y X_{ty\mu}(s),D_y^2 X_{ty\mu}(s),\bd X_{t\xi}(s,y),D_y\bd X_{t\xi}(s,y)\right).
	\end{align*}
	Moreover, the derivatives satisfy the following growth condition
	\begin{align}
		&\left|D_y\frac{dV}{d\nu}(t,x,\mu)(y)\right|+\left|D_y^2\frac{dV}{d\nu}(t,x,\mu)(y)\right|\le C (1+|x|+W_2(\mu,\delta_0)),\label{prop7_03}
	\end{align}
	and $(D_y\frac{dV}{d\nu},D_y^2\frac{dV}{d\nu})$ are continuous in $(x,\mu,y)$. 	
\end{theorem}

\begin{proof}
The equality \eqref{prop7_01} is a consequence of Theorems~\ref{lem:4} and \ref{prop:3}, Lemma~\ref{prop:4} and Theorem~\ref{prop:5}, and the proof is similar to that for \cite[Theorem 5.2]{AB11}, and we omit then. The equality \eqref{prop7_02} is a direct consequence of \eqref{prop7_01} and Theorem~\ref{prop:9}. Then, the estimate \eqref{prop7_03} is a direct consequence of \eqref{thm2_1}, \eqref{prop5_01}, and Theorems~\ref{prop:8} and \ref{prop:9}.    
\end{proof}

We finally provide the temporal regularity of $V$, whose proof is similar to that for \cite[Theorem 5.3]{AB12}, and we put it in Appendix~\ref{pf:prop:10}.

\begin{theorem}\label{prop:10}
	Under Assumptions (A1'), (A2') and (A3), and the validity of \eqref{lem:2_0} and \eqref{lem:2_12}, the value functional $V$ is $C^1$ in $t$ with the temporal derivative
    \begin{equation}\label{prop10_01}
        \begin{aligned}
            \dd_t V(t,x,\mu)=\ & -\frac{1}{2}\text{Tr}\left[\left(\sigma\sigma^\top\right) (t,x,\mu)D_x^2 V(t,x,\mu)\right]-H\left(t,x,\mu,D_x V(t,x,\mu)\right)\\
		    &-\int_\brn \left\{ \left(D_p H \left(t,y,\mu,D_x V(t,y,\mu)\right)\right)^\top D_y\frac{dV}{d\nu}(t,x,\mu)({y})\right.\\
		    &\quad\qquad\left. +\frac{1}{2}\text{Tr}\left[\left(\sigma\sigma^\top\right) (t,y,\mu)D_y^2\frac{dV}{d\nu}(t,x,\mu)({y})\right]\right\}d\mu(y)=0,\quad t\in[0,T).
        \end{aligned}
    \end{equation}
\end{theorem}

Last but not least, we also discuss the solvability of the corresponding master equation:

\begin{equation}\label{master}
	\left\{
	\begin{aligned}
		&\dd_t V(t,x,\mu)+\frac{1}{2}\text{Tr}\left[\left(\sigma\sigma^\top\right) (t,x,\mu)D_x^2 V(t,x,\mu)\right]+H\left(t,x,\mu,D_x V(t,x,\mu)\right)\\
		&+\int_\brn \left\{ \left(D_p H \left(t,y,\mu,D_x V(t,y,\mu)\right)\right)^\top D_y\frac{dV}{d\nu}(t,x,\mu)({y})\right.\\
		&\quad\qquad\left. +\frac{1}{2}\text{Tr}\left[\left(\sigma\sigma^\top\right) (t,y,\mu)D_y^2\frac{dV}{d\nu}(t,x,\mu)({y})\right]\right\}d\mu(y)=0,\quad t\in[0,T),\\ 
		&V(T,x,\mu)=g(x,\mu),\quad (x,\mu)\in \brn\times\pr_2(\brn),
	\end{aligned}	
	\right.
\end{equation}
We have the well-posedness of the master equation \eqref{master} as stated below. 

\begin{theorem}\label{thm:1}
	Under Assumptions (A1'), (A2') and (A3), and the validity of \eqref{lem:2_0} and \eqref{lem:2_12}, the value functional $V$ is the unique solution of the master equation \eqref{master} in the sense that the following derivatives
	\begin{align*}
		\dd_t V,\ D_x V,\ D_x^2 V,\ D_y\frac{dV}{d\nu},\ D_y^2\frac{dV}{d\nu}
	\end{align*}
	exist and  are all continuous, and they satisfy the growth conditions \eqref{prop6_04}, \eqref{prop6_05}, and \eqref{prop7_03}. 
\end{theorem}

\begin{proof}
The existence result is a direct consequence of Theorems~\ref{prop:6}, \ref{prop:7} and \ref{prop:10}. We next aim to establish the uniqueness result. Let ${U}$ be another classical solution of the master equation \eqref{master}. Since $U$ satisfies Condition \eqref{prop6_04}, from Proposition~\ref{prop:cone}, we know that the map:
\begin{align*}
    \brd\ni v\mapsto L\left(t,x,\mu,v,D_xU(t,x,\mu)\right)\in\br
\end{align*}
has a unique minimizer $\hv\left(t,x,\mu,D_xU(t,x,\mu)\right)$.
For any initial $(t,\mu)$, we choose any $\xi\sim\mu$, and define the process $Y_{t\xi}$ as follows:
\begin{equation}\label{SDE_bar_Y}
	\begin{aligned}
		Y_{t\xi}(s)=\xi&+\int_t^s b\left(r, Y_{t\xi}(r),\lr(Y_{t\xi}(r)),\hv(r,Y_{t\xi}(r),\lr(Y_{t\xi}(r)),D_x U(r, Y_{t\xi}(r),\lr(Y_{t\xi}(r)))) \right)dr\\
		&+ \int_t^s \sigma \left(r, Y_{t\xi}(r),\lr(Y_{t\xi}(r))\right)dB(r),\quad s\in[t,T].
	\end{aligned}
\end{equation}
The well-posedness of SDE \eqref{SDE_bar_Y} then follows from the regularity of the functional $U$, and the proof is a simplified version for that of Theorem~\ref{lem:2}. From \eqref{D_pH&D_xH}, we know that SDE \eqref{SDE_bar_Y} also writes
\begin{equation}\label{SDE_bar_Y'}
    \begin{aligned}
        Y_{t\xi}(s)=\xi&+\int_t^s D_p H\left(r, Y_{t\xi}(r),\lr(Y_{t\xi}(r)),D_x U(r, Y_{t\xi}(r),\lr(Y_{t\xi}(r))) \right)dr\\
		&+ \int_t^s \sigma \left(r, Y_{t\xi}(r),\lr(Y_{t\xi}(r))\right)dB(r),\quad s\in[t,T].
    \end{aligned}
\end{equation}
For any initial data $x$, we consider the control problem $J_{tx}(\cdot;\lr(Y_{t\xi}(s)),t\le s\le T)$: for an arbitrary admissible control $v(\cdot)\in\lr_{\f}^2(t,T)$, the corresponding controlled state process is 
\begin{align*}
	&X_{tx\mu}^v(s)=x+\int_t^s b\left(r,X^v_{tx\mu}(r), \lr(Y_{t\xi}(r)),v(r)\right)dr+\int_t^s \sigma\left(r,X^v_{tx\mu}(r), \lr(Y_{t\xi}(r))\right)dB(r),\quad s\in[t,T],
\end{align*}
and the cost functional is 
\begin{align*}
	J_{tx}(v;\lr(Y_{t\xi}(s)),t\le s\le T)=\e\left[\int_t^T f\left(s,X^v_{tx\mu}(s), \lr(Y_{t\xi}(s)),v(s)\right) ds +g\left(X^v_{tx\mu}(T), \lr(Y_{t\xi}(T))\right) \right].
\end{align*}
From \eqref{SDE_bar_Y'} and It\^o's formula for measure-dependent functionals (see \cite[Theorem 2.2]{AB5} and also \cite[Theorem 7.1]{BR}), we deduce that
\begin{align*}
	&\e\left[U\left(s,X^v_{tx\mu}(s),\lr(Y_{t\xi}(s))\right)-U(t,x,\mu)\right]\\
	=\ & \int_t^s \e \bigg\{ \dd_t U\left(r,X^v_{tx\mu}(r),\lr(Y_{t\xi}(r))\right)+b \left(r,X^v_{tx\mu}(r),\lr(Y_{t\xi}(r)),v(r)\right)^\top D_x U\left(r,X^v_{tx\mu}(r),\lr(Y_{t\xi}(r))\right)\\
	&\quad\qquad +\frac{1}{2}\text{Tr}\left[ D_x^2 U\left(r,X^v_{tx\mu}(r),\lr(Y_{t\xi}(r))\right)\left(\sigma\sigma^\top\right) \left(r,X^v_{tx\mu}(r),\lr(Y_{t\xi}(r))\right)\right]\\
	&\!\!\!+\widetilde{\e}\bigg[D_p H \left(r,\widetilde{Y_{t\xi}(r)},\lr(Y_{t\xi}(r)),D_x U(r, \widetilde{Y_{t\xi}(r)},\lr(Y_{t\xi}(r)))\right)^\top D_y\frac{dU}{d\nu}\left(r,X^v_{tx\mu}(r),\lr(Y_{t\xi}(r))\right)\left(\widetilde{Y_{t\xi}(r)}\right) \\
	&\quad\qquad\qquad +\frac{1}{2}\text{Tr}\left[ D_y^2\frac{dU}{d\nu}\left(r,X^v_{tx\mu}(r),\lr(Y_{t\xi}(r))\right)\left(\widetilde{Y_{t\xi}(r)}\right)\left(\sigma\sigma^\top\right) \left(r,\widetilde{Y_{t\xi}(r)},\lr(Y_{t\xi}(r))\right)\right]\bigg]\bigg\}dr\\
	=\ & \int_t^s \e\bigg\{\dd_t U\left(r,X^v_{tx\mu}(r),\lr(Y_{t\xi}(r))\right)-f\left(r,X^v_{tx\mu}(r),\lr(Y_{t\xi}(r)),v(r)\right)\\
    &\quad\qquad +L\Big(r,X^v_{tx\mu}(r),\lr(Y_{t\xi}(r)),v(r),D_x U\left(r,X^v_{tx\mu}(r),\lr(Y_{t\xi}(r))\right)\Big)\\
    &\quad\qquad +\frac{1}{2}\text{Tr}\left[ D_x^2 U\left(r,X^v_{tx\mu}(r),\lr(Y_{t\xi}(r))\right)\left(\sigma\sigma^\top\right) \left(r,X^v_{tx\mu}(r),\lr(Y_{t\xi}(r))\right)\right]\\
	&\!\!\!+\widetilde{\e}\bigg[D_p H \left(r,\widetilde{Y_{t\xi}(r)},\lr(Y_{t\xi}(r)),D_x U(r, \widetilde{Y_{t\xi}(r)},\lr(Y_{t\xi}(r)))\right)^\top D_y\frac{dU}{d\nu}\left(r,X^v_{tx\mu}(r),\lr(Y_{t\xi}(r))\right)\left(\widetilde{Y_{t\xi}(r)}\right)\\
	&\quad\qquad\qquad +\frac{1}{2}\text{Tr}\left[ D_y^2\frac{dU}{d\nu}\left(r,X^v_{tx\mu}(r),\lr(Y_{t\xi}(r))\right)\left(\widetilde{Y_{t\xi}(r)}\right)\left(\sigma\sigma^\top\right) \left(r,\widetilde{Y_{t\xi}(r)},\lr(Y_{t\xi}(r))\right)\right]\bigg] \bigg\}dr,
\end{align*}	
where $\widetilde{Y_{t\xi}}(r)$ is an independent copy of $Y_{t\xi}(r)$. Since $U$ satisfies Equation \eqref{master}, from the last equation, we see that
\begin{align}
		&\e\left[g\left(X^v_{tx\mu}(T),\lr(Y_{t\xi}(T))\right)-U(t,x,\mu)\right] \notag \\
		=\ & \int_t^T \e\Big\{-f\left(r,X^v_{tx\mu}(r),\lr(Y_{t\xi}(r)),v(r)\right)\notag\\
		&\quad\qquad +L\Big(r,X^v_{tx\mu}(r),\lr(Y_{t\xi}(r)),v(r),D_x U\left(r,X^v_{tx\mu}(r),\lr(Y_{t\xi}(r))\right)\Big) \notag \\
		&\quad\qquad -H\Big(r,X^v_{tx\mu}(r),\lr(Y_{t\xi}(r)),D_x U\left(r,X^v_{tx\mu}(r),\lr(Y_{t\xi}(r))\right)\Big)\Big\}dr, \label{thm1_1}
\end{align}
therefore, from the definition of the Hamiltonian $H$, we have
\begin{equation}\label{pf:uni_2}
	\begin{split}
		J_{tx}(v(\cdot);\lr(Y_{t\xi}(s),t\le s\le T))&\geq U(t,x,\mu).
	\end{split}
\end{equation}
We define the processes $(Y_{tx\mu},u_{tx\mu})$ as 
\begin{align}
	Y_{tx\mu}(s)=x&+\int_t^s b\left(r, Y_{tx\mu}(r),\lr(Y_{t\xi}(r)),\hv(r,Y_{tx\mu}(r),\lr(Y_{t\xi}(r)),D_x U(r, Y_{tx\mu}(r),\lr(Y_{t\xi}(r)))) \right)dr \notag \\
	&+ \int_t^s \sigma \left(r, Y_{tx\mu}(r),\lr(Y_{t\xi}(r))\right)dB(r),\quad s\in[t,T], \label{SDE_bar_Y_x}
\end{align}
where 
$$u_{tx\mu}(s):=\hv(s,Y_{tx\mu}(s),\lr(Y_{t\xi}(s)),D_x U(s, Y_{tx\mu}(s),\lr(Y_{t\xi}(s)))). $$
The well-posedness of $(Y_{tx\mu},u_{tx\mu})$ follows from the regularity of the functional $U$ and the well-posedness of the process $Y_{t\xi}$. If we set $v(s)=u_{tx\mu}(s)$ in \eqref{thm1_1}, then we obtain the following equation: 
\begin{align}\label{pf:uni_3}
    {U}(t,x,\mu)= J_{tx}(u_{tx\mu};\lr(Y_{t\xi}(s)),t\le s\le T).
\end{align}
We now define the processes 
\begin{equation}\label{pf:uni_1}
    \begin{aligned}
    	R_{t\xi}(s):=\ & D_x U\left(s,Y_{t\xi}(s),\lr(Y_{t\xi}(s))\right),\\
    	S_{t\xi}(s):=\ & D_x^2 U\left(s,Y_{t\xi}(s),\lr(Y_{t\xi}(s))\right)\sigma \left(s,Y_{t\xi}(s),\lr(Y_{t\xi}(s))\right),
    \end{aligned}
\end{equation}
then, from the SDE \eqref{SDE_bar_Y} for $Y_{t\xi}$, similar to \cite[Theorem 6.1]{AB11}, by using  It\^o's lemma and the master equation  \eqref{master} for $U$, we verify that $(R_{t\xi},S_{t\xi})$ satisfies the following BSDE as an adapted solution:
\begin{equation}\label{BSDE_bar_pq}
    \begin{aligned}
    	R_{t\xi}(T)=\ & D_x g\left(Y_{t\xi}(T),\lr(Y_{t\xi}(T))\right)-\int_s^T S_{t\xi}(r)dB(r)\\
    	&+\int_s^T \bigg[\left(D_xH \left(r, Y_{t\xi}(r),\lr(Y_{t\xi}(r)),R_{t\xi}(r) \right)\right)^\top R_{t\xi}(r)\\
    	&\quad\qquad +\sum_{j=1}^n \left(D_x\sigma^j\left(r, Y_{t\xi}(r),\lr(Y_{t\xi}(r))\right)\right)^\top S^j_{t\xi}(r)\bigg] dr,\quad s\in[t,T].
    	\end{aligned}
\end{equation}
From \eqref{SDE_bar_Y}, \eqref{pf:uni_1} and \eqref{BSDE_bar_pq}, the process $(Y_{t\xi}, R_{t\xi},S_{t\xi})$ satisfy the FBSDEs \eqref{FB:mfg_generic}. Then, from the uniqueness result of FBSDEs \eqref{FB:mfg_generic}, we know that $(Y_{t\xi}, R_{t\xi},S_{t\xi})=(X_{t\xi}, P_{t\xi},Q_{t\xi})$, and therefore, $\lr(Y_{t\xi}(s))=\lr(X_{t\xi}(s))$. Similar to the above, if we define 		
\begin{align*}
    R_{tx\mu}(s):=\ & D_x U\left(s, Y_{tx\mu}(s),\lr(X_{t\xi}(s))\right), \\
    S_{tx\mu}(s):=\ & D_x^2 U\left(s,Y_{tx\mu}(s),\lr(X_{t\xi}(s))\right)\sigma \left(s,Y_{tx\mu}(s),\lr(X_{t\xi}(s))\right),
\end{align*}
with the uniqueness result of FBSDEs, then $(Y_{tx\mu}, R_{tx\mu},S_{tx\mu})=(X_{tx\mu}, P_{tx\mu},Q_{tx\mu})$, and therefore $u_{tx\mu}=v_{tx\mu}$. From \eqref{pf:uni_2} and \eqref{pf:uni_3} and the definition of $V$ in  \eqref{intro_4}, we know that $U$ coincides with the value functional $V$.
\end{proof}

The proof of Theorems~\ref{prop:10} and \ref{thm:1} relies on the application of It\^o's formula for measure-dependent functionals \cite{AB5,BR}, which leads to the derivation of the corresponding master equation \eqref{master}. For the master equation corresponding to MFGs with common noises, a more general It\^o's formula (such as \cite[Theorem 3.2]{AB13} and \cite[Theorem 4.17]{book_mfg} ) is required to tackle the prevalence of conditional distribution issue. Similar to  the proof of Theorems~\ref{prop:10} and \ref{thm:1}, we can also give classical solutions of such kind of master equations, and we leave this in our future works. 

\begin{remark}
    Starting from the viewpoint from the viscosity solution for MFG master equation, Bertucci et al. \cite{Bertucci} extends more to the newly defined Lipschitz solutions of the MFG master equations, and it proves the unique existence of such kind of Lipschitz solutions for small time via an analytical approach based on the fixed point theorem when the initial condition is Lipschitz continuous and the Hamiltonian $H$ is differentiable with Lipschitz first order derivatives. From a very different perspective as Bertucci et al. \cite{Bertucci}, our work aims at reducing  regularity (and also monotonicity) on the coefficients to warrant the unique existence of classical (smooth) and global-in-time solution of MFG master equation \eqref{intro_4}. We require the regularity assumptions in (A2’) for the cost functionals, and we allow more generic drift and diffusion functions: the drift can be nonlinear in state, control and distribution, and the diffusion can be distribution-dependent and also nonlinear and degenerate. We emphasize that it is not that one result can cover another; indeed, the two results come from two different perspectives and via two totally different approaches.
\end{remark}

\section*{Acknowledgement}

Alain Bensoussan is supported by the National Science Foundation under grant NSF-DMS-2204795. Ziyu Huang acknowledges the financial supports as a postdoctoral fellow from Department of Statistics of The Chinese University of Hong Kong. Shanjian Tang is supported by the National Natural Science Foundation of China under grant nos. 11631004 and  12031009. Phillip Yam acknowledges the financial supports from HKGRF-14301321 with the project title ``General Theory for Infinite Dimensional Stochastic Control: Mean field and Some Classical Problems'', and HKGRF-14300123 with the project title ``Well-posedness of Some Poisson-driven Mean Field Learning Models and their Applications''. The work described in this article was also supported by a grant from the Germany/Hong Kong Joint Research Scheme sponsored by the Research Grants Council of Hong Kong and the German Academic Exchange Service of Germany (Reference No. G-CUHK411/23). He also thanks The University of Texas at Dallas for the kind invitation to be a Visiting Professor in Naveen Jindal School of Management.


\newpage

\normalsize 
\appendix
\appendixpage
\addappheadtotoc

In all the following proofs, we mean by $C(\alpha_1,\dots, \alpha_k)$  a constant depending only on parameters $(\alpha_1,\dots, \alpha_k)$; otherwise,  $C$  is a constant depending on all the components of the vector
$$(L,T,\lambda_b,\lambda_v,\lambda_x,l_b^m, l_\sigma^m,L_b^0,L_b^1,L_b^2,L_f^0,L_f^1). $$

\section{Proof of Theorem~\ref{lem:5}}\label{pf:lem:5}

Similar to the proof of \cite[Theorem 5.2]{AB12},  we have the well-posedness from~\cite[Lemma 2.1]{AB12} on the well-posedness of FBSDEs in Hilbert spaces under $\beta$-monotonicity. The proof of the sufficiency of the (necessity) maximum principle is similar to that of \cite[Theorem 5.1]{AB12}. Here, we only prove \eqref{lem5_1} and \eqref{lem5_3}, and the proof of \eqref{lem5_2} is similar to  that of \eqref{lem5_1}. For notational simplicity, we drop the subscript $tx\xi$ in $(X_{tx\xi},P_{tx\xi},Q_{tx\xi},v_{tx\xi})$, and denote by $m(s):=\lr\left(X_{t\xi}(s)\right)$ for $s\in[t,T]$. 

By applying It\^o's formula to $P(s)^\top X(s)$ and using the third equation of FBSDEs \eqref{intro_3}, similar as in \eqref{lem:2_1}, we have
\begin{align}
    &\e\left[D_xg(X(T),m(T))^\top X(T)-P(t)^\top x \right] \notag\\
    =\ & \e\Bigg\{\int_t^T \left[b(s,\cdot,m(s),\cdot)\Big|^{(X(s),v(s))}_{(0,0)} -\left[(D_x b,D_v b) (s,X(s),m(s),v(s)) \right]\begin{pmatrix}
        X(s)\\ v(s)
    \end{pmatrix}\right]^\top P(s)  \notag \\
    &\qquad -\left[\begin{pmatrix}
        D_x f\\ D_v f
    \end{pmatrix}(s,\cdot,m(s),\cdot)\Bigg|^{(X(s),v(s))}_{(0,0)} \right]^\top \begin{pmatrix}
        X(s)\\ v(s)
    \end{pmatrix}  -\left[\begin{pmatrix}
        D_x f\\ D_v f
    \end{pmatrix}(s,0,m(s),0) \right]^\top \begin{pmatrix}
        X(s)\\ v(s)
    \end{pmatrix}  \notag \\
    &\qquad + b(s,0,m(s),0))^\top P(s) + \sum_{j=1}^{n}  \sigma^j_0(s,m(s))^\top Q^j(s) ds\Bigg\}. \label{lem:5_1}
\end{align}
Similar to that for \eqref{lem:2_2}, the third equation of FBSDEs \eqref{intro_3}, Condition \eqref{generic:condition:b'} and Assumption (A2) also yield  
\begin{align}\label{lem:5_2}
    |P(s)|\le \frac{L^2}{\lambda_b} \left[1+|X(s)|+W_2(m(s),\delta_0)+|v(s)|\right].
\end{align}
Then, from Condition \eqref{generic:condition:b}, we have
\begin{align}
    &\left|\left[b(s,\cdot,m(s),\cdot)\bigg|^{(X(s),v(s))}_{(0,0)} -\left[(D_x b,D_v b) (s,X(s),m(s),v(s)) \right]\begin{pmatrix}
        X(s)\\ v(s)
    \end{pmatrix}\right]^\top P(s)\right| \notag \\
    \le\ & \frac{L^2 L_b^0}{\lambda_b} |X(s)|^2 +\frac{2L^2 L_b^1}{\lambda_b} |X(s)|\cdot |v(s)| +\frac{L^2 L_b^2}{\lambda_b} |v(s)|^2. \label{lem:5_3}
\end{align}
From \eqref{lem:5_2} and Assumption (A1), we have
\begin{align}
    \left|b(s,0,m(s),0))^\top P(s)\right| \le\ & \frac{L^3}{\lambda_b} \left(1+W_2(m(s),\delta_0) \right)\left(1+|X(s)|+W_2(m(s),\delta_0)+|v(s)|\right). \label{lem:5_4}
\end{align}
From the convexity assumption in (A3), we have
\begin{equation}\label{lem:5_5}
\begin{aligned}
    \left[\begin{pmatrix}
        D_x f\\ D_v f
    \end{pmatrix}(s,\cdot,m(s),\cdot)\bigg|^{(X(s),v(s))}_{(0,0)} \right]^\top \begin{pmatrix}
        X(s)\\ v(s)
    \end{pmatrix} \geq\ & 2\lambda_v|v(s)|^2+2\lambda_x |X(s)|^2; \\[3mm]
    D_xg(X(T),m(T))^\top X(T)\geq\ & 2\lambda_g |X(T)|^2\geq 0.
\end{aligned}
\end{equation}
From Assumption (A2), we also know that
\begin{equation}\label{lem:5_5'}
\begin{aligned}
    \left|\left[\begin{pmatrix}
        D_x f\\ D_v f
    \end{pmatrix}(s,0,m(s),0)\right]^\top \begin{pmatrix}
        X(s)\\ v(s)
    \end{pmatrix}\right| \le\ & L \left(1+ W_2(m(s),\delta_0)\right)(|X(s)|+|v(s)|);\\[3mm]
    \left|D_xg(0,m(T))^\top X(T)\right|\le \ & L(1+W_2(m(T),\delta_0))|X(T)|.
\end{aligned}
\end{equation}
From Assumption (A1) and Cauchy's inequality, we also know that
\begin{align}
    &\e \bigg|\int_t^T \sum_{j=1}^n \sigma^j_0(s,m(s))^\top Q^j(s) ds \bigg| \le L\int_t^T \left[1+ W_2(m(s),\delta_0)\right] \sum_{j=1}^n \left\|Q^j(s)\right\|_2 ds. \label{sigma_m_4}
\end{align} 
Substituting \eqref{lem:5_3}-\eqref{sigma_m_4} back into \eqref{lem:5_1}, we have
\begin{align}
    &\e\left[\int_t^T 2\lambda_v|v(s)|^2+2\lambda_x|X(s)|^2 ds \right] \notag\\
    \le\ & \int_t^T \bigg[ \frac{L^2 L_b^2}{\lambda_b} \|v(s)\|_2^2 +\frac{2L^2 L_b^1}{\lambda_b} \|X(s)\|_2\cdot \|v(s)\|_2 + \frac{L^2 L_b^0}{\lambda_b} \|X(s)\|_2^2 \bigg] ds  \notag \\
    &+C(L,\lambda_b)\bigg[\int_t^T \left(1+W_2(m(s),\delta_0)\right) \Big( 1+W_2(m(s),\delta_0)+\|X(s)\|_2 + \|v(s)\|_2 +\left\|Q(s)\right\|_2 \Big)ds \notag \\
    &\quad\qquad\qquad + \left(1+W_2(m(T),\delta_0)\right) \|X(T)\|_2 + \|P(t)\|_2 \cdot |x| \bigg]. \notag 
\end{align}
Then, from \eqref{lem:2_0} with an application of Young's inequality, we know that for any $\epsilon>0$, we have
\begin{align}
    & \left(2\lambda_v-\frac{L^2 L_b^2}{\lambda_b}-\frac{L^2 L_b^1}{\lambda_b\sqrt{2\lambda_x-\frac{L^2 L_b^0}{\lambda_b}}}\right) \int_t^T \|v(s)\|_2^2 ds \notag \\
    \le\ & \epsilon \left[\sup_{t\le s\le T}\|X(s)\|_2^2+\sup_{t\le s\le T}\|P(s)\|_2^2+\int_t^T  \Big( \left\|Q(s)\right\|_2^2 + \|v(s)\|_2^2 \Big) ds\right] \notag \\
    &\quad +C(L,T,\lambda_b)\left(1+\frac{1}{\epsilon}\right) \left(1+\sup_{t\le s\le T} W_2^2 (m(s),\delta_0) \right). \label{lem:5_6}
\end{align}
Using Gr\"onwall's inequality, we have
\begin{equation}\label{lem:5_7}
	\e\bigg[\sup_{t\le s\le T}|X(s)|^2\bigg]\le C(L,T)\left[1+|x|^2+\e\sup_{t\le s\le T}W_2^2(m(s),\delta_0)+\e\int_t^T |v(s)|^2ds\right].
\end{equation}
The estimate \eqref{lem:5_7}, the usual BSDE estimates in \cite{YH2,SP} and  the BDG inequality  yield 
\begin{equation}\label{lem:5_8}
	\e\bigg[\sup_{t\le s\le T}|P(s)|^2+\int_t^T \!\!\!  |Q(s)|^2 ds\bigg]\le C(L,T)\left[1+|x|^2+\e\sup_{t\le s\le T}W_2^2(m(s),\delta_0)+\e\int_t^T |v(s)|^2ds\right].
\end{equation}    
Substituting \eqref{lem:5_7} and \eqref{lem:5_8} into \eqref{lem:5_6}, we have
\begin{align*}
    & \left(2\lambda_v-\frac{L^2 L_b^2}{\lambda_b}-\frac{L_f^1+\frac{L^2 (2L_b^1+l_b^m)}{\lambda_b}}{2\sqrt{2\lambda_x-L_f^0-\frac{L^2 (L_b^0+2l_b^m)}{\lambda_b}}}\right)\cdot \int_t^T \|v(s)\|_2^2 ds \notag \\
    \le\ & \epsilon \  C(L,T)\left[\int_t^T  \|v(s)\|_2^2  ds\right]+C(L,T,\lambda_b)\left(1+\epsilon+\frac{1}{\epsilon}\right) \left(1+|x|^2+\sup_{t\le s\le T}W_2^2(m(s),\delta_0)\right).
\end{align*}
From \eqref{lem2_1}, we know that
\begin{align}\label{lem:5_8.5}
    \sup_{t\le s\le T}W_2^2(m(s),\delta_0)\le\e\left[ \sup_{t\le s\le T}\left|X_{t\xi}(s)\right|^2\right]\le C \left(1+\|\xi\|_2^2\right),
\end{align}
then, by choosing $\epsilon$ small enough, we have
\begin{align*}
    & \int_t^T \|v(s)\|_2^2 ds \le C \left(1+|x|^2+W_2^2(\mu,\delta_0) \right).
\end{align*}
Substituting the last estimate and \eqref{lem:5_8.5} back into \eqref{lem:5_7} and \eqref{lem:5_8}, we have
\begin{equation}\label{lem:5_9}
\begin{aligned}
    &\e\left[\sup_{t\le s\le T}\left|\left(X(s), P(s)\right)^\top\right|^2+\int_t^T |Q(s)|^2ds \right]\le C \left(1+|x|^2+W_2^2(\mu,\delta_0) \right).
\end{aligned}
\end{equation}
Similar to \eqref{lem:2_10}, from the third equation in FBSDEs \eqref{intro_3} and the convexity of $f$ in $v$ in Assumption (A3), we have
\begin{align}\label{lem:5_10}
    |v(s)|\le\ & \frac{L}{2\lambda_v}\left[1+|P(s)|+|X(s)|+ W_2(m(s),\delta_0)\right].
\end{align}
Then, from \eqref{lem:5_8.5}, \eqref{lem:5_9} and \eqref{lem:5_10}, we obtain \eqref{lem5_1}.

We next prove \eqref{lem5_3} under Condition \eqref{lem:2_12}. From \eqref{lem:5_2} and \eqref{lem:5_10}, we know that
\begin{align*}
    |P(s)|\le\ & \frac{L^2}{\lambda_b} \left[1+|X(s)|+W_2(m(s),\delta_0)\right]+\frac{L^2}{\lambda_b}|v(s)|\\
    \le\ &  \frac{L^2}{\lambda_b} \left[1+|X(s)|+W_2(\lr(X(s)),\delta_0)\right]+\frac{L^3}{2\lambda_v\lambda_b}\left[1+|P(s)|+|X(s)|+W_2(m(s),\delta_0)\right].
\end{align*}
Since $\frac{L^2}{\lambda_b}<1$, we deduce that
\begin{align}\label{lem:5_11}
    |P(s)|\le \left(1-\frac{L^3}{2\lambda_v\lambda_b}\right)^{-1}\frac{L^2}{\lambda_b}\left(1+\frac{L}{2\lambda_v}\right)\left[1+|X(s)|+W_2(m(s),\delta_0)\right].
\end{align}
Then, from Assumption (A1) and (A3), the estimate \eqref{lem:5_11} and Proposition~\ref{prop:cone}, we obtain \eqref{lem5_3} and \eqref{intro_3'}.

\section{Proof of the statements in Section~\ref{sec:distribution}}\label{pf:distribution}

\subsection{Proof of Theorem~\ref{lem:4}}\label{pf:lem:4}

For notational simplicity, we still adopt the notations $(\mathcal{X},\mathcal{P},\mathcal{Q},\mathcal{V})$ in place of the processes $\mathcal{D}_\eta \Theta_{t\xi}:=\left(\mathcal{D}_\eta X_{t\xi},\mathcal{D}_\eta P_{t\xi},\mathcal{D}_\eta Q_{t\xi},\mathcal{D}_\eta v_{t\xi}\right)$. For $\epsilon\in(0,1)$ and $s\in[t,T]$, set 
$$\de^\epsilon X(s):=\frac{1}{\epsilon}\left[X_{t\xi^\epsilon}(s)-X_{t\xi}(s)\right]\quad  \text{and} \quad \delta^\epsilon X(s):=\de^\epsilon X(s)-\mathcal{X}(s).$$
Define similarly  $\de^\epsilon P(s),\de^\epsilon Q(s),\de^\epsilon v(s)$ and $\delta^\epsilon P(s),\delta^\epsilon Q(s),\delta^\epsilon v(s)$. For sake of convenience, we denote by 
\begin{align*}
    &\de^\epsilon\Theta(s):=(\de^\epsilon X(s),\de^\epsilon P(s),\de^\epsilon Q(s),\de^\epsilon v(s)),\quad \delta^\epsilon\Theta(s):=(\delta^\epsilon X(s),\delta^\epsilon P(s),\delta^\epsilon Q(s),\delta^\epsilon v(s)).
\end{align*}
From \eqref{lem2_2},  we see that 
\begin{equation}\label{lem4_2}
	\begin{split}
		&\e\bigg[ \sup_{t\le s\le T}\left|\left(\de^\epsilon X(s),\de^\epsilon P(s),,\de^\epsilon v(s)\right)^\top\right|^2 + \int_t^T \left|\de^\epsilon Q(s)\right|^2 ds \bigg]\le C\|\eta\|_2^2;
	\end{split}
\end{equation}	
and from \eqref{lem3_1}, we see that
\begin{equation}\label{lem4_2'}
	\begin{split}
		&\e\left[\sup_{t\le s\le T}\left|\left(\mathcal{X}(s),\mathcal{P}(s),\mathcal{V}(s),\delta^\epsilon X(s),\delta^\epsilon P(s),\delta^\epsilon v(s)\right)^\top\right|^2+ \int_t^T\left|\left(\mathcal{Q}(s),\delta^\epsilon Q(s)\right)^\top \right|^2 ds\right] \le C\|\eta\|_2^2.
	\end{split}
\end{equation}
From the FBSDEs \eqref{FB:mfg_generic} and FBSDEs \eqref{FB:dr}-\eqref{FB:dr_condition}, we know that the process $\delta^\epsilon\Theta$ satisfies the following FBSDEs: 
\small
\begin{align*}
    \delta^\epsilon X(s)=\ & \int_t^s \Bigg\{D_x b(r, \theta_{t\xi}(r))\delta^\epsilon X(r)+\widetilde{\e}\bigg[ D_y \frac{db}{d\nu} (r,\theta_{t\xi}(r))\left(\widetilde{X_{t\xi}}(r)\right) \widetilde{\delta^\epsilon X}(r)\bigg]+D_v b(r, \theta_{t\xi}(r))\delta^\epsilon v(r) +\alpha^\epsilon_b(r)\Bigg\} dr \\
    &+ \int_t^s \Bigg\{ \sigma_1(r)\delta^\epsilon(r) + \widetilde{\e}\bigg[ D_y \frac{d\sigma_0}{d\nu} (r,\lr(X_{t\xi}(r)))\left(\widetilde{X_{t\xi}}(r)\right) \widetilde{\delta^\epsilon X}(r)\bigg] +\alpha^\epsilon_{\sigma}(r) \Bigg\}dB(r),\\
    \delta^\epsilon P(s)=\ & \left[D_x^2 g  (X_{t\xi}(T),\lr(X_{t\xi}(T))) \right]^\top \delta^\epsilon X(T) +\alpha^\epsilon_g \\
	&+\widetilde{\e}\left\{\left[\left(D_y \frac{d}{d\nu}D_x g \right) (X_{t\xi}(T),\lr(X_{t\xi}(T)))\left(\widetilde{X_{t\xi}}(T)\right) \right]^\top \widetilde{\delta^\epsilon X}(T)\right\}\\
    &+\int_s^T \bigg[ D_x b(r,\theta_{t\xi^\epsilon}(r))^\top \delta^\epsilon P(r)+\sum_{j=1}^n \left(\sigma^j_1(r)\right)^\top \delta^\epsilon Q^j(r)+\alpha^\epsilon_f(r) \bigg]  dr\\
    &+\int_s^T \left(P_{t\xi}(r)\right)^\top \bigg\{D_x^2 b(r,\theta_{t\xi}(r)) \delta^\epsilon X(r)+ D_vD_x b(r,\theta_{t\xi}(r)) \delta^\epsilon v(r)\\
    &\qquad\qquad\qquad\qquad +\widetilde{\e}\left[ D_y\frac{d}{d\nu}D_x b(r,\theta_{t\xi}(r))\left(\widetilde{X_{t\xi}}(r)\right) \widetilde{\delta^\epsilon X}(r)\right]\bigg\} dr\\
    &+\int_s^T \Bigg\{\left[D_x^2 f (r,\theta_{t\xi}(r)) \right]^\top \delta^\epsilon X(r)+ \left[D_vD_x f (r,\theta_{t\xi}(r)) \right]^\top \delta^\epsilon v(r)\\
    &\qquad\qquad + \widetilde{\e}\left[\bigg[ \left(D_y \frac{d}{d\nu}D_x f \right) (r,\theta_{t\xi}(r))\left(\widetilde{X_{t\xi}}(r)\right)\bigg]^\top \widetilde{\delta^\epsilon X}(r) \right]\Bigg\}dr -\int_s^T \delta^\epsilon Q(r)\; dB(r),\quad s\in[t,T];
\end{align*}
\normalsize
subject to the condition 
\begin{align}
    0=\ & D_v b(s,\theta_{t\xi^\epsilon}(s))^\top \delta^\epsilon P(s)+\alpha^\epsilon_v(s)  \notag \\
    &+\left(P_{t\xi}(s)\right)^\top \bigg\{D_xD_v b(s,\theta_{t\xi}(s)) \delta^\epsilon X(s)+ D_v^2 b(s,\theta_{t\xi}(s)) \delta^\epsilon v(s) \notag \\
    &\qquad\qquad\qquad +\widetilde{\e}\left[ D_y\frac{d}{d\nu}D_v b(s,\theta_{t\xi}(s))\left(\widetilde{X_{t\xi}}(s)\right) \widetilde{\delta^\epsilon X}(s)\right]\bigg\}  \notag \\
    &+ \left[D_xD_v f (s,\theta_{t\xi}(s)) \right]^\top \delta^\epsilon X(s)+ \left[D_v^2 f (s,\theta_{t\xi}(s)) \right]^\top \delta^\epsilon v(s) \notag \\
    &+ \widetilde{\e}\left[\bigg[ \left(D_y \frac{d}{d\nu}D_v f \right) (s,\theta_{t\xi}(s))\left(\widetilde{X_{t\xi}}(s)\right)\bigg]^\top \widetilde{\delta^\epsilon X}(s) \right],\quad s\in[t,T]; \label{lem:4_1}
\end{align}
where the processes are the remainder terms after the Taylor expansion up to order $1$:
\small
\begin{align*}
    \alpha^\epsilon_b(r):=\ & \int_0^1 \left[D_x b\left(r,\cdot\right)\Big|^{\theta_{t\xi}^{h,\epsilon}}_{\theta_{t\xi}(r)} \right]dh\ \de^\epsilon X(r) +\int_0^1 \widetilde{\e}\left[\left( D_y \frac{db}{d\nu} \left(r,\cdot\right)\left(\cdot\right)\bigg|^{\left(\theta_{t\xi}^{h,\epsilon}(r),\widetilde{X_{t\xi}^{h,\epsilon}}(r)\right)}_{\left(\theta_{t\xi}(r),\widetilde{X_{t\xi}}(r)\right)} \right)\widetilde{\de^\epsilon X}(r)\right] dh \\
    & +\int_0^1 \left[D_v b\left(r,\cdot\right)\Big|^{\theta_{t\xi}^{h,\epsilon}}_{\theta_{t\xi}(r)} \right]dh\ \de^\epsilon v(r),\\
    \alpha^\epsilon_\sigma (r):=\ & \int_0^1 \widetilde{\e}\left[\left( D_y \frac{d\sigma_0}{d\nu} \left(r,\cdot\right)\left(\cdot\right)\bigg|^{\left(\lr\left(X_{t\xi}^{h,\epsilon}(r)\right),\widetilde{X_{t\xi}^{h,\epsilon}}(r)\right)}_{\left(\lr(X_{t\xi}(r)),\widetilde{X_{t\xi}}(r)\right)} \right)\widetilde{\de^\epsilon X}(r)\right] dh,\\
    \alpha^\epsilon_f(r):=\ & \int_0^1  \left(P_{t\xi}(r)\right)^\top \Bigg\{ D_x^2b(r,\cdot)\Big|^{\theta_{t\xi}^{h,\epsilon}(r)}_{\theta_{t\xi}(r)} \de^\epsilon X(r) + D_v D_xb(r,\cdot)\Big|^{\theta_{t\xi}^{h,\epsilon}(r)}_{\theta_{t\xi}(r)} \de^\epsilon v(r)\\
    &\quad\qquad\qquad\qquad + \widetilde{\e}\Bigg[ D_y\frac{d}{d\nu}D_x b(r,\cdot)\left(\cdot\right) \bigg|^{\left(\theta_{t\xi}^{h,\epsilon}(r),\widetilde{X_{t\xi}^{h,\epsilon}}(r)\right)}_{\left(\theta_{t\xi}(r),\widetilde{X_{t\xi}}(r)\right)} \widetilde{\de^\epsilon X}(r)\Bigg]\Bigg\} dh \\
    &+ \int_0^1 \Bigg\{\left[D_x^2 f  \left(r,\cdot\right)\Big|^{\theta_{t\xi}^{h,\epsilon}(r)}_{\theta_{t\xi}(r)} \right]^\top \de^\epsilon X(r)+ \left[D_vD_x f  \left(r,\cdot\right)\Big|^{\theta_{t\xi}^{h,\epsilon}(r)}_{\theta_{t\xi}(r)} \right]^\top \de^\epsilon v(r)\\
	&\qquad\qquad +\widetilde{\e}\Bigg[\bigg[\left(D_y \frac{d}{d\nu}D_x f \right) \left(r,\cdot\right) \left(\cdot \right)\bigg|^{\left(\theta_{t\xi}^{h,\epsilon}(r),\widetilde{X_{t\xi}^{h,\epsilon}}(r)\right)}_{\left(\theta_{t\xi}(r),\widetilde{X_{t\xi}}(r)\right)} \bigg]^\top \widetilde{\de^\epsilon X}(s)\Bigg] \Bigg\}dh +\left(D_x b(r,\cdot)\Big|^{\theta_{t\xi^\epsilon}(r)}_{\theta_{t\xi}(r)} \right)^\top \mathcal{P}(r), \\
    \alpha^\epsilon_g:=\ &\int_0^1 \left[D_x^2 g \Big|^{\left(X_{t\xi}^{h,\epsilon}(T),\lr\left(X_{t\xi}^{h,\epsilon}(T)\right)\right)}_{(X_{t\xi}(T),\lr(X_{t\xi}(T)))} \right]^\top dh\  \de^\epsilon X(T) \\
    &+\widetilde{\e}\left\{\int_0^1 \left[\left(D_y \frac{d}{d\nu}D_x g \right) (\cdot)\left(\cdot\right)\bigg|^{\left(X_{t\xi}^{h,\epsilon}(T),\lr\left(X_{t\xi}^{h,\epsilon}(T)\right),\widetilde{X_{t\xi}^{h,\epsilon}}(T)\right)}_{\left(X_{t\xi}(T),\lr(X_{t\xi}(T)),\widetilde{X_{t\xi}}(T)\right)} \right]^\top \widetilde{\de^\epsilon X}(T) \right\},\\
    \alpha^\epsilon_v(s):=\ & \int_0^1 \left(P_{t\xi}(s)\right)^\top \Bigg\{ D_x D_v b(s,\cdot)\Big|^{\theta_{t\xi}^{h,\epsilon}(s)}_{\theta_{t\xi}(s)} \de^\epsilon X(s) + D_v^2 b(s,\cdot)\Big|^{\theta_{t\xi}^{h,\epsilon}(s)}_{\theta_{t\xi}(s)} \de^\epsilon v(s) \notag \\
    &\quad\qquad\qquad\qquad + \widetilde{\e}\Bigg[ D_y\frac{d}{d\nu}D_v b(s,\cdot)\left(\cdot\right) \bigg|^{\left(\theta_{t\xi}^{h,\epsilon}(s),\widetilde{X_{t\xi}^{h,\epsilon}}(s)\right)}_{\left(\theta_{t\xi}(s),\widetilde{X_{t\xi}}(s)\right)} \widetilde{\de^\epsilon X}(s)\Bigg]\Bigg\} dh \notag \\
    &+ \int_0^1 \Bigg\{\left[D_xD_v f  \left(s,\cdot\right)\Big|^{\theta_{t\xi}^{h,\epsilon}(s)}_{\theta_{t\xi}(s)} \right]^\top \de^\epsilon X(s)+ \left[D_v^2 f  \left(s,\cdot\right)\Big|^{\theta_{t\xi}^{h,\epsilon}(s)}_{\theta_{t\xi}(s)} \right]^\top \de^\epsilon v(s) \notag \\
	&\qquad+\widetilde{\e}\Bigg[\bigg[\left(D_y \frac{d}{d\nu}D_v f \right) \left(s,\cdot\right) \left(\cdot \right)\bigg|^{\left(\theta_{t\xi}^{h,\epsilon}(s),\widetilde{X_{t\xi}^{h,\epsilon}}(s)\right)}_{\left(\theta_{t\xi}(s),\widetilde{X_{t\xi}}(s)\right)} \bigg]^\top \widetilde{\de^\epsilon X}(s)\Bigg] \Bigg\}dh+\left(D_v b(s,\cdot)\Big|^{\theta_{t\xi^\epsilon}(s)}_{\theta_{t\xi}(s)} \right)^\top \mathcal{P}(s).
\end{align*}
\normalsize
Here, 
$$\theta_{t\xi}^{h,\epsilon}(s):=\theta_{t\xi}(s)+h\epsilon\de^\epsilon\theta(s), \quad  \Theta_{t\xi}^{h,\epsilon}(s):=\Theta_{t\xi}(s)+h\epsilon\de^\epsilon\Theta(s)$$
 for $h\in[0,1]$, and $\left(\widetilde{X_{t\xi}^{h,\epsilon}}(s),\widetilde{\de^\epsilon X}(s),\widetilde{\delta^\epsilon X}(s)\right)$ is an independent copy of $\left(X_{t\xi}^{h,\epsilon}(s),\de^\epsilon X(s),\delta^\epsilon X(s)\right)$. Then, by using the last equality and applying It\^o's formula to $\left(\delta^\epsilon P(s)\right)^\top \delta^\epsilon X(s)$, we have
\small
\begin{align}
    &\e \Bigg\{ \left[D_x^2 g  (X_{t\xi}(T),\lr(X_{t\xi}(T))) \right] \left(\delta^\epsilon X(T)\right)^{\otimes2} +\left(\delta^\epsilon X(T)\right)^\top \alpha^\epsilon_g  \notag \\
	&\quad + \left(\delta^\epsilon X(T)\right)^\top \widetilde{\e}\left[\left[\left(D_y \frac{d}{d\nu}D_x g \right) (X_{t\xi}(T),\lr(X_{t\xi}(T)))\left(\widetilde{X_{t\xi}}(T)\right) \right]^\top \widetilde{\delta^\epsilon X}(T)\right] \Bigg\} \notag \\
    =\ & \e \int_t^T \left(\delta^\epsilon P(s)\right)^\top \widetilde{\e}\bigg[ D_y \frac{db}{d\nu} (s,\theta_{t\xi}(s))\left(\widetilde{X_{t\xi}}(s)\right) \widetilde{\delta^\epsilon X}(s)\bigg] ds  \notag \\
    & + \sum_{j=1}^n  \e \int_t^T \left(\delta^\epsilon Q^j(s)\right)^\top \widetilde{\e}\bigg[ D_y \frac{d\sigma^j_0}{d\nu} (s,\lr(X_{t\xi}(s)))\left(\widetilde{X_{t\xi}}(s)\right) \widetilde{\delta^\epsilon X}(s)\bigg] ds \notag\\
    &-\e \int_t^T \left(P_{t\xi}(s)\right)^\top \left[\begin{pmatrix}
        D_x^2 b & D_xD_v b \\ D_vD_x b & D_v^2 b
    \end{pmatrix}(s,\theta_{t\xi}(s)) \right] \begin{pmatrix}
        \delta^\epsilon X(s) \\ \delta^\epsilon v(s)
    \end{pmatrix}^{\otimes 2} ds  \notag \\
    &-\e\int_t^T \left(P_{t\xi}(s)\right)^\top \widetilde{\e}\left[ D_y\frac{d}{d\nu}D_x b(s,\theta_{t\xi}(s))\left(\widetilde{X_{t\xi}}(s)\right) \widetilde{\delta^\epsilon X}(s)\right]\bigg] \delta^\epsilon X(s) ds \notag \\
    &-\e\int_t^T \left(P_{t\xi}(s)\right)^\top \widetilde{\e}\left[ D_y\frac{d}{d\nu}D_v b(s,\theta_{t\xi}(s))\left(\widetilde{X_{t\xi}}(s)\right) \widetilde{\delta^\epsilon X}(s)\right]\bigg] \delta^\epsilon v(s) ds \notag \\
    &-\e \int_s^T \left[\begin{pmatrix}
        D_x^2 f & D_xD_v f \\ D_vD_x f & D_v^2 f
    \end{pmatrix} (s,\theta_{t\xi}(s)) \right] \begin{pmatrix}
        \delta^\epsilon X(s) \\ \delta^\epsilon v(s)
    \end{pmatrix}^{\otimes2}  ds  \notag \\
    &-\e \int_s^T \left(\delta^\epsilon X(s)\right)^\top \widetilde{\e}\left[\bigg[ \left(D_y \frac{d}{d\nu}D_x f \right) (s,\theta_{t\xi}(s))\left(\widetilde{X_{t\xi}}(s)\right)\bigg]^\top \widetilde{\delta^\epsilon X}(s) \right]ds \notag \\
    &-\e \int_s^T \left(\delta^\epsilon v(s)\right)^\top \widetilde{\e}\left[\bigg[ \left(D_y \frac{d}{d\nu}D_v f \right) (s,\theta_{t\xi}(s))\left(\widetilde{X_{t\xi}}(s)\right)\bigg]^\top \widetilde{\delta^\epsilon X}(s) \right]ds \notag \\
    &+ \e\int_t^T \bigg[ \left(\delta^\epsilon P(s)\right)^\top \alpha^\epsilon_b(s) +\sum_{j=1}^n  \left(\delta^\epsilon Q^j(s)\right)^\top \left(\alpha^\epsilon_\sigma\right)^j (s) -  \left(\delta^\epsilon X(s)\right)^\top \alpha^\epsilon_f(s)-  \left(\delta^\epsilon v(s)\right)^\top \alpha^\epsilon_v(s)\bigg] ds. \label{lem:4_2}
\end{align}
\normalsize
From the convexity condition \eqref{convex'}, we know that 
\begin{equation}\label{lem:4_3}
\begin{aligned}
    \left[D_x^2 g  (X_{t\xi}(T),\lr(X_{t\xi}(T))) \right] \left(\delta^\epsilon X(T)\right)^{\otimes2} \geq\ & 2\lambda_g \left|\delta^\epsilon X(T)\right|^2;\\
    \text{and}\qquad \left[\begin{pmatrix}
        D_x^2 f & D_xD_v f \\ D_vD_x f & D_v^2 f
    \end{pmatrix} (s,\theta_{t\xi}(s)) \right] \begin{pmatrix}
        \delta^\epsilon X(s) \\ \delta^\epsilon v(s)
    \end{pmatrix}^{\otimes2} \geq\ & 2\lambda_x \left|\delta^\epsilon X(s)\right|^2+ 2\lambda_v \left|\delta^\epsilon v(s)\right|^2.
\end{aligned}
\end{equation}
From Condition \eqref{small_mf_condition}, we know that
\small
\begin{equation}\label{lem:4_4}
\begin{aligned}
    \e\left|\left(\delta^\epsilon X(T)\right)^\top \widetilde{\e}\left[\left[\left(D_y \frac{d}{d\nu}D_x g \right) (X_{t\xi}(T),\lr(X_{t\xi}(T)))\left(\widetilde{X_{t\xi}}(T)\right) \right]^\top \widetilde{\delta^\epsilon X}(T)\right] \right|\le\ & L_g \left\|\delta^\epsilon X(T)\right\|_2^2;\\
    \e \left|\int_s^T \left(\delta^\epsilon X(s)\right)^\top \widetilde{\e}\left[\bigg[ \left(D_y \frac{d}{d\nu}D_x f \right) (s,\theta_{t\xi}(s))\left(\widetilde{X_{t\xi}}(s)\right)\bigg]^\top \widetilde{\delta^\epsilon X}(s) \right]ds \right|\le\ & L_f^0 \int_t^T \left\|\delta^\epsilon X(s)\right\|_2^2 ds; \\
    \qquad \e \left|\int_s^T \left(\delta^\epsilon v(s)\right)^\top \widetilde{\e}\left[\bigg[ \left(D_y \frac{d}{d\nu}D_v f \right) (s,\theta_{t\xi}(s))\left(\widetilde{X_{t\xi}}(s)\right)\bigg]^\top \widetilde{\delta^\epsilon X}(s) \right]ds \right|\le\ & L_f^1 \int_t^T \left\|\delta^\epsilon X(s)\right\|_2\cdot \left\|\delta^\epsilon v(s)\right\|_2 ds.
\end{aligned}
\end{equation}
\normalsize
From Conditions \eqref{generic:condition:b} and \eqref{lem:2_2}, we know that
\begin{equation}\label{lem:4_5}
\begin{aligned}
    &\e \left|\int_t^T \left(P_{t\xi}(s)\right)^\top \left[\begin{pmatrix}
        D_x^2 b & D_xD_v b \\ D_vD_x b & D_v^2 b
    \end{pmatrix}(s,\theta_{t\xi}(s)) \right] \begin{pmatrix}
        \delta^\epsilon X(s) \\ \delta^\epsilon v(s)
    \end{pmatrix}^{\otimes 2} ds  \right| \\
    \le\ & \frac{L^2}{\lambda_b}\int_t^T \left( L_b^0 \left\|\delta^\epsilon X(s)\right\|_2^2+2L_b^1 \left\|\delta^\epsilon X(s)\right\|_2 \cdot \left\|\delta^\epsilon v(s)\right\|_2+L_b^2\left\|\delta^\epsilon v(s)\right\|^2_2 \right) ds; \\
    &\e\left|\int_t^T \left(P_{t\xi}(s)\right)^\top \widetilde{\e}\left[ D_y\frac{d}{d\nu}D_x b(s,\theta_{t\xi}(s))\left(\widetilde{X_{t\xi}}(s)\right) \widetilde{\delta^\epsilon X}(s)\right]\bigg] \delta^\epsilon X(s) ds \right| \\
    \le\ & \frac{L^2L_b^0}{\lambda_b}\int_t^T \left\|\delta^\epsilon X(s)\right\|_2^2 ds; \\
    &\e\left|\int_t^T \left(P_{t\xi}(s)\right)^\top \widetilde{\e}\left[ D_y\frac{d}{d\nu}D_v b(s,\theta_{t\xi}(s))\left(\widetilde{X_{t\xi}}(s)\right) \widetilde{\delta^\epsilon X}(s)\right]\bigg] \delta^\epsilon v(s) ds \right| \\
    \le\ & \frac{L^2L_b^1}{\lambda_b}\int_t^T \left\|\delta^\epsilon X(s)\right\|_2\cdot \left\|\delta^\epsilon v(s)\right\|_2 ds.
\end{aligned}
\end{equation}
From Condition \eqref{lem:4_1}, Assumption (A1) and Condition \eqref{lem:2_2}, we know that
\begin{align}
    \left|\delta^\epsilon P(s)\right|\le \frac{L}{\lambda_b}\left(L+\frac{L^3}{\lambda_b}\right) \Big(|\delta^\epsilon X(s)|+ |\delta^\epsilon v(s)|+\left\|\delta^\epsilon X (s)\right\|_2 \Big) +\frac{L}{\lambda_b} \left|\alpha^\epsilon_v(s)\right| ,\quad s\in[t,T]. \label{lem:4_6}
\end{align}
Then, from Assumption (A1), we know that
\begin{align}
    & \e \left|\int_t^T \left(\delta^\epsilon P(s)\right)^\top \widetilde{\e}\bigg[ D_y \frac{db}{d\nu} (s,\theta_{t\xi}(s))\left(\widetilde{X_{t\xi}}(s)\right) \widetilde{\delta^\epsilon X}(s)\bigg] ds  \right| \notag \\
    \le\ & l_b^m \int_t^T \left\|\delta^\epsilon P(s)\right\|_2\cdot \left\|\delta^\epsilon X(s)\right\|_2ds \notag \\
    \le\ & \frac{Ll_b^m}{\lambda_b}\left(L+\frac{L^3}{\lambda_b}\right) \int_t^T \Big(\left\|\delta^\epsilon v(s)\right\|_2\cdot \left\|\delta^\epsilon X (s)\right\|_2+2\left\|\delta^\epsilon X (s)\right\|_2^2 \Big) ds +\frac{Ll_b^m}{\lambda_b}\int_t^T \left\|\alpha^\epsilon_v(s)\right\|_2\cdot \left\|\delta^\epsilon X(s)\right\|_2 ds. \label{lem:4_7}
\end{align}
From Assumption (A1), we also know that
\begin{align}
    &\sum_{j=1}^n  \e \Bigg|\int_t^T \left(\delta^\epsilon Q^j(s)\right)^\top \widetilde{\e}\bigg[ D_y \frac{d\sigma^j_0}{d\nu} (s,\lr(X_{t\xi}(s)))\left(\widetilde{X_{t\xi}}(s)\right) \widetilde{\delta^\epsilon X}(s)\bigg] ds \Bigg| \notag \\
    \le\ & \sqrt{n}\ l_\sigma^m \left(\int_t^T \left\|\delta^\epsilon X(s)\right\|_2^2 ds\right)^{\frac{1}{2}} \left(\int_t^T \left\|\delta^\epsilon Q(s)\right\|_2^2 ds\right)^{\frac{1}{2}}. \label{sigma_m_7}
\end{align}
By applying It\^o's formula to $\left|\delta^\epsilon P(s)\right|^2$, we know that
\begin{align*}
    &\e\Bigg[\Bigg|\left[D_x^2 g  (X_{t\xi}(T),\lr(X_{t\xi}(T))) \right]^\top \delta^\epsilon X(T) +\alpha^\epsilon_g \\
    &\qquad +\widetilde{\e}\left\{\left[\left(D_y \frac{d}{d\nu}D_x g \right) (X_{t\xi}(T),\lr(X_{t\xi}(T)))\left(\widetilde{X_{t\xi}}(T)\right) \right]^\top \widetilde{\delta^\epsilon X}(T)\right\}\Bigg|^2- \left|\delta^\epsilon P(t)\right|^2\Bigg]\\
    =\ & \e\int_t^T \Bigg\{\sum_{j=1}^n \left|\delta^\epsilon Q^j(s)\right|^2-2\sum_{j=1}^n\left(\delta^\epsilon P(s)\right)^\top \left(\sigma_1^j (s)\right)^\top \delta^\epsilon Q^j(s)-2\left(\delta^\epsilon P(s)\right)^\top D_x b(s,\theta_{t\xi^\epsilon}(s))^\top \delta^\epsilon P(s)  \\
    &\quad\qquad -2 \left(P_{t\xi}(s)\right)^\top\bigg\{D_x^2 b(s,\theta_{t\xi}(s)) \delta^\epsilon X(s)+ D_vD_x b(s,\theta_{t\xi}(s)) \delta^\epsilon v(s)\\
    &\qquad\qquad\qquad\qquad\qquad +\widetilde{\e}\left[ D_y\frac{d}{d\nu}D_x b(s,\theta_{t\xi}(s))\left(\widetilde{X_{t\xi}}(s)\right) \widetilde{\delta^\epsilon X}(s)\right]\bigg\} \delta^\epsilon P(s)\notag\\
    &\quad\qquad -2 \left(\delta^\epsilon P(s)\right)^\top D_x^2 f  (s,\theta_{t\xi}(s))^\top \delta^\epsilon X(s) -2 \left(\delta^\epsilon P(s)\right)^\top D_vD_x f (s,\theta_{t\xi}(s))^\top \delta^\epsilon v(s) \\
    &\quad\qquad -2\left(\delta^\epsilon P(s)\right)^\top \widetilde{\e}\left[\left(D_y \frac{d}{d\nu}D_x f  (s,\theta_{t\xi}(s))\left(\widetilde{X_{t\xi}}(s)\right)\right)^\top \widetilde{\delta^\epsilon X}(s)\right]-2 \left(\delta^\epsilon P(s)\right)^\top \alpha^\epsilon_f(s) \Bigg\}ds,
\end{align*}
then, from the estimate \eqref{lem:2_2} and Assumptions (A1) and (A2), by using Cauchy's inequality, we can deduce that
\begin{align*}
    &\int_t^T\left\|\delta^\epsilon Q(s)\right\|_2^2ds \\
    \le\ & \int_t^T \bigg[2L\sum_{j=1}^n\left\|\delta^\epsilon P(s)\right\|_2\cdot \left\|\delta^\epsilon Q^j(s)\right\|_2+2L\left\|\delta^\epsilon P(s)\right\|_2^2 +2 \left\|\delta^\epsilon P(s)\right\|_2\cdot \left\|\alpha^\epsilon_f (s)\right\|_2 \notag\\
    &\quad\qquad +2L\left(1+\frac{L^2}{\lambda_b}\right) \left(2 \left\|\delta^\epsilon X(s)\right\|_2+ \left\| \delta^\epsilon v(s)\right\|_2 \right)\cdot \left\|\delta^\epsilon P(s)\right\|_2 \bigg]ds+6L^2\|\delta^\epsilon X(T)\|_2^2+3 \left\| \alpha^\epsilon_g \right\|_2^2 \\
    \le\ & \frac{1}{2}\int_t^T \left\|\delta^\epsilon Q(s)\right\|_2^2 ds +6L^2\|\delta^\epsilon X(T)\|_2^2 +3 \left\| \alpha^\epsilon_g \right\|_2^2 \\
    &+ \int_t^T \left[2L(1+nL) \left\|\delta^\epsilon P(s)\right\|_2^2+2L\left(1+\frac{L^2}{\lambda_b}\right) \left[2 \|\delta^\epsilon X(s)\|_2+  \| \delta^\epsilon v(s)\|_2 \right]\cdot \|\delta^\epsilon P(s)\|_2\right]ds.
\end{align*}
Subatituting \eqref{lem:4_6} into the last inequality, we have
\begin{align*}
    &\int_t^T\left\|\delta^\epsilon Q(s)\right\|_2^2ds \\
    \le\ & 12 L^2\|\delta^\epsilon X(T)\|_2^2 + \frac{4L^3}{\lambda_b}\left(1+\frac{L^2}{\lambda_b}\right)^2 \left(3+\frac{4L^2(1+nL)}{\lambda_b}\right) \int_t^T \left(2 \|\delta^\epsilon X(s)\|_2^2+  \| \delta^\epsilon v(s)\|_2^2 \right) ds \\
    &\!\!\!+ 6 \left\| \alpha^\epsilon_g \right\|_2^2 +\int_t^T \bigg[\frac{4L^2}{\lambda_b} \left(1+\frac{L^2}{\lambda_b}\right) \left(2 \|\delta^\epsilon X(s)\|_2+  \| \delta^\epsilon v(s)\|_2 \right)\left\|\alpha^\epsilon_v(s)\right\|_2 +\frac{16 L^3(1+nL)}{\lambda_b^2} \left\|\alpha^\epsilon_v(s)\right\|_2^2\bigg]ds.
\end{align*}
Substituting the last estimate back into \eqref{sigma_m_7}, we know that
\begin{align}
    &\sum_{j=1}^n  \e \Bigg|\int_t^T \left(\delta^\epsilon Q^j(s)\right)^\top \widetilde{\e}\bigg[ D_y \frac{d\sigma^j_0}{d\nu} (s,\lr(X_{t\xi}(s)))\left(\widetilde{X_{t\xi}}(s)\right) \widetilde{\delta^\epsilon X}(s)\bigg] ds \Bigg| \notag \\
    \le\ & 2\sqrt{3n}\:L\:l_\sigma^m \left(\int_t^T \left\|\delta^\epsilon X(s)\right\|_2^2 ds\right)^{\frac{1}{2}} \|\delta^\epsilon X(T)\|_2 \notag \\
    &+\frac{2\sqrt{2nL}\: L\: l_\sigma^m}{\sqrt{\lambda_b}}\left(1+\frac{L^2}{\lambda_b}\right)\sqrt{3+\frac{4L^2(1+nL)}{\lambda_b}} \int_t^T \left\|\delta^\epsilon X(s)\right\|_2^2 ds \notag \\
    &+\frac{2\sqrt{nL}\: L\: l_\sigma^m}{\sqrt{\lambda_b}}\left(1+\frac{L^2}{\lambda_b}\right)\sqrt{3+\frac{4L^2(1+nL)}{\lambda_b}}\left(\int_t^T \left\|\delta^\epsilon X(s)\right\|_2^2 ds\right)^{\frac{1}{2}} \left(\int_t^T \| \delta^\epsilon v(s)\|_2^2 ds\right)^{\frac{1}{2}} \notag \\
    &+\sqrt{6n}\: l_\sigma^m \left(\int_t^T \left\|\delta^\epsilon X(s)\right\|_2^2 ds\right)^{\frac{1}{2}} \left\| \alpha^\epsilon_g \right\|_2 \notag \\
    &+ \frac{2\sqrt{n}\: L\: l_\sigma^m}{\sqrt{\lambda_b}} \sqrt{1+\frac{L^2}{\lambda_b}} \left(\int_t^T \left\|\delta^\epsilon X(s)\right\|_2^2 ds\right)^{\frac{1}{2}} \left[\int_t^T \left(2 \|\delta^\epsilon X(s)\|_2+  \| \delta^\epsilon v(s)\|_2 \right)\left\|\alpha^\epsilon_v(s)\right\|_2 ds\right]^{\frac{1}{2}} \notag \\
    &+\frac{4\sqrt{nL(1+nL)}\:L\: l_\sigma^m}{\lambda_b}\left(\int_t^T \left\|\delta^\epsilon X(s)\right\|_2^2 ds\right)^{\frac{1}{2}} \left(\int_t^T \left\|\alpha^\epsilon_v(s)\right\|_2^2 ds\right)^{\frac{1}{2}}. \label{sigma_m_8}
\end{align}
Substituting \eqref{lem:4_3}, \eqref{lem:4_4}, \eqref{lem:4_5} and \eqref{lem:4_7} and \eqref{sigma_m_8} back into \eqref{lem:4_2}, we can deduce that
\small
\begin{align*}
    &\left(2\lambda_g-L_g\right) \left\|\delta^\epsilon X(T)\right\|_2^2 +\left(2\lambda_v-\frac{L^2L_b^2}{\lambda_b}\right)\int_t^T \left\|\delta^\epsilon v(s)\right\|_2^2 ds\\
    &+\left(2\lambda_x -L_f^0-\frac{2L^2}{\lambda_b} \left(L_b^0+l_b^m+\frac{L^2l_b^m}{\lambda_b}\right)  -\frac{2\sqrt{2nL} L l_\sigma^m}{\sqrt{\lambda_b}}\left(1+\frac{L^2}{\lambda_b}\right)\sqrt{3+\frac{4L^2(1+nL)}{\lambda_b}}\right)\int_t^T \left\|\delta^\epsilon X(s)\right\|_2^2 ds  \notag \\
    \le\ & I^\epsilon+ 2\sqrt{3n} L l_\sigma^m \left(\int_t^T \left\|\delta^\epsilon X(s)\right\|_2^2 ds\right)^{\frac{1}{2}} \|\delta^\epsilon X(T)\|_2 \notag \\
    &+\left(L_f^1+\frac{L^2}{\lambda_b}\left({3L_b^1}+{l_b^m}+\frac{L^2l_b^m}{\lambda_b}\right)   +\frac{2\sqrt{nL} L l_\sigma^m}{\sqrt{\lambda_b}}\left(1+\frac{L^2}{\lambda_b}\right)\sqrt{3+\frac{4L^2(1+nL)}{\lambda_b}}  \right)\\
    &\qquad \cdot \left(\int_t^T \left\|\delta^\epsilon X(s)\right\|_2^2 ds\right)^{\frac{1}{2}} \cdot \left(\int_t^T \| \delta^\epsilon v(s)\|_2^2 ds\right)^{\frac{1}{2}},
\end{align*}
\normalsize
where
\small
\begin{align*}
    I^\epsilon:=&\int_t^T \bigg[ \left\|\delta^\epsilon P(s)\right\|\cdot \left\|\alpha^\epsilon_b(s)\right\|_2  + \left\|\delta^\epsilon Q(s)\right\|_2 \cdot \left\|\alpha^\epsilon_\sigma(s)\right\|_2 +  \left\|\delta^\epsilon X(s)\right\|_2 \cdot \left\|\alpha^\epsilon_f(s)\right\|_2 +  \left\|\delta^\epsilon v(s)\right\|_2 \cdot \left\|\alpha^\epsilon_v(s)\right\|_2 \\
    &\qquad\qquad +\frac{Ll_b^m}{\lambda_b} \left\|\alpha^\epsilon_v(s)\right\|_2\cdot \left\|\delta^\epsilon X(s)\right\|_2 \bigg] ds\\
    & +\sqrt{6n} l_\sigma^m \left(\int_t^T \left\|\delta^\epsilon X(s)\right\|_2^2 ds\right)^{\frac{1}{2}} \left\| \alpha^\epsilon_g \right\|_2 +\frac{4\sqrt{nL(1+nL)}L l_\sigma^m}{\lambda_b}\left(\int_t^T \left\|\delta^\epsilon X(s)\right\|_2^2 ds\right)^{\frac{1}{2}} \left(\int_t^T \left\|\alpha^\epsilon_v(s)\right\|_2^2 ds\right)^{\frac{1}{2}} \notag \\
    & + \frac{2\sqrt{n}L l_\sigma^m}{\sqrt{\lambda_b}} \sqrt{1+\frac{L^2}{\lambda_b}} \left(\int_t^T \left\|\delta^\epsilon X(s)\right\|_2^2 ds\right)^{\frac{1}{2}} \left[\int_t^T \left(2 \|\delta^\epsilon X(s)\|_2+  \| \delta^\epsilon v(s)\|_2 \right)\left\|\alpha^\epsilon_v(s)\right\|_2 ds\right]^{\frac{1}{2}}.
\end{align*}
\normalsize
Then, from Condition \eqref{lem:2_0} and Young's inequality, we can deduce that
\begin{align}\label{lem:4_8}
    \int_t^T \left\|\delta^\epsilon v(s)\right\|_2^2 ds\le C I^\epsilon.
\end{align}
From Assumptions (A1) and (A2), the estimates \eqref{lem4_2} and \eqref{lem4_2'}, and the dominated convergence theorem, we can deduce that 
\begin{align*}
    \lim_{\epsilon\to0} \int_t^T \left\|\alpha^\epsilon_b(s)\right\|_2^2+\left\|\alpha^\epsilon_\sigma(s)\right\|_2^2+\left\|\alpha^\epsilon_v(s)\right\|_2^2+\left\|\alpha^\epsilon_v(s)\right\|_2^2  ds=0,\quad \lim_{\epsilon\to0} \left\|\alpha^\epsilon_v(s)\right\|_2=0,
\end{align*}
with which we have 
\begin{align}\label{lem:4_11}
	\lim_{\epsilon\to0} I^\epsilon=0.
\end{align}
From \eqref{lem:4_8} and \eqref{lem:4_11}, we know that
\begin{align}\label{lem:4_9}
    \lim_{\epsilon\to0} \int_t^T \left\|\delta^\epsilon v(s)\right\|_2^2 ds =0. 
\end{align}
With a similar approach to \eqref{lem:4_9}, the Gr\"onwall's inequality for the SDE for $\delta^\epsilon X$ and the BDG inequality for the BSDE for $\left(\delta^\epsilon P,\delta^\epsilon Q\right)$, we deduce
\begin{equation}\label{lem:4_10}
	\begin{split}
		&\lim_{\epsilon\to0}\e\bigg[\sup_{t\le s\le T}\left|\left(\delta^\epsilon X(s),\delta^\epsilon P(s)\right)^\top\right|^2 + \int_t^T \left|\delta^\epsilon Q(s) \right|^2ds \bigg]=0.
	\end{split}    	
\end{equation}    
From Condition \eqref{lem:4_1}, Assumption (A1), the convexity of $f$ in $v$, and Condition \eqref{lem:2_2}, we have
\begin{align*}
    \left(2\lambda_v - \frac{L^2L_b^2}{\lambda_b}\right) \left|\delta^\epsilon v(s)\right|\le\ &  \left(L+\frac{L^2L_b^1}{\lambda_b}\right) \left|\delta^\epsilon X(s)\right|+ \left(L_f^1+\frac{L^2L_b^1}{\lambda_b}\right) \left\|\delta^\epsilon X(s) \right\|_2+L \left|\delta^\epsilon P(s)\right|+\left|\alpha^\epsilon_v(s)\right|.
\end{align*}
Then, from \eqref{lem:2_0}, \eqref{lem:4_11} and \eqref{lem:4_10}, we obtain \eqref{lem4_1}. The desired G\^ateaux differentiability is then a consequence of \eqref{lem4_1} and Lemma~\ref{lem:3}. The proof of the continuity in $\xi$ of the process $(\mathcal{X}(s),\mathcal{P}(s),\mathcal{Q}(s),\mathcal{V}(s))$ is similar to that of \eqref{lem4_1}, and is omitted here.

\subsection{Proof of Lemma~\ref{prop:4}}\label{pf:prop:4}

The solvability of FBSDEs \eqref{FB:dr'} can be shown in a similar manner as that for Lemma~\ref{lem:3}, and so we omit it. We only prove \eqref{prop4_1} here. From the uniqueness of the solution of FBSDEs \eqref{FB:mfg_generic}, we see $\Theta_{tx\mu}(s)|_{x=\xi}=\Theta_{t\xi}(s),\ s\in[t,T]$; see \cite{BR} for details. Then, from FBSDEs \eqref{FB:x} and \eqref{FB:dr'}, we know that the sum of the processes
\begin{align*}
	D_x\Theta_{tx\mu}(s)\big|_{x=\xi}\ \eta+\bd_\eta \Theta_{t\xi}(s) 
\end{align*}
satisfies the following FBSDEs 
\small
\begin{align*}
		&D_xX_{tx\mu}(s)\big|_{x=\xi}\ \eta+\bd_\eta X_{t\xi}(s)\\
        =\ & \eta+\int_t^s \bigg\{D_x b\left(r,\theta_{t\xi}(r)\right)\left(D_xX_{tx\mu}(r)\big|_{x=\xi}\ \eta+\bd_\eta X_{t\xi}(r)\right)\\
        &\qquad\qquad + D_v b\left(r,\theta_{t\xi}(r)\right)\left(D_xv_{tx\mu}(r)\big|_{x=\xi}\ \eta +\bd_\eta v_{t\xi}(r)\right) \notag \\
        &\qquad\qquad +\widetilde{\e}\bigg[ D_y \frac{db}{d\nu} (r,\theta_{t\xi}(r))\left(\widetilde{X_{t\xi}}(r)\right) \left(\widetilde{D_x X_{tx\mu}}(r)\left|_{x=\widetilde{\xi}}\ \right. \widetilde{\eta}+\widetilde{\bd_\eta X_{t\xi}}(r)\right)\bigg] \bigg\} dr\\
        & +\int_t^s \bigg\{ \sigma_1(r) \left(D_xX_{tx\mu}(r)\big|_{x=\xi}\ \eta +\bd_\eta X_{t\xi}(r)\right)\\
        &\qquad\qquad +\widetilde{\e}\bigg[ D_y \frac{d\sigma_0}{d\nu} (r,\lr(X_{t\xi}(r)))\left(\widetilde{X_{t\xi}}(r)\right) \left(\widetilde{D_x X_{tx\mu}}(r)\left|_{x=\widetilde{\xi}}\ \right. \widetilde{\eta}+\widetilde{\bd_\eta X_{t\xi}}(r)\right)\bigg] \bigg\} dB(r) \notag \\
		\text{and} \qquad &D_xP_{tx\mu}(s)\big|_{x=\xi}\ \eta+\bd_\eta P_{t\xi}(s)\\
        =\ & D_x^2 g\left(X_{t\xi}(T),\lr(X_{t\xi}(T))\right)^\top \left(D_xX_{tx\mu}(T)\big|_{x=\xi}\ \eta+\bd_\eta X_{t\xi}(T) \right) \notag \\
        &+\widetilde{\e}\left[\left(D_y \frac{d}{d\nu}D_x g (X_{t\xi}(T),\lr(X_{t\xi}(T)))\left(\widetilde{X_{t\xi}}(T)\right) \right)^\top \left( \widetilde{D_xX_{tx\mu}}(T)\left|_{x=\widetilde{\xi}}\ \right. \widetilde{\eta}+\widetilde{\bd_\eta X_{t\xi}}(T)\right)\right] \\
        & +\int_s^T \Bigg\{ D_x b\left(r,\theta_{t\xi}(r)\right)^\top \left(D_x P_{tx\mu}(r)\big|_{x=\xi}\ \eta+\bd_\eta P_{t\xi}(r)\right) \\
        &\qquad\qquad +\sum_{j=1}^n\left(\sigma_1^j(r)\right)^\top \left(D_x Q_{tx\mu}^j(r)\big|_{x=\xi}\ \eta +\bd_\eta Q_{t\xi}^j(r)\right) \notag \\
        &\qquad\qquad +\left(P_{t\xi}(r)\right)^\top \bigg\{D_x^2 b\left(r,\theta_{t\xi}(r)\right)\left(D_xX_{tx\mu}(r)\big|_{x=\xi}\ \eta +\bd_\eta X_{t\xi}(r)\right) \\
        &\qquad\qquad\qquad\qquad\qquad + D_v D_x b\left(r,\theta_{t\xi}(r)\right) \left(D_xv_{tx\mu}(r)\big|_{x=\xi}\ \eta+\bd_\eta v_{t\xi}(r)\right) \notag \\
        &\qquad\qquad\qquad\qquad\qquad +\widetilde{\e}\left[ D_y\frac{d}{d\nu}D_x b(r,\theta_{t\xi}(r))\left(\widetilde{X_{t\xi}}(r)\right) \left( \widetilde{D_xX_{tx\mu}}(r)\left|_{x=\widetilde{\xi}}\ \right. \widetilde{\eta}+\widetilde{\bd_\eta X_{t\xi}}(r)\right)\right] \bigg\}\\
        &\qquad\qquad +\left[D_x^2 f (r,\theta_{t\xi}(r)) \right]^\top \left(D_x X_{tx\mu}(r)\big|_{x=\xi}\ \eta+\bd_\eta X_{t\xi}(r)\right) \\
        &\qquad\qquad + \left[D_vD_x f (r,\theta_{t\xi}(r))\right]^\top \left(D_x v_{tx\mu}(r) \big|_{x=\xi}\ \eta+\bd_\eta v_{t\xi}(r)\right)\\
        &\qquad\qquad +\widetilde{\e}\left[\left(D_y \frac{d}{d\nu}D_x f  (r,\theta_{t\xi}(r))\left(\widetilde{X_{t\xi}}(r)\right)\right)^\top  \left(\widetilde{D_xX_{tx\mu}}(r)\left|_{x=\widetilde{\xi}}\ \right. \widetilde{\eta}+\widetilde{\bd_\eta X_{t\xi}}(r)\right)\right]\Bigg\}dr\\
        &-\int_s^T \left(D_x Q_{tx\mu} (r)\big|_{x=\xi}\ \eta+\bd_\eta Q_{t\xi}(r)\right) dB(r),\quad s\in[t,T].
\end{align*}
\normalsize
Multiplying Condition \eqref{FB:x_condition} by $\eta$, in view of  Condition \eqref{FB:dr'_condition}, we see that 
\begin{align*}
	0=\ &\left[D_xD_v f (s,\theta_{t\xi}(s)) \right]^\top \left(D_x X_{tx\mu}(s)\big|_{x=\xi}\ \eta+\bd_\eta X_{t\xi}(s)\right) \\
    &+ \left[D_v^2 f (s,\theta_{t\xi}(s))\right]^\top \left(D_x v_{tx\mu}(s)\big|_{x=\xi}\ \eta+\bd_\eta v_{t\xi}(s)\right)\\
    &+\widetilde{\e}\left[\left(D_y \frac{d}{d\nu}D_v f  (s,\theta_{t\xi}(s))\left(\widetilde{X_{t\xi}}(s)\right)\right)^\top \left(\widetilde{D_xX_{tx\mu}}(s)\left|_{x=\widetilde{\xi}}\ \right. \widetilde{\eta}+\widetilde{\bd_\eta X_{t\xi}}(s)\right)\right]\\
    &+D_v b(s,\theta_{t\xi}(s))^\top \left(D_xP_{tx\mu}(s)\big|_{x=\xi}\ \eta +\bd_\eta P_{t\xi}(s)\right) \notag \\
    &+\left(P_{t\xi}(s)\right)^\top \bigg\{D_xD_v b\left(s,\theta_{t\xi}(s)\right) \left(D_xX_{tx\mu}(s)\big|_{x=\xi}\ \eta +\bd_\eta X_{t\xi}(s)\right)\\
    &\qquad\qquad\qquad + D_v^2  b\left(s,\theta_{t\xi}(s)\right) \left(D_xv_{tx\mu}(s)\big|_{x=\xi}\ \eta+\bd_\eta v_{t\xi}(s)\right) \\
    &\qquad\qquad\qquad + \widetilde{\e}\left[ D_y\frac{d}{d\nu}D_v b(s,\theta_{t\xi}(s))\left(\widetilde{X_{t\xi}}(s) \right) \left(\widetilde{D_xX_{tx\mu}}(s)\left|_{x=\widetilde{\xi}}\ \right. \widetilde{\eta}+\widetilde{\bd_\eta X_{t\xi}}(s)\right)\right] \bigg\}.
\end{align*}
All the above are exactly the FBSDEs system \eqref{FB:dr}-\eqref{FB:dr_condition} for $D_\eta \Theta_{t\xi}$. From the uniqueness of the solution of FBSDEs \eqref{FB:dr}-\eqref{FB:dr_condition}, we deduce \eqref{prop4_1}.

\subsection{Proof of Theorem~\ref{prop:3}}\label{pf:prop:3}

In view of Lemma~\ref{prop:4}, the system of FBSDEs \eqref{FB:mu'} also reads, for $s\in[t,T]$,
\small
\begin{align}
		D_\eta X_{tx\xi}(s)=\ & \int_t^s \bigg\{ \widetilde{\e}\bigg[ D_y \frac{db}{d\nu} (r,\theta_{tx\mu}(r))\left(\widetilde{X_{t\xi}}(r)\right) \widetilde{D_\eta X_{t\xi}}(r)\bigg] \notag \\
        &\qquad +D_x b(r, \theta_{tx\mu}(r))D_\eta X_{tx\xi}(s)+D_v b(r, \theta_{tx\mu}(r))D_\eta v_{tx\xi}(r)  \bigg\}dr \notag \\
        &+\int_t^s \bigg\{ 
		\widetilde{\e}\bigg[ D_y \frac{d\sigma_0}{d\nu} (r,\lr(X_{t\xi}(r)))\left(\widetilde{X_{t\xi}}(r)\right) \widetilde{D_\eta X_{t\xi}}(r)\bigg]+ \sigma_1(r)D_\eta X_{tx\xi}(r) \bigg\}dB(r), \notag \\[3mm]
		D_\eta P_{tx\xi}(s)=\ &  -\int_s^TD_\eta Q_{tx\xi}(r)dB(r)+D_x^2 g (X_{tx\mu}(T),\lr(X_{t\xi}(T)))^\top  D_\eta X_{tx\xi}(T) \notag\\
        &+\widetilde{\e}\left[\left(D_y \frac{d}{d\nu}D_x g (X_{tx\mu}(T),\lr(X_{t\xi}(T)))\left(\widetilde{X_{t\xi}}(T)\right) \right)^\top  \widetilde{D_\eta X_{t\xi}}(T) \right] \notag \\
		&+\int_s^T\Bigg\{ D_xb(r,\theta_{tx\mu}(r))^\top D_\eta P_{tx\xi}(r)+\sum_{j=1}^n\left(\sigma_1^j (r)\right)^\top D_\eta Q_{tx\xi}^j(r) \notag \\
        &\quad\qquad + \left(P_{tx\mu}(r) \right)^\top \bigg\{\widetilde{\e}\left[ D_y\frac{d}{d\nu}D_x b(r,\theta_{tx\mu}(r))\left(\widetilde{X_{t\xi}}(r)\right) \widetilde{D_\eta X_{t\xi}}(r)\right] \notag \\
        &\qquad\qquad\qquad\qquad\qquad +D_x^2 b(r,\theta_{tx\mu}(r)) D_\eta X_{tx\xi}(r)+ D_vD_x b(r,\theta_{tx\mu}(r)) D_\eta v_{tx\xi}(r) \bigg\} \notag\\
        &\quad\qquad +\widetilde{\e}\left[\left(D_y \frac{d}{d\nu}D_x f  (r,\theta_{tx\mu}(r))\left(\widetilde{X_{t\xi}}(r)\right)\right)^\top \widetilde{D_\eta X_{t\xi}}(r)\right] \notag \\
        &\quad\qquad +D_x^2 f  (r,\theta_{tx\mu}(r))^\top D_\eta X_{tx\xi}(r) + D_vD_x f (r,\theta_{tx\mu}(r))^\top  D_\eta v_{tx\xi}(r)\Bigg\}dr;  \label{FB:mu} 
\end{align}
\normalsize
and in view of Conditions \eqref{FB:x_condition}, \eqref{FB:dr'_condition} and Lemma~\ref{prop:4}, Condition \eqref{FB:mu'_condition} also reads
\small
\begin{align}
    0=\ & \widetilde{\e}\left[\left(D_y \frac{d}{d\nu}D_v f  (s,\theta_{tx\mu}(s))\left(\widetilde{X_{t\xi}}(s)\right)\right)^\top \widetilde{D_\eta X_{t\xi}}(s)\right] \notag \\
    &+D_x D_v f  (s,\theta_{tx\mu}(s))^\top D_\eta X_{tx\xi}(s) +D_v^2 f (s,\theta_{tx\mu}(s))^\top  D_\eta v_{tx\xi}(s)+D_v b(s,\theta_{tx\mu}(s))^\top D_\eta P_{tx\xi}(s) \notag\\
    &+ \left(P_{tx\mu}(s)\right)^\top \bigg\{\widetilde{\e}\left[ D_y\frac{d}{d\nu}D_v b(s,\theta_{tx\mu}(s))\left(\widetilde{X_{t\xi}}(s) \right) \widetilde{D_\eta X_{t\xi}}(s)\right] \notag \\
    &\quad\qquad\qquad\qquad +D_x D_v b(s,\theta_{tx\mu}(s)) D_\eta X_{tx\xi}(s)+ D_v^2 b(s,\theta_{tx\mu}(s)) D_\eta v_{tx\xi}(s)\bigg\}. \label{FB:mu_condition}
\end{align}
\normalsize
Then, the proof of the well-posedness of the above FBSDEs is similar to that for Lemma~\ref{lem:3}, and the proof of the G\^ateaux differentiability is similar to that for Theorem~\ref{lem:4}, and they are all omitted here. In view of FBSDEs \eqref{FB:mu}-\eqref{FB:mu_condition}, the estimate \eqref{prop3_1} is derived in a similar way to  that of \eqref{lem3_1} (for the solution $D_\eta \Theta_{t\xi}$ of FBSDEs \eqref{FB:dr}-\eqref{FB:dr_condition}), with the only difference that the latter FBSDEs only depends on the process $\theta_{t\xi}$ while the former depends on both processes $\theta_{tx\mu}$ and $(\widetilde{X_{t\xi}},\widetilde{D_\eta X_{t\xi}})$. For the reader's convenience, we prove \eqref{prop3_1} in details here. Applying It\^o's formula to $D_\eta P_{tx\xi}(s)^\top D_\eta X_{tx\xi}(s)$, similar to \eqref{lem:3_1}, we have
\begin{align*}
    &\e\Bigg\{D_x^2 g (X_{tx\mu}(T),\lr(X_{t\xi}(T))) \left(  D_\eta X_{tx\xi}(T)\right)^{\otimes 2} \\
    &\quad +\widetilde{\e}\left[\left(\widetilde{D_\eta X_{t\xi}}(T)\right)^\top\left(D_y \frac{d}{d\nu}D_x g (X_{tx\mu}(T),\lr(X_{t\xi}(T)))\left(\widetilde{X_{t\xi}}(T)\right) \right) \right]D_\eta X_{tx\xi}(T) \Bigg\}\\
    =\ & \e \int_t^T \Bigg\{ D_\eta P_{tx\xi}(s)^\top \widetilde{\e}\bigg[ D_y \frac{db}{d\nu} (s,\theta_{tx\mu}(s))\left(\widetilde{X_{t\xi}}(s)\right) \widetilde{D_\eta X_{t\xi}}(s)\bigg]+ D_\eta P_{tx\xi}(s)^\top D_v b(s, \theta_{tx\mu}(s))D_\eta v_{tx\xi}(s)\\
    &\quad\qquad  +\sum_{j=1}^n \left(D_\eta Q_{tx\xi}^j(s)\right)^\top \widetilde{\e}\bigg[ D_y \frac{d\sigma_0^j}{d\nu} (s,\lr(X_{t\xi}(s)))\left(\widetilde{X_{t\xi}}(s)\right) \widetilde{D_\eta X_{t\xi}}(s)\bigg] \\ 
    &\quad\qquad -P_{tx\mu}(s)^\top \bigg[D_x^2 b(s,\theta_{tx\mu}(s)) D_\eta X_{tx\xi}(s)+\widetilde{\e}\left[ D_y\frac{d}{d\nu}D_x b(s,\theta_{tx\mu}(s))\left(\widetilde{X_{t\xi}}(s) \right)\widetilde{D_\eta X_{t\xi}}(s)\right]\\
    &\qquad\qquad\qquad\qquad + D_vD_x b(s,\theta_{tx\mu}(s)) D_\eta v_{tx\xi}(s)\bigg]  D_\eta X_{tx\xi}(s) \\
    &\quad\qquad - D_\eta X_{tx\xi}(s)^\top D_x^2 f  (s,\theta_{tx\mu}(s))^\top D_\eta X_{tx\xi}(s) -D_\eta X_{tx\xi}(s)^\top D_vD_x f (s,\theta_{tx\mu}(s))^\top  D_\eta v_{tx\xi}(s)\\
    &\quad\qquad -D_\eta X_{tx\xi}(s)^\top \widetilde{\e}\left[\left(D_y \frac{d}{d\nu}D_x f (s,\theta_{tx\mu}(s))\left(\widetilde{X_{t\xi}}(s)\right)\right)^\top \widetilde{D_\eta X_{t\xi}}(s)\right] \Bigg\}ds.
\end{align*}
Substituting Condition \eqref{FB:mu_condition} into this last equality, we deduce that
\begin{align}
    &\e\Bigg\{D_x^2 g (X_{tx\mu}(T),\lr(X_{t\xi}(T))) \left(  D_\eta X_{tx\xi}(T)\right)^{\otimes 2} \notag \\
    &\quad +\widetilde{\e}\left[\left(\widetilde{D_\eta X_{t\xi}}(T)\right)^\top\left(D_y \frac{d}{d\nu}D_x g (X_{t\xi}(T),\lr(X_{t\xi}(T)))\left(\widetilde{X_{t\xi}}(T)\right) \right) \right]D_\eta X_{tx\xi}(T) \Bigg\}\notag \\
    =\ & \e \int_t^T \Bigg\{ D_\eta P_{tx\xi}(s)^\top \widetilde{\e}\bigg[ D_y \frac{db}{d\nu} (s,\theta_{tx\mu}(s))\left(\widetilde{X_{t\xi}}(s)\right) \widetilde{D_\eta X_{t\xi}}(s)\bigg]\notag \\
    &\quad\qquad  +\sum_{j=1}^n \left(D_\eta Q_{tx\xi}^j(s)\right)^\top \widetilde{\e}\bigg[ D_y \frac{d\sigma_0^j}{d\nu} (s,\lr(X_{t\xi}(s)))\left(\widetilde{X_{t\xi}}(s)\right) \widetilde{D_\eta X_{t\xi}}(s)\bigg] \\ 
    &\quad\qquad - P_{tx\mu}(s)^\top \bigg\{ \left[\begin{pmatrix}
        D_v^2 b & D_vD_xb\\ D_xD_v b & D_x^2b
    \end{pmatrix}(s,\theta_{tx\mu}(s))\right]\begin{pmatrix}
        D_\eta v_{tx\xi}(s)\\ D_\eta X_{tx\xi}(s)
    \end{pmatrix}^{\otimes 2} \notag \\
    &\quad\qquad\qquad\qquad\qquad +\widetilde{\e}\left[ D_y\frac{d}{d\nu}D_v b(s,\theta_{tx\mu}(s))\left(\widetilde{X_{t\xi}}(s)\right) \widetilde{D_\eta X_{t\xi}}(s)\right] D_\eta v_{tx\xi}(s)  \notag \\
    &\quad\qquad\qquad\qquad\qquad +\widetilde{\e}\left[ D_y\frac{d}{d\nu}D_x b(s,\theta_{tx\mu}(s))\left(\widetilde{X_{t\xi}}(s)\right) \widetilde{D_\eta X_{t\xi}} \right] D_\eta X_{tx\xi}(s) \bigg\} \notag \\
    &\quad\qquad -\left[\begin{pmatrix}
        D_v^2 f & D_vD_xf\\ D_xD_v f & D_x^2f
    \end{pmatrix}(s,\theta_{tx\mu}(s))\right]\begin{pmatrix}
        D_\eta v_{tx\xi}(s)\\ D_\eta X_{tx\xi}(s)
    \end{pmatrix}^{\otimes 2} \notag \\
    &\quad\qquad -D_\eta X_{tx\xi}(s)^\top \widetilde{\e}\left[\left(D_y \frac{d}{d\nu}D_x f  (s,\theta_{tx\mu}(s))\left(\widetilde{X_{t\xi}}(s)\right)\right)^\top \widetilde{D_\eta X_{t\xi}}(s)\right] \notag  \\
    &\quad\qquad -D_\eta v_{tx\xi}(s)^\top \widetilde{\e}\left[\left(D_y \frac{d}{d\nu}D_v f  (s,\theta_{tx\mu}(s))\left(\widetilde{X_{t\xi}}(s)\right)\right)^\top \widetilde{D_\eta X_{t\xi}}(s)\right] \Bigg\}ds. \label{prop:3_1}
\end{align}
From Condition \eqref{FB:mu_condition} and Assumption (A1), we also know that
\begin{align}
    D_\eta P_{tx\xi}(s)= -&\left[\left((D_v b)(D_v b)^\top\right)^{-1}(D_vb) (s,\theta_{tx\mu}(s))\right]  \notag \\
    &\cdot \Bigg\{ D_x D_v f (s,\theta_{tx\mu}(s))^\top D_\eta X_{tx\xi}(s) +D_v^2 f (s,\theta_{tx\mu}(s))^\top  D_\eta v_{tx\xi}(s) \notag \\
    &\qquad +\widetilde{\e}\left[\left(D_y \frac{d}{d\nu}D_v f  (s,\theta_{tx\mu}(s))\left(\widetilde{X_{t\xi}}(s)\right)\right)^\top \widetilde{D_\eta X_{t\xi}}(s)\right] \notag\\
    &\qquad + \left(P_{tx\mu}(s)\right)^\top \bigg\{D_x D_v b(s,\theta_{tx\mu}(s)) D_\eta X_{tx\xi}(s) + D_v^2 b(s,\theta_{tx\mu}(s)) D_\eta v_{tx\xi}(s) \notag \\
    &\quad\qquad\qquad\qquad\qquad +\widetilde{\e}\left[ D_y\frac{d}{d\nu}D_v b(s,\theta_{tx\mu}(s))\left(\widetilde{X_{t\xi}}(s) \right) \widetilde{D_\eta X_{t\xi}}(s)\right] \bigg\} \Bigg\}, \label{prop:3_2'}
\end{align}
and then, from the assumptions (A2), \eqref{generic:condition:b}, \eqref{generic:condition:b'} and the estimate \eqref{lem:5_2}, we have
\begin{align}\label{prop:3_2}
    |D_\eta P_{tx\xi}(s)|\le &\frac{L^2}{\lambda_b} \left(1+\frac{L^2}{\lambda_b}\right) \left(|D_\eta X_{tx\xi}(s)|+ \left\|D_\eta X_{t\xi}(s)\right\|_2 +|D_\eta v_{tx\xi}(s)|\right).
\end{align}
Similar estimate as \eqref{prop:3_2} is also used in \cite{AB10'}, which is motivated by the idea of dividing the time span of temporal variable; while here is motivated by that of dividing domain of the spatial variable. From \eqref{prop:3_2} and Assumption (A1), we also see that
\begin{align}
    &\left|D_\eta P_{tx\xi}(s)^\top \widetilde{\e}\bigg[ D_y \frac{db}{d\nu} (s,\theta_{tx\mu}(s))\left(\widetilde{X_{t\xi}}(s)\right) \widetilde{D_\eta X_{t\xi}}(s)\bigg]\right| \notag \\
    \le\ & \frac{L^2l_b^m}{\lambda_b} \left(1+\frac{L^2}{\lambda_b}\right) \left(|D_\eta X_{tx\xi}(s)|+ \left\|D_\eta X_{t\xi}(s)\right\|_2 +|D_\eta v_{tx\xi}(s)|\right)\left\|D_\eta X_{t\xi}(s)\right\|_2. \label{prop:3_3}
\end{align}
From the estimate \eqref{lem:5_2} and the assumption \eqref{generic:condition:b}, we know that 
\begin{align}
    &\Bigg| P_{tx\mu}(s)^\top \bigg\{ \left[\begin{pmatrix}
        D_v^2 b & D_vD_xb\\ D_xD_v b & D_x^2b
    \end{pmatrix}(s,\theta_{tx\mu}(s))\right]\begin{pmatrix}
        D_\eta v_{tx\xi}(s)\\ D_\eta X_{tx\xi}(s)
    \end{pmatrix}^{\otimes 2} \notag \\
    &\qquad\qquad +\widetilde{\e}\left[ D_y\frac{d}{d\nu}D_v b(s,\theta_{tx\mu}(s))\left(\widetilde{X_{t\xi}}(s)\right) \widetilde{D_\eta X_{t\xi}}(s)\right] D_\eta v_{tx\xi}(s)  \notag \\
    &\qquad\qquad +\widetilde{\e}\left[ D_y\frac{d}{d\nu}D_x b(s,\theta_{tx\mu}(s))\widetilde{X_{t\xi}}(s) \widetilde{D_\eta X_{t\xi}}(s)\right]  D_\eta X_{tx\xi}(s) \bigg\}\Bigg| \notag \\
    \le\ & \frac{L^2}{\lambda_b} \Big[L_b^0 |D_\eta X_{tx\xi}(s)|\left(|D_\eta X_{tx\xi}(s)|+ \left\|D_\eta X_{t\xi}(s)\right\|_2\right) \notag \\
    &\qquad +L_b^1 |D_\eta v_{tx\xi}(s)| \left(2|D_\eta X_{tx\xi}(s)|+ \left\|D_\eta X_{t\xi}(s)\right\|_2\right) +L_b^2 |D_\eta v_{tx\xi}(s)|^2\Big]. \label{prop:3_4}
\end{align}
From the convexity condition \eqref{convex'}, we have
\begin{equation}\label{prop:3_5}
\begin{aligned}
    \left[\begin{pmatrix}
        D_v^2 f & D_vD_xf\\ D_xD_v f & D_x^2f
    \end{pmatrix}(s,\theta_{tx\mu}(s))\right]\begin{pmatrix}
        D_\eta v_{tx\xi}(s)\\ D_\eta X_{tx\xi}(s)
    \end{pmatrix}^{\otimes 2}\geq \ & 2\lambda_v |D_\eta v_{tx\xi}(s)|^2+2\lambda_x |D_\eta X_{tx\xi}(s)|^2;\\
    \text{and}\qquad D_x^2 g (X_{tx\mu}(T),\lr(X_{t\xi}(T))) \left(  D_\eta X_{tx\xi}(T)\right)^{\otimes 2} \geq\ & 2\lambda_g |D_\eta X_{tx\xi}(T)|^2.
\end{aligned}
\end{equation}
From the condition \eqref{small_mf_condition}, we know that
\small
\begin{align}
    &\left|D_\eta X_{tx\xi}(s)^\top \widetilde{\e}\left[\left(D_y \frac{d}{d\nu}D_x f  (s,\theta_{tx\mu}(s))\left(\widetilde{X_{t\xi}}(s)\right)\right)^\top \widetilde{D_\eta X_{t\xi}}(s)\right]\right| \le L_f^0\  |D_\eta X_{tx\xi}(s)|\cdot \left\|D_\eta X_{t\xi}(s)\right\|_2; \notag \\
    &\left|D_\eta v_{tx\xi}(s)^\top \widetilde{\e}\left[\left(D_y \frac{d}{d\nu}D_v f  (s,\theta_{tx\mu}(s))\left(\widetilde{X_{t\xi}}(s)\right)\right)^\top \widetilde{D_\eta X_{t\xi}}(s)\right] \right| \le L_f^1 \  |D_\eta v_{tx\xi}(s)|\cdot \left\|D_\eta X_{t\xi}(s)\right\|_2; \notag \\
    &\left| \widetilde{\e}\left[\left(\widetilde{D_\eta X_{t\xi}}(T)\right)^\top\left(D_y \frac{d}{d\nu}D_x g (X_{tx\mu}(T),\lr(X_{t\xi}(T)))\left(\widetilde{X_{t\xi}}(T)\right) \right) \right]D_\eta X_{tx\xi}(T)\right| \notag \\
    &\qquad\qquad\qquad\qquad\qquad\qquad\qquad\qquad\qquad\qquad\qquad\qquad\qquad\qquad \le L_g   |D_\eta X_{tx\xi}(T)|\cdot \left\|D_\eta X_{t\xi}(T)\right\|_2. \label{prop:3_6}
\end{align}
\normalsize
From Assumption (A1) and Cauchy's inequality, we also know that
\begin{align}
    &\Bigg|\e \int_t^T \sum_{j=1}^n D_\eta Q_{tx\xi}^j(s)^\top  \widetilde{\e}\bigg[ D_y \frac{d\sigma_0^j}{d\nu} (s,\lr(X_{t\xi}(s)))\left(\widetilde{X_{t\xi}}(s)\right) \widetilde{D_\eta X_{t\xi}}(s)\bigg] ds \Bigg| \notag \\
    \le\ & l_{\sigma}^m \int_t^T \left\|\widetilde{D_\eta X_{t\xi}}(s)\right\|_2 \sum_{j=1}^n \left\|Q^j(s) \right\|_2 ds \notag \\
    \le\ & \sqrt{n}l_{\sigma}^m \left(\int_t^T \left\|D_\eta X_{t\xi}(s)\right\|_2^2 ds \right)^{\frac{1}{2}}\left(\int_t^T \| Q(s)\|_2^2 ds\right)^{\frac{1}{2}}. \label{sigma_m_9}
\end{align}
Substituting \eqref{prop:3_3}-\eqref{sigma_m_9} back into \eqref{prop:3_1}, from Cauchy's inequality, we have
\begin{align*}
    &\e\left[2\lambda_g |D_\eta X_{tx\xi}(T)|^2 + \int_t^T\left(2\lambda_v |D_\eta v_{tx\xi}(s)|^2+2\lambda_x |D_\eta X_{tx\xi}(s)|^2\right) ds \right]\notag \\
    \le\ & \e\left[ L_g \  |D_\eta X_{tx\xi}(T)|\cdot \left\|D_\eta X_{t\xi}(T)\right\|_2 \right]  + \sqrt{n}\:l_{\sigma}^m \left(\int_t^T \left\|D_\eta X_{t\xi}(s)\right\|_2^2 ds \right)^{\frac{1}{2}}\cdot \left(\int_t^T \| Q(s)\|_2^2 ds\right)^{\frac{1}{2}} \notag \\
    &+\e \int_t^T \Bigg\{ \frac{L^2l_b^m}{\lambda_b} \left(1+\frac{L^2}{\lambda_b}\right) \left(|D_\eta X_{tx\xi}(s)|+ \left\|D_\eta X_{t\xi}(s)\right\|_2 +|D_\eta v_{tx\xi}(s)|\right)\left\|D_\eta X_{t\xi}(s)\right\|_2 \notag \\
    &\quad\qquad\qquad +\frac{L^2}{\lambda_b} \Big[L_b^0 |D_\eta X_{tx\xi}(s)|\left(|D_\eta X_{tx\xi}(s)|+ \left\|D_\eta X_{t\xi}(s)\right\|_2\right)\\
    &\qquad\qquad\qquad\qquad +L_b^1 |D_\eta v_{tx\xi}(s)| \left(2|D_\eta X_{tx\xi}(s)|+ \left\|D_\eta X_{t\xi}(s)\right\|_2\right) +L_b^2 |D_\eta v_{tx\xi}(s)|^2\Big] \notag \\
    &\quad\qquad\qquad +L_f^0\  |D_\eta X_{tx\xi}(s)|\cdot \left\|D_\eta X_{t\xi}(s)\right\|_2+L_f^1 \  |D_\eta v_{tx\xi}(s)|\cdot \left\|D_\eta X_{t\xi}(s)\right\|_2  \Bigg\}ds\\
    \le\ &  L_g \ \left\|D_\eta X_{tx\xi}(T)\right\|_2 \cdot \left\|D_\eta X_{t\xi}(T)\right\|_2  + \sqrt{n}\:l_{\sigma}^m \left(\int_t^T \left\|D_\eta X_{t\xi}(s)\right\|_2^2 ds \right)^{\frac{1}{2}}\left(\int_t^T \| Q(s)\|_2^2 ds\right)^{\frac{1}{2}} \notag \\
    &+ \int_t^T \Bigg\{ \frac{L^2L_b^2}{\lambda_b}\|D_\eta v_{tx\xi}(s)\|_2^2 + \frac{2L^2L_b^1}{\lambda_b} \|D_\eta v_{tx\xi}(s)\|_2 \cdot \left\|D_\eta X_{tx\xi}(s)\right\|_2 +\left[\frac{L^2L_b^0}{\lambda_b} \right]\left\|D_\eta X_{tx\xi}(s)\right\|_2^2\\
    &\qquad\qquad + \left[L_f^1+ \frac{L^2L_b^1}{\lambda_b}+\frac{L^2l_b^m}{\lambda_b} \left(1+\frac{L^2}{\lambda_b}\right) \right] \|D_\eta v_{tx\xi}(s)\|_2 \cdot \left\|D_\eta X_{t\xi}(s)\right\|_2 \\
    &\qquad\qquad +\left[L_f^0+ \frac{L^2L_b^0}{\lambda_b} +\frac{L^2l_b^m}{\lambda_b} \left(1+\frac{L^2}{\lambda_b}\right) \right]\left\|D_\eta X_{tx\xi}(s)\right\|_2 \cdot \left\|D_\eta X_{t\xi}(s)\right\|_2\\
    &\qquad\qquad +\left[\frac{L^2l_b^m}{\lambda_b} \left(1+\frac{L^2}{\lambda_b}\right) \right] \left\|D_\eta X_{t\xi}(s)\right\|_2^2 \Bigg\}ds.
\end{align*}
Therefore, from Condition \eqref{lem:2_0}, we know that for any $\epsilon>0$,
\begin{align}
    &\left[2\lambda_v-\frac{L^2 L_b^2}{\lambda_b}-\frac{L^2L_b^1}{\sqrt{2\lambda_x\lambda_b^2-2L^2L_b^0\lambda_b }}\right]\cdot \int_t^T \|D_\eta v_{tx\xi}(s)\|_2^2 ds \notag \\
    \le\ &  C(L,\lambda_b)\bigg[\left\|D_\eta X_{tx\xi}(T)\right\|_2 \cdot \left\|D_\eta X_{t\xi}(T)\right\|_2+\int_t^T \left( \|D_\eta v_{tx\xi}(s)\|_2+ \|D_\eta X_{tx\xi}(s)\|_2\right)\cdot \left\|D_\eta X_{t\xi}(s)\right\|_2 ds \notag \\
    &\quad\qquad\qquad + \int_t^T \left\|D_\eta X_{t\xi}(s)\right\|_2^2 ds + \left(\int_t^T \left\|D_\eta X_{t\xi}(s)\right\|_2^2 ds \right)^{\frac{1}{2}}\left(\int_t^T \| Q(s)\|_2^2 ds\right)^{\frac{1}{2}} \bigg] \notag \\
    \le\ & \epsilon\bigg[\sup_{t\le s\le T}\|D_\eta X_{tx\xi}(s)\|_2^2+\int_t^T \left(\|D_\eta v_{tx\xi}(s)\|_2^2 +\|D_\eta Q_{tx\xi}(s)\|_2^2 \right) ds\bigg] \notag \\
    &+C(L,T,\lambda_b)\left(1+\frac{1}{\epsilon}\right) \sup_{t\le s\le T}\|D_\eta X_{t\xi}(s)\|_2^2. \label{prop:3_7}
\end{align}
By using Gr\"onwall's inequality, we have
\begin{equation}\label{prop:3_8}
	\e\bigg[\sup_{t\le s\le T}|D_\eta X_{tx\xi}(s)|^2\bigg]\le C(L,T)\int_t^T \left(\|D_\eta v_{tx\xi}(s)\|_2^2+\|D_\eta X_{t\xi}(s)\|_2^2 \right) ds.
\end{equation}
From Assumption \eqref{generic:condition:b} and Estimates \eqref{lem:5_2} and \eqref{prop:3_8}, the usual BSDE estimates in \cite{YH2,SP} together with BDG inequality altogether give
\begin{equation}\label{prop:3_9}
	\e\bigg[\sup_{t\le s\le T}|D_\eta P_{tx\xi}(s)|^2+\int_t^T |D_\eta Q_{tx\xi}(s)|^2 ds\bigg]\le C(L,T,\lambda_b)\bigg[\int_t^T \|D_\eta v_{tx\xi}(s)\|_2^2ds+ \sup_{t\le s\le T}\|D_\eta X_{t\xi}(s)\|_2^2\bigg].
\end{equation}    
Substituting \eqref{prop:3_8} and \eqref{prop:3_9} into \eqref{prop:3_7}, we have
\begin{align*}
    &\left[2\lambda_v-\frac{L^2 L_b^2}{\lambda_b}-\frac{L^2L_b^1}{\sqrt{2\lambda_x\lambda_b^2-2L^2L_b^0\lambda_b }}\right]\cdot \int_t^T \|D_\eta v_{tx\xi}(s)\|_2^2 ds  \\
    \le\ & \epsilon\ C(L,T,\lambda_b)\bigg[\int_t^T \|D_\eta v_{tx\xi}(s)\|_2^2ds\bigg] +C(L,T,\lambda_b)\left(1+\frac{1}{\epsilon}\right) \sup_{t\le s\le T}\|D_\eta X_{t\xi}(s)\|_2^2.
\end{align*}
By choosing $\epsilon$ small enough such that $\epsilon<\frac{1}{C(L,T,\lambda_b)}\left[2\lambda_v-\frac{L^2 L_b^2}{\lambda_b}-\frac{L^2L_b^1}{\sqrt{2\lambda_x\lambda_b^2-2L^2L_b^0\lambda_b }}\right]$, we deduce that
\begin{align*}
    \int_t^T \|D_\eta v_{tx\xi}(s)\|_2^2 ds \le   C \sup_{t\le s\le T}\|D_\eta X_{t\xi}(s)\|_2^2. 
\end{align*}
Substituting the last estimate back into \eqref{prop:3_8} and \eqref{prop:3_9}, we have
\begin{align}\label{prop:3_10}
    &\e\bigg[\sup_{t\le s\le T}|(D_\eta X_{tx\xi}(s),D_\eta P_{tx\xi}(s))|^2+\int_t^T |D_\eta Q_{tx\xi}(s)|^2 ds\bigg]\le C \sup_{t\le s\le T}\|D_\eta X_{t\xi}(s)\|_2^2.
\end{align}
From Condition \eqref{FB:mu_condition} and the convexity of $f$ in $v$, we know that
\begin{align*}
    &2\lambda_v |D_\eta v_{tx\xi}(s)|^2 \\
    \le\ &D_v^2 f (s,\theta_{tx\mu}(s)) (D_\eta v_{tx\xi}(s))^{\otimes 2} \\
    =\ & -D_\eta v_{tx\xi}(s)^\top D_x D_v f  (s,\theta_{tx\mu}(s))^\top D_\eta X_{tx\xi}(s) \\
    &- D_\eta v_{tx\xi}(s)^\top \widetilde{\e}\left[\left(D_y \frac{d}{d\nu}D_v f  (s,\theta_{tx\mu}(s))\left(\widetilde{X_{t\xi}}(s)\right)\right)^\top \widetilde{D_\eta X_{t\xi}}(s)\right] \notag \\
    &- D_\eta v_{tx\xi}(s)^\top \bigg\{D_x D_v b(s,\theta_{tx\mu}(s)) D_\eta X_{tx\xi}(s) + D_v^2 b(s,\theta_{tx\mu}(s)) D_\eta v_{tx\xi}(s) \\
    &\quad\qquad\qquad\qquad +\widetilde{\e}\left[ D_y\frac{d}{d\nu}D_v b(s,\theta_{tx\mu}(s))\left(\widetilde{X_{t\xi}}(s) \right) \widetilde{D_\eta X_{t\xi}}(s)\right] \bigg\}^\top P_{tx\mu}(s) \\
    &- D_\eta v_{tx\xi}(s)^\top D_v b(s,\theta_{tx\mu}(s))^\top D_\eta P_{tx\xi}(s),
\end{align*}
and then, from Assumptions (A1) and (A2), we can deduce that
\begin{align*}
    2\lambda_v |D_\eta v_{tx\xi}(s)|\le \ & L |D_\eta X_{tx\xi}(s)|+L \left\|D_\eta X_{t\xi}(s)\right\|_2+ L|D_\eta P_{tx\xi}(s)|\\
    &+\frac{L^2}{\lambda_b} \left(L |D_\eta X_{tx\xi}(s)|+L \left\|D_\eta X_{t\xi}(s)\right\|_2+L_b^2|D_\eta v_{tx\xi}(s)|\right). 
\end{align*}
Again from Condition \eqref{lem:2_0}, we also have
\begin{align}\label{prop:3_11}
    |D_\eta v_{tx\xi}(s)|\le \ & \left(2\lambda_v-\frac{L^2 L_b^2}{\lambda_b}\right)^{-1}\left[\left(L+\frac{L^3}{\lambda_b}\right) \left(|D_\eta X_{tx\xi}(s)|+\left\|D_\eta X_{t\xi}(s)\right\|_2\right)+ L|D_\eta P_{tx\xi}(s)|\right]. 
\end{align}
From \eqref{prop:3_10}, \eqref{prop:3_11} and Theorem~\ref{lem:4}, we conclude with \eqref{prop3_1}. 

\subsection{Proof of Theorem~\ref{prop:5}}\label{pf:prop:5}

The proof of the well-posedness of two systems of FBSDEs \eqref{FB:xi_y}-\eqref{FB:xi_y_v} and \eqref{FB:mu_y}-\eqref{FB:mu_y_condition}, and the derivation of Estimate \eqref{prop5_01} are similar to respective arguments leading to Lemmas~\ref{lem:3} and \ref{prop:3}, and we omit here. We only prove \eqref{prop5_03} and \eqref{prop5_04} here.  Note that $\xi$ and $\eta$ and  their respective independent copies $\widehat{\xi}$ and $\widehat{\eta}$ are all required to be  independent of the Brownian motion $B$. Substituting $y=\widehat{\xi}$ in the SDE of \eqref{FB:xi_y}, multiplying it by $\widehat{\eta}$, and taking expectation with respect to $\widehat{\xi}$ and $\widehat{\eta}$ (denoted by $\widehat{\e}$ instead of $\widehat{\e}_{\widehat{\xi},\widehat{\eta}}$,  to reduce the heavy notations), we have
\begin{align*}
    \widehat{\e}\left[\bd X_{t\xi}\left(s,\widehat{\xi}\right)\widehat{\eta}\right]=\ & \int_t^s \bigg\{D_x b(r, \theta_{t\xi}(r))\widehat{\e}\left[\bd X_{t\xi}\left(r,\widehat{\xi}\right)\widehat{\eta}\right]+D_v b(r, \theta_{t\xi}(r))\widehat{\e}\left[\bd v_{t\xi}\left(r,\widehat{\xi}\right)\widehat{\eta}\right] \notag \\
    &\qquad +\widetilde{\e}\widehat{\e}\bigg[ D_y \frac{db}{d\nu} (r,\theta_{t\xi}(r))\left(\widetilde{X_{ty\mu}}(r)\Big|_{y=\widehat{\xi}}\right) \widetilde{D_y X_{ty\mu}}(r)\Big|_{y=\widehat{\xi}}\  \widehat{\eta}\bigg] \notag \\
    &\qquad +\widetilde{\e}\bigg[ D_y \frac{db}{d\nu} (r,\theta_{t\xi}(r))\left(\widetilde{X_{t\xi}}(r)\right) \widehat{\e}\left[\widetilde{\bd X_{t\xi}}\left(r,\widehat{\xi}\right)\widehat{\eta}\right]\bigg]\bigg\}dr \\
    &+\int_t^s \bigg\{ \sigma_1(r)\widehat{\e}\left[\bd X_{t\xi}\left(r,\widehat{\xi}\right)\widehat{\eta} \right]\\
    &\quad\qquad +\widetilde{\e}\widehat{\e}\bigg[ D_y \frac{d\sigma_0}{d\nu} (r,\lr(X_{t\xi}(r)))\left(\widetilde{X_{ty\mu}}(r)\Big|_{y=\widehat{\xi}}\right) \widetilde{D_y X_{ty\mu}}(r)\Big|_{y=\widehat{\xi}}\  \widehat{\eta}\bigg] \\
    &\quad\qquad +\widetilde{\e}\bigg[ D_y \frac{d\sigma_0}{d\nu} (r,\lr(X_{t\xi}(r)))\left(\widetilde{X_{t\xi}}(r)\right) \widehat{\e}\left[\widetilde{\bd X_{t\xi}}\left(r,\widehat{\xi}\right)\widehat{\eta}\right]\bigg] \bigg\}dB(r),
\end{align*}
where the stochastic Fubini theorem \cite{MR2227239} is used to interchange the order of the stochastic integral (with respect to the Brownian motion $B$) and the expectation with respect to $(\widetilde{\xi},\widetilde{\eta})$. Since $(\widehat{\xi},\widehat{\eta})$ are independent of $(\xi,\eta)$ and shares the law of $(\widetilde{\xi},\widetilde{\eta})$, we see that, for instance,
\begin{equation}\label{prop5_5}
	\begin{split}
		&\widetilde{\e}\widehat{\e}\bigg[ D_y \frac{db}{d\nu} (r,\theta_{t\xi}(r))\left(\widetilde{X_{ty\mu}}(r)\Big|_{y=\widehat{\xi}}\right) \widetilde{D_y X_{ty\mu}}(r)\Big|_{y=\widehat{\xi}} \ \widehat{\eta}\bigg] \\
		=\ &\widetilde{\e} \bigg[D_y \frac{db}{d\nu} (r,\theta_{t\xi}(r))\left(\widetilde{X_{t\xi}}(r)\right) \widetilde{D_y X_{ty\mu}}(r)\Big|_{y=\widetilde{\xi}}\  \widetilde{\eta}\bigg].
	\end{split}
\end{equation}
Therefore, we see that
\small
\begin{align}
    &\widehat{\e}\left[\bd X_{t\xi}\left(s,\widehat{\xi}\right)\widehat{\eta}\right]\notag \\
    =\ & \int_t^s \bigg\{D_x b(r, \theta_{t\xi}(r))\widehat{\e}\left[\bd X_{t\xi}\left(r,\widehat{\xi}\right)\widehat{\eta}\right]+D_v b(r, \theta_{t\xi}(r))\widehat{\e}\left[\bd v_{t\xi}\left(r,\widehat{\xi}\right)\widehat{\eta}\right] \notag \\
    &\qquad +\widetilde{\e} \bigg[D_y \frac{db}{d\nu} (r,\theta_{t\xi}(r))\left(\widetilde{X_{t\xi}}(r)\right)\left( \widetilde{D_y X_{ty\mu}}(r)\Big|_{y=\widetilde{\xi}} \ \widetilde{\eta} +\widehat{\e}\left[\widetilde{\bd X_{t\xi}}\left(r,\widehat{\xi}\right)\widehat{\eta}\right]\right) \bigg]\bigg\}dr \notag \\
    &+\int_t^s \bigg\{ \sigma_1(r)\widehat{\e}\left[\bd X_{t\xi}\left(r,\widehat{\xi}\right)\widehat{\eta} \right] \notag \\
    &\quad\qquad +\widetilde{\e} \bigg[D_y \frac{d\sigma_0}{d\nu} (r,\lr(X_{t\xi}(r)))\left(\widetilde{X_{t\xi}}(r)\right)\left( \widetilde{D_y X_{ty\mu}}(r)\Big|_{y=\widetilde{\xi}} \ \widetilde{\eta} +\widehat{\e}\left[\widetilde{\bd X_{t\xi}}\left(r,\widehat{\xi}\right)\widehat{\eta}\right]\right) \bigg] \bigg\} dB(r). \label{prop5_1}
\end{align}
\normalsize
Dealing with all the other terms in a similar way, from the BSDE in \eqref{FB:xi_y}, we have	
\small
\begin{align}
		&\widehat{\e}\left[\bd P_{t\xi}\left(s,\widehat{\xi}\right)\widehat{\eta}\right]\notag \\
        =\ & -\int_s^T\widehat{\e}\left[\bd Q_{t\xi}\left(r,\widehat{\xi}\right)\widehat{\eta}\right]dB(r)+ D_x^2 g (X_{t\xi}(T),\lr(X_{t\xi}(T)))^\top  \widehat{\e}\left[\bd X_{t\xi}\left(T,\widehat{\xi}\right) \widehat{\eta}\right] \notag \\
        &+\widetilde{\e}\left[\left(D_y \frac{d}{d\nu}D_x g (X_{t\xi}(T),\lr(X_{t\xi}(T)))\left(\widetilde{X_{t\xi}}(T)\right) \right)^\top \left(\widetilde{D_y X_{ty\mu}}(T)\Big|_{y=\widetilde{\xi}} \ \widetilde{\eta} +\widehat{\e}\left[\widetilde{\bd X_{t\xi}}\left(T,\widehat{\xi}\right)\widehat{\eta}\right]\right)\right] \notag \\
		&+\int_s^T\Bigg\{ D_xb(r,\theta_{t\xi}(r))^\top \widehat{\e}\left[\bd P_{t\xi}\left(r,\widehat{\xi}\right)\widehat{\eta}\right] +\sum_{j=1}^n\left(\sigma_1^j (r)\right)^\top \widehat{\e}\left[\bd Q_{t\xi}^j\left(r,\widehat{\xi}\right)\widehat{\eta}\right] \notag \\
        &\quad+ \left(P_{t\xi}(r) \right)^\top \bigg\{D_x^2 b(r,\theta_{t\xi}(r)) \widehat{\e}\left[\bd X_{t\xi}\left(r,\widehat{\xi}\right)\widehat{\eta}\right]+ D_vD_x b(r,\theta_{t\xi}(r)) \widehat{\e}\left[\bd v_{t\xi}\left(r,\widehat{\xi}\right)\widehat{\eta} \right] \notag \\
        &\quad\qquad+\widetilde{\e}\left[ D_y\frac{d}{d\nu}D_x b(r,\theta_{t\xi}(r))\left(\widetilde{X_{t\xi}}(r)\right)\left( \widetilde{D_y X_{ty\mu}}(r)\Big|_{y=\widetilde{\xi}}\ \widetilde{\eta}+ \widehat{\e}\left[ \widetilde{\bd X_{t\xi}}\left(r,\widehat{\xi}\right)\widehat{\eta}\right]\right)\right]\bigg\} \notag\\
        &+D_x^2 f  (r,\theta_{t\xi}(r))^\top \widehat{\e}\left[\bd X_{t\xi}\left(r,\widehat{\xi}\right)\widehat{\eta}\right] + D_vD_x f (r,\theta_{t\xi}(r))^\top  \widehat{\e}\left[\bd v_{t\xi}\left(r,\widehat{\xi}\right)\widehat{\eta} \right] \notag\\
        &+\widetilde{\e}\left[\left(D_y \frac{d}{d\nu}D_x f  (r,\theta_{t\xi}(r))\left(\widetilde{X_{t\xi}}(r)\right)\right)^\top   \left(\widetilde{D_y X_{ty\mu}}(r)\Big|_{y=\widetilde{\xi}}\ \widetilde{\eta}+   \widehat{\e}\left[ \widetilde{\bd X_{t\xi}}\left(r,\widehat{\xi}\right)\widehat{\eta}\right]\right) \right]\Bigg\}dr.\label{prop5_3}
\end{align}
\normalsize
In a similar way, substituting $y=\widehat{\xi}$ in Condition \eqref{FB:xi_y_v}, multiplying it by $\widehat{\eta}$, and taking expectation with respect to $\widehat{\xi}$ and $\widehat{\eta}$, we  deduce that 
\small
\begin{align}
    0=\ & \widetilde{\e}\left[\left(D_y \frac{d}{d\nu}D_v f  (s,\theta_{t\xi}(s))\left(\widetilde{X_{t\xi}}(s)\right)\right)^\top \left(\widetilde{D_y X_{ty\mu}}(s)\Big|_{y=\widetilde{\xi}} \widetilde{\eta} + \widehat{\e}\left[\widetilde{\bd X_{t\xi}}\left(s,\widehat{\xi}\right) \widehat{\eta}\right]\right) \right] \notag \\
    &+D_x D_v f  (s,\theta_{t\xi}(s))^\top \widehat{\e}\left[\bd X_{t\xi}\left(s,\widehat{\xi}\right)\widehat{\eta} \right] +D_v^2 f (s,\theta_{t\xi}(s))^\top  \widehat{\e}\left[\bd v_{t\xi}\left(s,\widehat{\xi}\right) \widehat{\eta}\right] \notag \\
    &+D_v b(s,\theta_{t\xi}(s))^\top \widehat{\e}\left[\bd P_{t\xi}\left(s,\widehat{\xi}\right)\widehat{\eta} \right] \notag\\
    &+ \left(P_{t\xi}(s)\right)^\top \bigg\{D_x D_v b(s,\theta_{t\xi}(s)) \widehat{\e}\left[\bd X_{t\xi}\left(s,\widehat{\xi}\right)\widehat{\eta}\right] + D_v^2 b(s,\theta_{t\xi}(s)) \widehat{\e}\left[ \bd v_{t\xi}\left(s,\widehat{\xi}\right) \widehat{\eta} \right] \notag \\
    &\qquad\qquad\qquad +\widetilde{\e}\left[ D_y\frac{d}{d\nu}D_v b(s,\theta_{t\xi}(s))\left(\widetilde{X_{t\xi}}(s) \right) \left( \widetilde{D_y X_{ty\mu}}(s)\Big|_{y=\widetilde{\xi}} \widetilde{\eta} + \widehat{\e}\left[ \widetilde{\bd X_{t\xi}}\left(s,\widehat{\xi}\right)\widehat{\eta}\right] \right]\right) \bigg\}. \label{prop5_2} 
\end{align}
\normalsize
From \eqref{prop5_1}-\eqref{prop5_2}, we know that the component processes
\begin{align*}
	\left(\widehat{\e}\left[\bd X_{t\xi}\left(s,\widehat{\xi}\right) \widehat{\eta}\right],\ \widehat{\e}\left[\bd v_{t\xi}\left(s,\widehat{\xi}\right) \widehat{\eta}\right],\ \widehat{\e}\left[\bd P_{t\xi}\left(s,\widehat{\xi}\right) \widehat{\eta}\right],\ \widehat{\e}\left[\bd Q_{t\xi}\left(s,\widehat{\xi}\right) \widehat{\eta}\right] \right)
\end{align*}
satisfy FBSDEs \eqref{FB:dr'}-\eqref{FB:dr'_condition}. From the uniqueness result of FBSDEs \eqref{FB:dr'}-\eqref{FB:dr'_condition}, we obtain \eqref{prop5_03}. We now establish \eqref{prop5_04}. From the SDE of \eqref{FB:mu_y} and the conditions \eqref{prop5_03} and \eqref{prop5_5}, we have for $s\in[t,T]$,
\begin{align*}
    &\widehat{\e}\left[D X_{tx\mu}\left(s,\widehat{\xi}\right)\widehat{\eta}\right]\\
    =\ & \int_t^s \bigg\{ D_x b(r, \theta_{tx\mu}(r))D X_{tx\mu}\left(r,\widehat{\xi}\right)\widehat{\e}\left[\widehat{\eta}\right]+D_v b(r, \theta_{tx\mu}(r))\widehat{\e}\left[D v_{tx\mu}\left(r,\widehat{\xi}\right)\widehat{\eta} \right]\notag \\
    &\qquad +\widetilde{\e}\bigg[ D_y \frac{db}{d\nu} (r,\theta_{tx\mu}(r))\left(\widetilde{X_{t\xi}}(r)\right) \left(\widetilde{D_y X_{ty\mu}}(r)\Big|_{y=\widetilde{\xi}} \widetilde{\eta}+\widetilde{\bd_\eta X_{t\xi}}(r)\right)\bigg] \bigg\}dr \\
    &+\int_t^s \bigg\{ \sigma_1(r)\widehat{\e}\left[D X_{tx\mu}\left(r,\widehat{\xi}\right)\widehat{\eta} \right]\\
    &\quad\qquad + \widetilde{\e}\bigg[ D_y \frac{d\sigma_0}{d\nu} (r,\lr(X_{t\xi}(r)))\left(\widetilde{X_{t\xi}}(r)\right) \left(\widetilde{D_y X_{ty\mu}}(r)\Big|_{y=\widetilde{\xi}} \widetilde{\eta}+\widetilde{\bd_\eta X_{t\xi}}(r)\right)\bigg] \bigg\}dB(r).
\end{align*}
In a similar manner, we deduce that
\small
\begin{align*}
		&\widehat{\e}\left[D P_{tx\mu}\left(s,\widehat{\xi}\right)\widehat{\eta}\right]\\
        =\ & -\int_s^T\widehat{\e}\left[D Q_{tx\mu}\left(r,\widehat{\xi}\right)\widehat{\eta}\right] dB(r)+D_x^2 g (X_{tx\mu}(T),\lr(X_{t\xi}(T)))^\top \widehat{\e}\left[ D X_{tx\mu}\left(T,\widehat{\xi}\right)\widehat{\eta}\right] \notag\\
        &+\widetilde{\e}\left[\left(D_y \frac{d}{d\nu}D_x g (X_{tx\mu}(T),\lr(X_{t\xi}(T)))\left(\widetilde{X_{t\xi}}(T)\right) \right)^\top  \left(\widetilde{D_y X_{ty\mu}}(T)\Big|_{y=\widetilde{\xi}}\ \widetilde{\eta}+ \widetilde{\bd_\eta X_{t\xi}}(T) \right) \right] \notag \\
		&+\int_s^T\Bigg\{ D_xb(r,\theta_{tx\mu}(r))^\top \widehat{\e}\left[D P_{tx\mu}\left(r,\widehat{\xi}\right)\widehat{\eta}\right]+\sum_{j=1}^n\left(\sigma_1^j (r)\right)^\top \widehat{\e}\left[D Q_{tx\mu}^j\left(r,\widehat{\xi}\right)\widehat{\eta}\right] \notag \\
        &\quad\qquad + \left(P_{tx\mu}(r)\right)^\top \bigg\{\widetilde{\e}\left[ D_y\frac{d}{d\nu}D_x b(r,\theta_{tx\mu}(r))\left(\widetilde{X_{t\xi}}(r)\right) \left(\widetilde{D_y X_{ty\mu}}(r)\Big|_{y=\widetilde{\xi}}\ \widetilde{\eta}+\widetilde{\bd_\eta X_{t\xi}}(r)\right) \right] \notag \\
        &\quad\qquad\qquad\qquad\qquad +D_x^2 b(r,\theta_{tx\mu}(r)) \widehat{\e}\left[D X_{tx\mu}\left(r,\widehat{\xi}\right)\widehat{\eta}\right]+ D_vD_x b(r,\theta_{tx\mu}(r)) \widehat{\e}\left[D v_{tx\mu}\left(r,\widehat{\xi}\right)\widehat{\eta}\right] \bigg\} \notag\\
        &\quad\qquad +\widetilde{\e}\left[\left(D_y \frac{d}{d\nu}D_x f  (r,\theta_{tx\mu}(r))\left(\widetilde{X_{t\xi}}(r)\right)\right)^\top \left(\widetilde{D_y X_{ty\mu}}(r)\Big|_{y=\widetilde{\xi}}\ \widetilde{\eta}+ \widetilde{\bd_\eta X_{t\xi}}(r)\right)\right] \notag \\
        &\quad\qquad +D_x^2 f  (r,\theta_{tx\mu}(r))^\top \widehat{\e}\left[D X_{tx\mu}\left(r,\widehat{\xi}\right)\widehat{\eta}\right] + D_vD_x f (r,\theta_{tx\mu}(r))^\top  \widehat{\e}\left[D v_{tx\mu}\left(r,\widehat{\xi}\right)\widehat{\eta}\right] \Bigg\}dr, 
\end{align*}
\normalsize
and that
\small
\begin{align*}
    0=\ & \widetilde{\e}\left[\left(D_y \frac{d}{d\nu}D_v f  (s,\theta_{tx\mu}(s))\left(\widetilde{X_{t\xi}}(s)\right)\right)^\top \left(\widetilde{D_y X_{ty\mu}}(s) \Big|_{y=\widetilde{\xi}}\ \widetilde{\eta} +\widetilde{\bd_\eta X_{t\xi}}(s)\right) \right]\notag\\
    &+D_v^2 f (s,\theta_{tx\mu}(s))^\top  \widehat{\e}\left[D v_{tx\mu}\left(s,\widehat{\xi}\right)\widehat{\eta}\right] \notag \\
    &+D_x D_v f  (s,\theta_{tx\mu}(s))^\top \widehat{\e}\left[ D X_{tx\mu}\left(s,\widehat{\xi}\right)\widehat{\eta}\right] +D_v b(s,\theta_{tx\mu}(s))^\top \widehat{\e}\left[D P_{tx\mu}\left(s,\widehat{\xi}\right)\widehat{\eta}\right] \notag\\
    &+ \left(P_{tx\mu}(s)\right)^\top \bigg\{\widetilde{\e}\left[ D_y\frac{d}{d\nu}D_v b(s,\theta_{tx\mu}(s))\left(\widetilde{X_{t\xi}}(s) \right) \left(\widetilde{D_y X_{ty\mu}}(s)\Big|_{y=\widetilde{\xi}}\ \widetilde{\eta} +\widetilde{\bd_\eta X_{t\xi}}(s)\right) \right] \notag \\
    &\quad\qquad\qquad +D_x D_v b(s,\theta_{tx\mu}(s)) \widehat{\e}\left[D X_{tx\mu}\left(s,\widehat{\xi}\right)\widehat{\eta}\right]+ D_v^2 b(s,\theta_{tx\mu}(s)) \widehat{\e}\left[D v_{tx\mu}\left(s,\widehat{\xi}\right)\widehat{\eta}\right]\bigg\}.
\end{align*}
\normalsize
Therefore, the component processes
\begin{align*}
	\left(\widehat{\e}\left[D X_{tx\mu}\left(s,\widehat{\xi}\right) \widehat{\eta}\right],\ \widehat{\e}\left[D v_{tx\mu}\left(s,\widehat{\xi}\right) \widehat{\eta}\right],\ \widehat{\e}\left[D P_{tx\mu}\left(s,\widehat{\xi}\right) \widehat{\eta}\right],\ \widehat{\e}\left[D Q_{tx\mu}\left(s,\widehat{\xi}\right) \widehat{\eta}\right] \right)
\end{align*}
also satisfy FBSDEs \eqref{FB:mu'}-\eqref{FB:mu'_condition}. From the uniqueness result of FBSDEs \eqref{FB:mu'}-\eqref{FB:mu'_condition}, we conclude with \eqref{prop5_04}.

\section{Proof of Theorem~\ref{prop:9}}\label{pf:prop:9}

We here give the respective systems of FBSDEs for the derivatives matrix $D_y \bd \Theta_{t\xi}(s,y)$ and the derivative $D_y \dr \Theta_{tx\mu}(s,y)$. From FBSDEs \eqref{FB:xi_y}-\eqref{FB:xi_y_v} and Theorems~\ref{prop:2} and \ref{prop:8}, the component G\^ateaux derivatives of $\bd \Theta_{t\xi}(s,y)$ in $y$ can be characterized as the solution of the following FBSDEs: for $(s,y)\in[t,T]\times\brn$,
\small
\begin{align}
		D_y\bd X_{t\xi}(s,y)=\ & \int_t^s \bigg\{D_x b(r, \theta_{t\xi}(r))D_y\bd X_{t\xi}(r,y)+D_v b(r, \theta_{t\xi}(r))D_y\bd v_{t\xi}(r,y) \notag \\
        &\qquad +\widetilde{\e}\bigg[ D_y \frac{db}{d\nu} (r,\theta_{t\xi}(r))\left(\widetilde{X_{t\xi}}(r)\right) \widetilde{D_y\bd X_{t\xi}}(r,y)\bigg] \bigg\}dr \notag \\
        & +\int_t^s \widetilde{\e}\bigg[ D_y \frac{db}{d\nu} (r,\theta_{t\xi}(r))\left(\widetilde{X_{ty\mu}}(r)\right) \widetilde{D_y^2 X_{ty\mu}}(r) \bigg] dr  \notag \\
        & +\int_t^s \widetilde{\e}\bigg[ D_y^2 \frac{db}{d\nu} (r,\theta_{t\xi}(r))\left(\widetilde{X_{ty\mu}}(r)\right) \left(\widetilde{D_y X_{ty\mu}}(r)\right)^{\otimes2} \bigg] dr \notag \\
        & +\int_t^s \bigg\{ \sigma_1(r)D_y\bd X_{t\xi}(r,y) + \widetilde{\e}\bigg[ D_y \frac{d\sigma_0}{d\nu} (r,\lr(X_{t\xi}(r)))\left(\widetilde{X_{t\xi}}(r)\right) \widetilde{D_y\bd X_{t\xi}}(r,y)\bigg] \bigg\}dB(r)\notag\\
        & +\int_t^s \widetilde{\e}\bigg[ D_y \frac{d\sigma_0}{d\nu} (r,\lr(X_{t\xi}(r)))\left(\widetilde{X_{ty\mu}}(r)\right) \widetilde{D_y^2 X_{ty\mu}}(r) \bigg] dB(r)  \notag \\
        & +\int_t^s \widetilde{\e}\bigg[ D_y^2 \frac{d\sigma_0}{d\nu} (r,\lr(X_{t\xi}(r)))\left(\widetilde{X_{ty\mu}}(r)\right) \left(\widetilde{D_y X_{ty\mu}}(r)\right)^{\otimes2} \bigg] dB(r);\notag \\
		\text{and}\quad D_y\bd P_{t\xi}(s,y)=\ & -\int_s^T D_y \bd Q_{t\xi}(r,y)dB(r)+ D_x^2 g (X_{t\xi}(T),\lr(X_{t\xi}(T)))^\top D_y\bd X_{t\xi}(T,y) \notag \\
        &+\widetilde{\e}\left[\left(D_y \frac{d}{d\nu}D_x g (X_{t\xi}(T),\lr(X_{t\xi}(T)))\left(\widetilde{X_{t\xi}}(T)\right) \right)^\top \widetilde{D_y \bd X_{t\xi}}(T,y)\right] \notag \\
        &+\widetilde{\e}\left[\left(D_y \frac{d}{d\nu}D_x g (X_{t\xi}(T),\lr(X_{t\xi}(T)))\left(\widetilde{X_{ty\mu}}(T)\right) \right)^\top \widetilde{D_y^2 X_{ty\mu}}(T)\right] \notag \\
        &+\widetilde{\e}\left[\left(D_y^2 \frac{d}{d\nu}D_x g (X_{t\xi}(T),\lr(X_{t\xi}(T)))\left(\widetilde{X_{ty\mu}}(T)\right) \right) \left(\widetilde{D_y X_{ty\mu}}(T)\right)^{\otimes2} \right] \notag \\
		&+\int_s^T\Bigg\{ D_xb(r,\theta_{t\xi}(r))^\top D_y \bd P_{t\xi}(r,y)+\sum_{j=1}^n\left(\sigma_1^j (r)\right)^\top D_y \bd Q_{t\xi}^j(r,y) \notag \\
        &\qquad\qquad + \left(P_{t\xi}(r)\right)^\top \bigg\{D_x^2 b(r,\theta_{t\xi}(r)) D_y\bd X_{t\xi}(r,y)+ D_vD_x b(r,\theta_{t\xi}(r)) D_y\bd v_{t\xi}(r,y) \notag \\
        &\qquad\qquad\qquad\qquad\qquad +\widetilde{\e}\left[ D_y\frac{d}{d\nu}D_x b(r,\theta_{t\xi}(r))\left(\widetilde{X_{t\xi}}(r)\right) \widetilde{D_y \bd X_{t\xi}}(r,y)\right]\bigg\}  \notag\\
        &\qquad\qquad +D_x^2 f  (r,\theta_{t\xi}(r))^\top D_y\bd X_{t\xi}(r,y) + D_vD_x f (r,\theta_{t\xi}(r))^\top  D_y\bd v_{t\xi}(r,y) \notag\\
        &\qquad\qquad+\widetilde{\e}\left[\left(D_y \frac{d}{d\nu}D_x f  (r,\theta_{t\xi}(r))\left(\widetilde{X_{t\xi}}(r)\right)\right)^\top   \widetilde{D_y\bd X_{t\xi}}(r,y)\right]\Bigg\}dr \notag \\
        & +\int_s^T \left(P_{t\xi}(r)\right)^\top \widetilde{\e}\left[ D_y\frac{d}{d\nu}D_x b(r,\theta_{t\xi}(r))\left(\widetilde{X_{ty\mu}}(r)\right) \widetilde{D_y^2 X_{ty\mu}}(r)\right]dr \notag\\
        & +\int_s^T \left(P_{t\xi}(r)\right)^\top \widetilde{\e}\left[ D_y^2 \frac{d}{d\nu}D_x b(r,\theta_{t\xi}(r))\left(\widetilde{X_{ty\mu}}(r)\right) \left(\widetilde{D_y X_{ty\mu}}(r)\right)^{\otimes2} \right] dr \notag\\
        & +\int_s^T \widetilde{\e}\left[\left(D_y \frac{d}{d\nu}D_x f  (r,\theta_{t\xi}(r))\left(\widetilde{X_{ty\mu}}(r)\right)\right)^\top   \widetilde{D_y^2 X_{ty\mu}}(r) \right] dr\notag \\
        &+\int_s^T \widetilde{\e}\left[\left(D_y^2 \frac{d}{d\nu}D_x f  (r,\theta_{t\xi}(r))\left(\widetilde{X_{ty\mu}}(r)\right)\right) \left(\widetilde{D_y X_{ty\mu}}(r)\right)^{\otimes2} \right] dr, \label{FB:xi_yy}
\end{align}
\normalsize
with (as a consequence of taking the derivative with respect to $y$ of Condition \eqref{FB:xi_y_v})
\small
\begin{align}
    0=\ & \widetilde{\e}\left[\left(D_y \frac{d}{d\nu}D_v f  (s,\theta_{t\xi}(s))\left(\widetilde{X_{t\xi}}(s)\right)\right)^\top  \widetilde{D_y\bd X_{t\xi}}(s,y)\right] \notag \\
    &+D_x D_v f  (s,\theta_{t\xi}(s))^\top D_y\bd X_{t\xi}(s,y) +D_v^2 f (s,\theta_{t\xi}(s))^\top  D_y\bd v_{t\xi}(s,y)+D_v b(s,\theta_{t\xi}(s))^\top D_y\bd P_{t\xi}(s,y) \notag\\
    &+ \left(P_{t\xi}(s)\right)^\top \bigg\{D_x D_v b(s,\theta_{t\xi}(s)) D_y\bd X_{t\xi}(s,y)+ D_v^2 b(s,\theta_{t\xi}(s)) D_y\bd v_{t\xi}(s,y) \notag \\
    &\qquad\qquad\qquad +\widetilde{\e}\left[ D_y\frac{d}{d\nu}D_v b(s,\theta_{t\xi}(s))\left(\widetilde{X_{t\xi}}(s) \right) \widetilde{D_y\bd X_{t\xi}}(s,y)\right]\bigg\}\notag\\
    &+\widetilde{\e}\left[\left(D_y \frac{d}{d\nu}D_v f  (s,\theta_{t\xi}(s))\left(\widetilde{X_{ty\mu}}(s)\right)\right)^\top \widetilde{D_y^2 X_{ty\mu}}(s)\right]\notag\\
    &+\widetilde{\e}\left[\left(D_y^2 \frac{d}{d\nu}D_v f  (s,\theta_{t\xi}(s))\left(\widetilde{X_{ty\mu}}(s)\right)\right) \left(\widetilde{D_y X_{ty\mu}}(s)\right)^{\otimes2} \right]\notag\\
    &+ \left(P_{t\xi}(s)\right)^\top \bigg\{\widetilde{\e}\left[ D_y\frac{d}{d\nu}D_v b(s,\theta_{t\xi}(s))\left(\widetilde{X_{ty\mu}}(s) \right) \widetilde{D_y^2 X_{ty\mu}}(s) \right]\notag\\
    &\qquad\qquad\qquad +\widetilde{\e}\left[ D_y^2\frac{d}{d\nu}D_v b(s,\theta_{t\xi}(s))\left(\widetilde{X_{ty\mu}}(s) \right) \left(\widetilde{D_y X_{ty\mu}}(s)\right)^{\otimes2} \right]\bigg\}, \label{FB:xi_yy_condition}
\end{align}
\normalsize
where 
$$\left(D_y X_{ty\mu},D_y^2 X_{ty\mu}\right)=\left(D_x X_{tx\mu},D_x^2 X_{tx\mu} \right) \Big|_{x=y},$$ 
and $\left(\widetilde{X_{ty\mu}}(s),\widetilde{X_{t\xi}}(s),\widetilde{D_y X_{ty\mu}}(s),\widetilde{D_y^2 X_{ty\mu}}(s),\widetilde{D_y\bd X_{t\xi}}(s,y)\right)$ is the respective independent copy of $\left(X_{ty\mu}(s),X_{t\xi}(s),D_y X_{ty\mu}(s),D_y^2 X_{ty\mu}(s),D_y\bd X_{t\xi}(s,y)\right)$. Then, from FBSDEs \eqref{FB:mu_y}-\eqref{FB:mu_y_condition}, the component G\^ateaux derivatives of $D \Theta_{tx\mu}(s,y)$ in $y$ can be characterized as the solution of the following FBSDEs: for $(s,y)\in[t,T]\times\brn$,
\small
\begin{align}
		D_yD X_{tx\mu}(s,y)=\ & \int_t^s \bigg[ D_x b(r, \theta_{tx\mu}(r))D_yD X_{tx\mu}(r,y)+D_v b(r, \theta_{tx\mu}(r))D_yD v_{tx\mu}(r,y)\bigg]dr \notag\\
        & +\int_t^s \widetilde{\e}\bigg[ D_y \frac{db}{d\nu} (r,\theta_{tx\mu}(r))\left(\widetilde{X_{t\xi}}(r)\right) \widetilde{D_y\bd X_{t\xi}}(r,y)\bigg]dr \notag\\
        & +\int_t^s \widetilde{\e}\bigg[ D_y \frac{db}{d\nu} (r,\theta_{tx\mu}(r))\left(\widetilde{X_{ty\mu}}(r)\right) \widetilde{D_y^2 X_{ty\mu}}(r) \bigg]dr \notag \\
        & +\int_t^s \widetilde{\e}\bigg[ D_y^2 \frac{db}{d\nu} (r,\theta_{tx\mu}(r))\left(\widetilde{X_{ty\mu}}(r)\right) \left(\widetilde{D_y X_{ty\mu}}(r)\right)^{\otimes2} \bigg]dr \notag \\
        &+\int_t^s \sigma_1(r)D X_{tx\mu}(r,y) dB(r) \notag \\
        & +\int_t^s \widetilde{\e}\bigg[ D_y \frac{d\sigma_0}{d\nu} (r,\lr(X_{t\xi}(r)))\left(\widetilde{X_{t\xi}}(r)\right) \widetilde{D_y\bd X_{t\xi}}(r,y)\bigg]dB(r) \notag\\
        & +\int_t^s \widetilde{\e}\bigg[ D_y \frac{d\sigma_0}{d\nu} (r,\lr(X_{t\xi}(r)))\left(\widetilde{X_{ty\mu}}(r)\right) \widetilde{D_y^2 X_{ty\mu}}(r) \bigg]dB(r) \notag \\
        & +\int_t^s \widetilde{\e}\bigg[ D_y^2 \frac{d\sigma_0}{d\nu} (r,\lr(X_{t\xi}(r)))\left(\widetilde{X_{ty\mu}}(r)\right) \left(\widetilde{D_y X_{ty\mu}}(r)\right)^{\otimes2} \bigg]dB(r); \notag \\
		\text{and}\quad D_yD P_{tx\mu}(s,y)=\ & -\int_s^T D_yD Q_{tx\mu}(r,y)dB(r)+D_x^2 g (X_{tx\mu}(T),\lr(X_{t\xi}(T)))^\top  D_yD X_{tx\mu}(T,y) \notag\\
        &+\widetilde{\e}\left[\left(D_y \frac{d}{d\nu}D_x g (X_{tx\mu}(T),\lr(X_{t\xi}(T)))\left(\widetilde{X_{t\xi}}(T)\right) \right)^\top \widetilde{D_y\bd X_{t\xi}}(T,y) \right] \notag \\
        &+\widetilde{\e}\left[\left(D_y \frac{d}{d\nu}D_x g (X_{tx\mu}(T),\lr(X_{t\xi}(T)))\left(\widetilde{X_{ty\mu}}(T)\right) \right)^\top  \widetilde{D_y^2 X_{ty\mu}}(T)\right] \notag \\
        &+\widetilde{\e}\left[\left(D_y^2 \frac{d}{d\nu}D_x g (X_{tx\mu}(T),\lr(X_{t\xi}(T)))\left(\widetilde{X_{ty\mu}}(T)\right) \right)  \left(\widetilde{D_y X_{ty\mu}}(T)\right)^{\otimes2} \right] \notag \\
		&+\int_s^T\Bigg\{ D_xb(r,\theta_{tx\mu}(r))^\top D_yD P_{tx\mu}(r,y)+\sum_{j=1}^n\left(\sigma_1^j (r)\right)^\top D_yD Q_{tx\mu}^j(r,y) \notag \\
        &+ \left(P_{tx\mu}(r)\right)^\top \bigg[D_x^2 b(r,\theta_{tx\mu}(r)) D_yD X_{tx\mu}(r,y)+ D_vD_x b(r,\theta_{tx\mu}(r)) D_yD v_{tx\mu}(r,y) \bigg] \notag\\
        &+D_x^2 f  (r,\theta_{tx\mu}(r))^\top D_yD X_{tx\mu}(r,y) + D_vD_x f (r,\theta_{tx\mu}(r))^\top  D_yD v_{tx\mu}(r,y)\Bigg\}dr \notag \\
        & +\int_s^T \left(P_{tx\mu}(r)\right)^\top \widetilde{\e}\left[ D_y\frac{d}{d\nu}D_x b(r,\theta_{tx\mu}(r))\left(\widetilde{X_{t\xi}}(r)\right) \widetilde{D_y\bd X_{t\xi}}(r,y) \right]  dr \notag\\
        &+\int_s^T \left(P_{tx\mu}(r)\right)^\top \widetilde{\e}\left[ D_y\frac{d}{d\nu}D_x b(r,\theta_{tx\mu}(r))\left(\widetilde{X_{ty\mu}}(r)\right) \widetilde{D_y^2 X_{ty\mu}}(r) \right]dr \notag \\
        &+\int_s^T \left(P_{tx\mu}(r)\right)^\top \widetilde{\e}\left[ D_y^2 \frac{d}{d\nu}D_x b(r,\theta_{tx\mu}(r))\left(\widetilde{X_{ty\mu}}(r)\right) \left(\widetilde{D_y X_{ty\mu}}(r) \right)^{\otimes2} \right]dr \notag \\
        &+\int_s^T \widetilde{\e}\left[\left(D_y \frac{d}{d\nu}D_x f  (r,\theta_{tx\mu}(r))\left(\widetilde{X_{t\xi}}(r)\right)\right)^\top \widetilde{D_y\bd X_{t\xi}}(r,y) \right] dr \notag \\
        & +\int_s^T \widetilde{\e}\left[\left(D_y \frac{d}{d\nu}D_x f  (r,\theta_{tx\mu}(r))\left(\widetilde{X_{ty\mu}}(r)\right)\right)^\top \widetilde{D_y^2 X_{ty\mu}}(r)\right]dr \notag \\
        & +\int_s^T \widetilde{\e}\left[\left(D_y^2 \frac{d}{d\nu}D_x f  (r,\theta_{tx\mu}(r))\left(\widetilde{X_{ty\mu}}(r)\right)\right) \left(\widetilde{D_y X_{ty\mu}}(r)\right)^{\otimes2} \right]dr, \label{FB:mu_yy}
\end{align}
\normalsize
with (via taking the derivative with respect to $y$ in Condition \eqref{FB:mu_y_condition})
\small
\begin{align}
    0=\ & D_x D_v f  (s,\theta_{tx\mu}(s))^\top D_yD X_{tx\mu}(s,y) +D_v^2 f (s,\theta_{tx\mu}(s))^\top  D_yD v_{tx\mu}(s,y)+D_v b(s,\theta_{tx\mu}(s))^\top D_yD P_{tx\mu}(s,y) \notag\\
    &+ \left(P_{tx\mu}(s)\right)^\top \bigg[D_x D_v b(s,\theta_{tx\mu}(s)) D_yD X_{tx\mu}(s,y)+ D_v^2 b(s,\theta_{tx\mu}(s)) D_yD v_{tx\mu}(s,y)\bigg] \notag \\
    &+\widetilde{\e}\left[\left(D_y \frac{d}{d\nu}D_v f  (s,\theta_{tx\mu}(s))\left(\widetilde{X_{t\xi}}(s)\right)\right)^\top \widetilde{D_y\bd X_{t\xi}}(s,y) \right]\notag\\
    &+\widetilde{\e}\left[\left(D_y \frac{d}{d\nu}D_v f  (s,\theta_{tx\mu}(s))\left(\widetilde{X_{ty\mu}}(s)\right)\right)^\top \widetilde{D_y^2 X_{ty\mu}}(s) \right] \notag \\
    &+\widetilde{\e}\left[\left(D_y^2 \frac{d}{d\nu}D_v f  (s,\theta_{tx\mu}(s))\left(\widetilde{X_{ty\mu}}(s)\right)\right) \left(\widetilde{D_y X_{ty\mu}}(s) \right)^{\otimes2} \right] \notag \\
    &+\left(P_{tx\mu}(s)\right)^\top \widetilde{\e}\left[ D_y\frac{d}{d\nu}D_v b(s,\theta_{tx\mu}(s))\left(\widetilde{X_{t\xi}}(s) \right) \widetilde{D_y\bd X_{t\xi}}(s,y) \right] \notag \\
    &+ \left(P_{tx\mu}(s)\right)^\top \widetilde{\e}\left[ D_y\frac{d}{d\nu}D_v b(s,\theta_{tx\mu}(s))\left(\widetilde{X_{ty\mu}}(s) \right) \widetilde{D_y^2 X_{ty\mu}}(s)\right] \notag \\
    &+ \left(P_{tx\mu}(s)\right)^\top \widetilde{\e}\left[ D_y^2 \frac{d}{d\nu}D_v b(s,\theta_{tx\mu}(s))\left(\widetilde{X_{ty\mu}}(s) \right) \left(\widetilde{D_y X_{ty\mu}}(s)\right)^{\otimes2} \right]. \label{FB:mu_yy_condition}
\end{align}
\normalsize
From Theorem \ref{prop:8} and the $L^4$-norm boundedness of $D_x \Theta_{tx\mu}$ in \eqref{thm2_1} of Theorem~\ref{prop:2}, the proofs of the well-posedness of FBSDEs \eqref{FB:xi_yy}-\eqref{FB:xi_yy_condition} and \eqref{FB:mu_yy}-\eqref{FB:mu_yy_condition}, and of the corresponding convergence results for finite differences are similar to those for the statements in Section~\ref{sec:distribution} (such as Theorem~\ref{lem:4}), and we simply omit. From the continuity given in  Assumption (A2') and Theorem~\ref{prop:5}, similar to the proof of Theorem~\ref{lem:4}, we see that the derivatives are continuous. 

\section{Proof of Theorem~\ref{prop:10}}\label{pf:prop:10}

From  Condition \eqref{lem:2_12} and Proposition~\ref{prop:cone}, we know that for $s\in[t,T]$, 
\begin{equation}\label{lem2_0.1}
\begin{aligned}
    &D_p H \left(s,X_{t\xi}(s),\lr\left(X_{t\xi}(s)\right),P_{t\xi}(s) \right)=b \left(s,X_{t\xi}(s),\lr\left(X_{t\xi}(s)\right),v_{t\xi}(s) \right);
\end{aligned}
\end{equation}
and from \eqref{lem5_3}, we have
\begin{equation}\label{lem2_0.2}
\begin{aligned}
    &H\left(s,X_{tx\mu}(s),\lr\left(X_{t\xi}(s)\right),P_{tx\mu}(s) \right)=L\left(s,X_{tx\mu}(s),\lr\left(X_{t\xi}(s)\right),v_{tx\mu}(s),P_{tx\mu}(s) \right).\\
\end{aligned}
\end{equation}
In view of \eqref{lem2_0.1} and \eqref{lem2_0.2}, we now prove \eqref{prop10_01},  similarly to that for \cite[Theorem 5.3]{AB11}. From the estimate \eqref{lem2_1}, using Cauchy-Schwarz inequality, we deduce that 
\begin{align}
		\e\left[|X_{t\xi}(t')-\xi|^2\right] =\ & \e\Bigg[\left| \int_t^{t'}b(s,\theta_{t\xi}(s)) ds+ \int_t^{t'}\sigma(s,X_{t\xi}(s),\lr(X_{t\xi}(s))) dB(s) \right|^2\Bigg] \notag \\
		\le\ & C(L,T) \int_t^{t'}\e\left[1+|X_{t\xi}(s)|^2+\|X_{t\xi}(s)\|_2^2+|v_{t\xi}(s)|^2\right]ds \notag \\
		\le\ & C\left(1+W_2^2(\mu,\delta_0)\right)|t'-t|.  \label{prop10_11}
\end{align}
In a similar fashion, from \eqref{lem2_1} and \eqref{lem5_1}, we can deduce that 
\begin{equation}\label{prop10_12}
	\begin{split}
		&\e\left[|X_{tx\mu}(t')-x|^2\right] \le C \left(1+|x|^2+W_2^2(\mu,\delta_0)\right) |t'-t|.
	\end{split}		
\end{equation}
Note that
\begin{align}
	&P_{t\xi}(t')-P_{t\xi}(t) \notag \\
    =\ & P_{t\xi}(t')-\e\left[P_{t\xi}(t')|\xi\right] -\int_t^{t'}\e\Bigg[\bigg( D_x b(r,X_{t\xi}(r),\lr(X_{t\xi}(r)),v_{t\xi}(r))^\top P(r)+\sum_{j=1}^n \left(\sigma^j_1(r)\right)^\top Q^j(r) \notag \\
    &\quad\qquad\qquad\qquad\qquad\qquad\qquad\qquad + D_x f(r,X_{t\xi}(r),\lr(X_{t\xi}(r)),v_{t\xi}(r))\bigg)\Bigg|\xi\Bigg]ds. \label{prop10_13}
\end{align}  
From Cauchy-Schwarz inequality and \eqref{lem2_1}, we deduce that
\begin{align*}
	&\e\Bigg\{\Bigg|\int_t^{t'}\e\Bigg[ \bigg( D_x b(r,X_{t\xi}(r),\lr(X_{t\xi}(r)),v_{t\xi}(r))^\top P_{t\xi}(r)+\sum_{j=1}^n \left(\sigma^j_1(r)\right)^\top Q_{t\xi}^j(r)  \\
    &\qquad\qquad\qquad + D_x f(r,X_{t\xi}(r),\lr(X_{t\xi}(r)),v_{t\xi}(r))\bigg)\Bigg|\xi\Bigg]ds\Bigg|^2\Bigg\}\\
	\le\ & |t'-t|\cdot \e\Bigg[\int_t^{t'} \Bigg|D_x b(r,X_{t\xi}(r),\lr(X_{t\xi}(r)),v_{t\xi}(r))^\top P_{t\xi}(r)+\sum_{j=1}^n \left(\sigma^j_1(r)\right)^\top Q_{t\xi}^j(r)  \\
    &\qquad\qquad\qquad + D_x f(r,X_{t\xi}(r),\lr(X_{t\xi}(r)),v_{t\xi}(r)) \Bigg|^2 ds\Bigg]\\
	\le\ &  |t'-t|\cdot C(L,T) \e\bigg[\int_t^{t'} \left(|X_{t\xi}(r)|^2 + \|X_{t\xi}(r)\|_2^2+|v_{t\xi}(r)|^2+|P_{t\xi}(r)|^2+|Q_{t\xi}(r)|^2 \right) ds\bigg]\\
	\le\ & C\left(1+W_2^2(\mu,\delta_0)\right) |t'-t| .
\end{align*}
Substituting the last estimate into \eqref{prop10_13}, we have
\begin{equation}\label{prop10_14}
	\begin{split}
		&\e\left[|P_{t\xi}(t')-P_{t\xi}(t)|^2\right]\le C\left(1+W_2^2(\mu,\delta_0)\right)|t'-t|.
	\end{split}		
\end{equation}
Similarly, we can deduce that
\begin{equation}\label{prop10_15}
	\begin{split}
		&\e\left[|P_{tx\mu}(t')-P_{tx\mu}(t)|^2\right]\le C\left(1+|x|^2+W_2^2(\mu,\delta_0)\right) |t'-t|.
	\end{split}		
\end{equation}
By the usual dynamic programming principle, for any $\epsilon\in[0,T-t]$,
\begin{equation*}
	V(t,x,\mu)=\e\left[\int_t^{t+\epsilon}f(s,\theta_{tx\mu}(s))ds\right]+V\left(t+\epsilon,X_{tx\mu}(t+\epsilon),\lr(X_{t\xi}(t+\epsilon))\right),
\end{equation*}
so we have
\begin{align}
		\frac{1}{\epsilon}\left[V(t+\epsilon,x,\mu)-V(t,x,\mu)\right] \notag =\ & \frac{1}{\epsilon}\e[V(t+\epsilon,x,\mu)-V(t,x,\mu)] \notag \\
		=\ & \frac{1}{\epsilon} \e\left[V(t+\epsilon,x,\mu)-V(t+\epsilon,X_{tx\mu}(t+\epsilon),\lr(X_{t\xi}(t+\epsilon)))\right] \notag \\
		& -\frac{1}{\epsilon} \e\left[\int_t^{t+\epsilon}f(s,\theta_{tx\mu}(s))ds\right]. \label{prop10_1}
\end{align}	
From Theorems~\ref{prop:6} and \ref{prop:7}, Estimates \eqref{prop10_11} and \eqref{prop10_12} and It\^o's formula for measure-dependent functionals (see \cite[Theorem 2.2]{AB5} and also \cite[Theorem 7.1]{BR}), we deduce that
\begin{equation}\label{prop10_2}
	\begin{split}
		&\lim_{\epsilon\to0}\frac{1}{\epsilon}\e\left[V(t+\epsilon,x,\mu)-V(t+\epsilon,X_{tx\mu}(t+\epsilon),\lr(X_{t\xi}(t+\epsilon)))\right]\\
		=\ & - b (t,x,\mu,v_{tx\mu}(t))^\top D_x V(t,x,\mu)-\frac{1}{2}\text{Tr}\left(\sigma (t,x,\mu)^\top D_x^2 V(t,x,\mu)\sigma(t,x)\right)\\
		& -{\e}\left[b (t,{\xi},\mu,{v_{t\xi}}(t))^\top D_y\frac{dV}{d\nu}(t,x,\mu)({\xi})+\frac{1}{2}\text{Tr}\left(\sigma (t,\xi,\mu)^\top D_y^2\frac{dV}{d\nu}(t,x,\mu)({\xi})\sigma(t,{\xi})\right) \right].
	\end{split}
\end{equation}
From \eqref{prop10_11}, \eqref{prop10_12}, \eqref{prop10_14} and \eqref{prop10_15}, we also have
\begin{equation*}\label{prop10_3}
	\lim_{\epsilon\to0}\frac{1}{\epsilon}\e\left[\int_t^{t+\epsilon}f(s,X_{tx\mu}(s),\lr(X_{t\xi}(s)),v_{tx\mu}(s))ds\right]=f(t,x,\mu,v_{tx\mu}(t));
\end{equation*}
substituting \eqref{prop10_2} and the previous equation back to \eqref{prop10_1}, from \eqref{prop6_03}, \eqref{rk_1}, \eqref{lem2_0.1} and \eqref{lem2_0.2}, we have
\begin{align*}
		&\lim_{\epsilon\to0}\frac{1}{\epsilon}\left[V(t+\epsilon,x,\mu)-V(t,x,\mu)\right]\\
		=\ & -b(t,x,\mu,v_{tx\mu}(t))^\top D_x V(t,x,\mu) -\frac{1}{2}\text{Tr}\left[\sigma (t,x,\mu)^\top D_x^2 V(t,x,\mu)\sigma(t,x,\mu)\right]-f(t,x,\mu,v_{tx\mu}(t))\\
		& -{\e}\left[b(t,{\xi},\mu,{v_{t\xi}}(t))^\top D_y\frac{dV}{d\nu}(t,x,\mu)({\xi})+\frac{1}{2}\text{Tr}\left(\sigma (t,{\xi},\mu)^\top D_y^2\frac{dV}{d\nu}(t,x,\mu)({\xi})\sigma(t,\xi,\mu)\right) \right]\\
		=\ & -\frac{1}{2}\text{Tr}\left(\sigma (t,x,\mu)^\top D_x^2 V(t,x,\mu)\sigma(t,x,\mu)\right)-b \left(t,x,\mu,v_{tx\mu}(t)\right)^\top P_{tx\mu}(t) -f\left(t,x,\mu,v_{tx\mu}(t)\right)\\
		& -{\e}\bigg[D_p H \left(t,{\xi},\mu,P_{t\xi}(t)\right)^\top D_y\frac{dV}{d\nu}(t,x,\mu)({\xi}) +\frac{1}{2}\text{Tr}\left(\sigma (t,\xi,\mu)^\top D_y^2\frac{dV}{d\nu}(t,x,\mu)(\xi)\sigma(t,\xi,\mu)\right) \bigg]\\
		=\ & -\frac{1}{2}\text{Tr}\left(\sigma (t,x,\mu)^\top D_x^2 V(t,x,\mu)\sigma(t,x,\mu)\right)-H\left(t,x,\mu,P_{tx\mu}(t)\right)\\
		& -\e\bigg[ D_p H \left(t,\xi,\mu,P_{t\xi}(t)\right)^\top D_y\frac{dV}{d\nu}(t,x,\mu)(\xi)+\frac{1}{2}\text{Tr}\left( \sigma (t,\xi,\mu)^\top D_y^2\frac{dV}{d\nu}(t,x,\mu)(\xi)\sigma(t,\xi,\mu)\right)\bigg] \\
        =\ & -\frac{1}{2}\text{Tr}\left(\sigma (t,x,\mu)^\top D_x^2 V(t,x,\mu)\sigma(t,x,\mu)\right)-H\left(t,x,\mu,D_x V(t,x,\mu)\right)\\
		& -\e\bigg[ D_p H \left(t,\xi,\mu,D_x V(t,\xi,\mu)\right)^\top D_y\frac{dV}{d\nu}(t,x,\mu)(\xi)+\frac{1}{2}\text{Tr}\left( \sigma\sigma^\top (t,\xi,\mu) D_y^2\frac{dV}{d\nu}(t,x,\mu)(\xi)\right)\bigg] \\
        =\ & -\frac{1}{2}\text{Tr}\left(\sigma (t,x,\mu)^\top D_x^2 V(t,x,\mu)\sigma(t,x,\mu)\right)-H\left(t,x,\mu,D_x V(t,x,\mu)\right)\\[3mm]
		&\!\!\!-\int_\brn\!\!\!\bigg[ D_p H \left(t,y,\mu,D_x V(t,y,\mu)\right)^\top D_y\frac{dV}{d\nu}(t,x,\mu)(y)
		+\frac{1}{2}\text{Tr}\left( \sigma\sigma^\top (t,y,\mu) D_y^2\frac{dV}{d\nu}(t,x,\mu)({y})\right)\bigg] d\mu(y),
\end{align*}
from which we can conclude with \eqref{prop10_01}.

\end{document}